\documentclass[reqno,10pt]{amsart}

\title{Conformality for a robust class of non-conformal attractors}

\author[]{Beatrice Pozzetti}

\author[]{Andr\'es Sambarino}

\author[]{Anna Wienhard}
\thanks{A.S. was partially financed by ANR DynGeo ANR-16-CE40-0025. B.P. and  A.W acknowledge funding by the Deutsche Forschungsgemeinschaft within the Priority Program SPP 2026 ``Geometry at Infinity''. A.W. acknowledges funding  by the European Research Council under ERC-Consolidator grant 614733, and by the Klaus-Tschira-
Foundation.}
\date{}

\subjclass[]{}

\usepackage[colorlinks = true,backref = page,hyperindex,breaklinks]{hyperref} 
\usepackage{ stmaryrd }
\usepackage{amssymb}
\usepackage{mathtools}      
\usepackage{mathabx}        
\usepackage[bb = fourier,cal = euler,scr = rsfs]{mathalfa}	
\usepackage{enumitem}       
\usepackage[english]{babel}
\usepackage{tikz}			
\usepackage{tikz-cd}		
\usepackage[font = small]{caption}	
\usepackage{nicefrac}

\usetikzlibrary{calc}

\renewcommand*{\backref}[1]{}
\renewcommand*{\backrefalt}[4]{\quad \tiny
  \ifcase #1 (\textbf{NOT CITED.})%
  \or    (Cited on page~#2.)%
  \else   (Cited on pages~#2.)%
  \fi}

\makeatletter

\def\MRbibitem{\@ifnextchar[\my@lbibitem\my@bibitem}

\def\mybiblabel#1#2{\@biblabel{{\hyperref{http://www.ams.org/mathscinet-getitem?mr=#1}{}{}{#2}}}}

\def\myhyperanchor#1{\Hy@raisedlink{\hyper@anchorstart{cite.#1}\hyper@anchorend}}

\def\my@lbibitem[#1]#2#3#4\par{%
  \item[\mybiblabel{#2}{#1}\myhyperanchor{#3}\hfill]#4%
  \@ifundefined{ifbackrefparscan}{}{\BR@backref{#3}}%
  \if@filesw{\let\protect\noexpand\immediate
    \write\@auxout{\string\bibcite{#3}{#1}}}\fi\ignorespaces%
}

\def\my@bibitem#1#2#3\par{%
  \refstepcounter\@listctr
  \item[\mybiblabel{#1}{\the\value\@listctr}\myhyperanchor{#2}\hfill]#3%
  \@ifundefined{ifbackrefparscan}{}{\BR@backref{#2}}%
  \if@filesw\immediate\write\@auxout
    {\string\bibcite{#2}{\the\value\@listctr}}\fi\ignorespaces%
}

\makeatother

\newcommand{\xqedhere}[2]{%
  \rlap{\hbox to#1{\hfil\llap{\ensuremath{#2}}}}}

\newcommand{\Z}{\mathbb{Z}} 

\newcommand{\R}{\mathbb{R}} 
\newcommand{\C}{\mathbb{C}}
\newcommand{\N}{\mathbb{N}}
\renewcommand{\P}{\mathbb{P}}

\newcommand{\F}{\mathbb{F}}
\newcommand{\K}{\mathbb K}
\newcommand{\HH}{\mathbb H}

\newcommand{\sym}{{\sf{S}}}

\newcommand{\lb}{\llbracket}
\newcommand{\rb}{\rrbracket}
\newcommand{\eps}{\varepsilon}
\newcommand{\G}{\sf{\Gamma}}    
\newcommand{\Gr}{\cal G}    
\newcommand{\cone}{{\cal C}}

\newcommand{\E}{\Sigma}
\newcommand{\g}{\gamma}
\newcommand{\z}{\zeta}

\newcommand{\bord}{\partial}

\renewcommand{\L}{\Lambda}

\newcommand{\bordcone}{\cone_\infty}
\newcommand{\sroot}{{\sf{a}}}
\newcommand{\II}{\mathbf{I}}
\newcommand{\JJ}{\mathbf{J}}
\newcommand{\PP}{\mathbf{P}}

\renewcommand{\split}{\textrm{split}}

\newcommand{\scr}{\mathscr}
\renewcommand{\sf}[1]{{\mathsf{#1}}}

\newcommand{\cal}{\mathcal}
\renewcommand{\frak}{\mathfrak}

\newcommand{\stp}{\pi}
\renewcommand{\over}{\overline}

\newcommand{\Bending}{\ker(d_\rho\conj+\id)}
\newcommand{\EE}{{\sf{E}}}

\DeclareMathOperator{\holder}{Hol}
\DeclareMathOperator{\Anosov}{hom}

\DeclareMathOperator{\supp}{supp}

\DeclareMathOperator{\conj}{\tau}

\DeclareMathOperator{\class}{C}
\DeclareMathOperator{\diam}{diam}
\DeclareMathOperator{\SL}{{\mathsf{SL}}}
\DeclareMathOperator{\PSL}{{\mathsf{PSL}}}
\DeclareMathOperator{\GL}{{\mathsf{GL}}}
\DeclareMathOperator{\SO}{{\mathsf{SO}}}
\DeclareMathOperator{\PGL}{{\mathsf{PGL}}}
\DeclareMathOperator{\PO}{{\mathsf{PO}}}
\DeclareMathOperator{\PSO}{{\mathsf{PSO}}}

\DeclareMathOperator{\id}{id}
\DeclareMathOperator{\Isom}{Isom}
\DeclareMathOperator{\LC}{LC{}}

\DeclareMathOperator{\Hff}{{Hf{}f}}

\DeclareMathOperator{\Hess}{Hess}

\DeclareMathOperator{\Out}{Out}
\DeclareMathOperator{\hitchin}{\mathscr{H}}

\newcommand{\mm}{\mathbf{m}}

\newcommand{\st}{\,\mathord{\colon}\,} 

\renewcommand{\angle}{\measuredangle}
\newcommand{\Wedge}{\mathsf{\Lambda}}  

\DeclareMathOperator{\Sp}{{\mathsf{Sp}}}

\DeclareMathOperator{\PSp}{{\mathsf{PSp}}}
\DeclareMathOperator{\PU}{{\mathsf{PU}}}
\DeclareMathOperator{\POK}{\PO_\K}
\newcommand{\HdC}{\mathbb H^d_\C}
\newcommand{\HR}{\mathbb H_\R}
\newcommand{\HdH}{\mathbb H^d_\HH}
\newcommand{\HoH}{\mathbb H^1_\HH}
\newcommand{\HdK}{\mathbb H^d_\K}

\newcommand{\calO}{\mathcal O}

\DeclareMathOperator{\shift}{shift}
\newcommand{\bsm}{\left(\begin{smallmatrix}}
\newcommand{\esm}{\end{smallmatrix}\right)}
\newcommand{\bpm}{\begin{pmatrix}}
\newcommand{\epm}{\end{pmatrix}}

\renewcommand{\epsilon}{\varepsilon}

\hypersetup{bookmarksdepth  =  3} 

\setlist[enumerate,1]{label = {\upshape(\roman*)},ref = \roman*}
\setlist[enumerate,2]{label = {\upshape(\alph*)},ref = \alph*}

\newtheorem{thmA}{Theorem}

\newtheorem*{thm*}{Theorem}
\newtheorem*{cor*}{Corollary}
\newtheorem*{prop*}{Proposition}

\newtheorem{thm}{Theorem}[section]

\newtheorem{cor}[thm]{Corollary}
\newtheorem{lemma}[thm]{Lemma}
\newtheorem{prop}[thm]{Proposition}

\theoremstyle{definition}

\newtheorem*{defi*}{Definition}
\newtheorem{defi}[thm]{Definition}

\newtheorem{remark}[thm]{Remark}
\newtheorem{ex}[thm]{Example}

\theoremstyle{remark}
\newtheorem{obs}[thm]{Remark}

\newcommand{\nocontentsline}[3]{}
\newcommand{\tocless}[2]{\bgroup\let\addcontentsline=\nocontentsline#1{#2}\egroup}

\begin{document}

\begin{abstract} In this paper we investigate the Hausdorff dimension of limit sets of Anosov representations. In this context we revisit and extend the framework of hyperconvex representations and establish a convergence property for them, analogue to a differentiability property. As an application of this convergence, we prove that the Hausdorff dimension of the limit set of a hyperconvex representation is equal to a suitably chosen critical exponent. 
\end{abstract}

\maketitle

\tableofcontents

\section{Introduction}

In his seminal paper, Sullivan \cite{sullivan} describes the Hausdorff dimension of the limit set $\sf L_\G,$ of a discrete group $\G$ acting on the real hyperbolic $n$-space, in terms of the Dirichlet series $$s\mapsto\sum_{\g\in\G}e^{-sd(o,\g o)}.$$ More precisely, the \emph{critical exponent} of such a series is $$h_\G=\inf\Big\{s:\sum_{\g\in\G}e^{-sd(o,\g o)}<\infty\Big\}=\sup\Big\{s:\sum_{\g\in\G}e^{-sd(o,\g o)}=\infty\Big\}$$ and Sullivan shows:

\begin{thm*}[Sullivan]If $\G$ is a convex co-compact subgroup of $\PSO(1,n)$ then the Hausdorff dimension of  $\sf L_\G$ is $h_\G.$
\end{thm*}

This is related to understanding the Hausdorff dimension of a hyperbolic set in dynamical terms. Indeed, the non-wandering set of the geodesic flow of $\G\backslash\HH^n$ is, by definition, a maximal isolated compact hyperbolic set, $h_\G$ is its topological entropy and Sullivan's result can be interpreted in terms of the Ledrappier-Young formula \cite{ly1}.

Describing the Hausdorff dimension of a hyperbolic repeller as a dynamical quantity is today well understood in the \emph{conformal} setting, i.e. when the derivative of the dynamics, restricted to the unstable distribution, acts as a conformal map (see Chen-Pesin's survey \cite{nonconformal}  and references therein). 
Analogously, Sullivan's result has been generalized to convex-cocompact groups of a CAT(-1)-space $X$ (see for example Bourdon \cite{bourdon} and Yue \cite{yue}). The metric on the visual boundary $\bord X$ used to compute the Hausdorff dimension is the \emph{visual metric}, for which the action of $\Isom X$ is conformal (i.e. sends balls to balls).

However, other natural metrics on $\bord X$ appear in very common situations: if $X$ is a rank 1 symmetric space of non-compact type, then its visual boundary carries the structure of a differentiable manifold and thus one would also like to understand the Hausdorff dimension of  limit sets for a (any) Riemannian metric on $\bord X.$ Unless $X$ is the real hyperbolic $n$-dimensional space, the Riemannian structure behaves differently from the visual structure: the action of $\Isom X$ is no longer conformal.

The dynamical characterization of Hausdorff dimension in a non-conformal setting is still not completely understood. We refer the reader again to Chen-Pesin's survey \cite{nonconformal}. Let us also note that only very recently B\'ar\'any-Hochman-Rapaport \cite{PlanarSets} provided a complete answer for Iterated-Function-Systems on the plane.
On the discrete groups side, Dufloux \cite{HCSchottky} has studied a class of Schottky subgroups of isometries of the complex hyperbolic $n$-space, that he calls \emph{well positioned}, and proves the analogue of Sullivan's result for the Hausdorff dimension of the limit set with respect to any Riemannian metric.

\subsection{This paper} 
In this paper we are interested in describing the Hausdorff dimension of the limit set of discrete subgroups of
a semi-simple Lie group $G,$ for a Riemannian structure on the flag spaces (or boundaries) of $G.$ The groups we will consider, called \emph{Anosov representations}, are in many ways similar to convex cocompact subgroups of $\SO(1,n),$ but do not act conformally on the boundaries of $G.$
\medskip

Anosov representations where introduced by Labourie \cite{Labourie-Anosov} for fundamental groups of negatively curved closed manifolds and the definition was extended by Guichard-W. \cite{GW-Domains} to any hyperbolic group. Such representations provide the appropriate generalization of the class of convex co-compact subgroups in the context of Lie groups of higher rank \cite{GW-Domains, KLP-Morse, KLP1}.  
 
We will not use the original definition but follow a more recent approach, developed by Kapovich-Leeb-Porti \cite{KLP1}, G\'eritaud-Guichard-Kassel-W. \cite{GGKW} and in particular Bochi-Potrie-S. \cite{BPS}, that provides a simplified definition and gives better quantitative control of Anosov representations. 

Let $\K=\R$ or $\C,$ consider an inner (or Hermitian if $\K=\C$) product in $\K^d$ and, for $g\in\GL_d(\K),$ denote by $g\mapsto g^*$ the corresponding adjoint operator. The singular values of $g$, i.e. the square root of the modulus of the eigenvalues of $gg^*$, are denoted by $$\sigma_1(g)\geq\cdots\geq\sigma_d(g).$$

Let $\G$ be a finitely generated discrete group, consider a finite symmetric generating set $S$ and denote by $|\,|$ the associated word metric on $\G.$ Given $p\in\lb1,d-1\rb$ denote by $\Gr_p(\K^d)$ the Grassmannian of $p$-dimensional subspaces of $\K^d.$
For a homomorphism $\rho:\G\to\PGL_d(\K),$ the following are equivalent:

\begin{itemize}
\item[i)] There exist positive constants $c,\mu$ such that for all $\g\in\G$ one has $$\frac{\sigma_{p+1}}{\sigma_p}\big(\rho(\g)\big)\leq ce^{-\mu|\g|},$$ 
\item[ii)] The group $\G$ is word-hyperbolic and there exist $\rho$-equivariant maps $(\xi^p,\xi^{d-p}):\bord\G\to\Gr_p(\K^d)\times\Gr_{d-p}(\K^d)$ such that for every $x\neq y\in\bord\G$ one has $$\xi^p(x)\oplus\xi^{d-p}(y)=\K^d,$$ and a suitable associated flow is contracting.
\end{itemize}

If either condition is satisfied we will say that $\rho$ is an $\{\sroot_p\}$-\emph{Anosov representation}\footnote{ The implication ii)$\Rightarrow$i) comes from Labourie \cite{Labourie-Anosov} and Guichard-W. \cite{GW-Domains}. The implication i)$\Rightarrow$ii) is more recent and due to Kapovich-Leeb-Porti \cite{KLP1}, see also Gu\'eritaud-Guichard-Kassel-W. \cite{GGKW} and Bochi-Potrie-S. \cite{BPS} for  different approaches. In the language of Bochi-Potrie-S. \cite[Section 3.1]{BPS} a representation verifying condition i) is called \emph{$p$-dominated}.}. For such a representation, the critical exponent $h^{\sroot_p}_\rho$  of the Dirichlet series \begin{equation}\label{DirichletRoot}
\Phi_\rho^{\sroot_p}(s)=\sum_{\g\in\G} \left(\frac{\sigma_{p+1}}{\sigma_p}\big(\rho(\g)\big)\right)^s\end{equation} 
is well defined. By definition, the series is convergent for every $s>h^{\sroot_p}_\rho$ and divergent for every $0<s<h^{\sroot_p}_\rho.$ 

If $\rho$ is furthermore $\{\sroot_{p+1}\}$-Anosov then $h^{\sroot_p}_\rho$ is analytic with respect to $\rho,$ and agrees with the entropy of a suitably defined flow (see for example Bridgeman-Canary-Labourie-S. \cite{pressure} and Potrie-S. \cite[Corollary 4.9]{exponentecritico}). But in general little is known about $h^{\sroot_p}_\rho$ without this extra assumption.

We will mainly focus on $\{\sroot_{1}\}$-Anosov representations. The chosen inner product on $\K^d$ induces a metric on $\P(\K^d),$ we will denote by $\Hff(A)$ the Hausdorff dimension of a subset $A\subset\P(\K^d)$ for this metric. As a first result we obtain the following, independently obtained by Glorieux-Monclair-Tholozan \cite{GMT}.
\begin{prop*}[Proposition~\ref{Hffand1st}]
Let $\rho: \G\to\PGL_d(\K)$ be $\{\sroot_{1}\}$-Anosov. Then 
$$\Hff\big(\xi^1(\bord\G)\big)\leq h^{\sroot_1}_\rho.$$
\end{prop*}

In order to discuss situations in which equality holds, we introduce the notion of \emph{locally conformal points} of $\rho$ (Definition \ref{LocallyConformal}), these are points of $\bord\G$ designed to detect some asymptotic conformality of the non-conformal action of $\rho(\G)$ when restricted to the limit set $\xi^{1}(\bord\G).$ Using Patterson's construction we then obtain a (not necessarily quasi-invariant) measure $\mu^{\sroot_1}_\rho$ on $\bord\G$. Following Sullivan, we then prove the following.

\begin{thm*}[Theorem \ref{thm:eqifLC}]
 If the set of locally conformal points of $\rho$ has positive $\mu_{\rho}^{\sroot_{1}}$-measure, then 
 $$\Hff\big(\xi^1(\bord\G)\big) =  h^{\sroot_1}_\rho.$$
\end{thm*}

\medskip
Interestingly, for a rich class of Anosov representations, a 3-point transversality condition, inspired by Labourie \cite{Labourie-Anosov}, forces asymptotic conformality:

\begin{defi*} Consider $p,q,r\in\lb1,d-1\rb$ such that $p+q\leq r.$ A $\{\sroot_p,\sroot_q,\sroot_r\}$-Anosov representation $\rho:\G\to\PGL_d(\K)$ is called $(p,q,r)$-\emph{hyperconvex} if for every triple of pairwise distinct points $x,y,z\in\bord\G$ one has $$\big(\xi^p(x)\oplus \xi^q(y)\big)\cap \xi^{d-r}(z)=\{0\}.$$ (Note that $p$ and $q$ are not required to be distinct.)
\end{defi*}

The main result of this paper is the following.

\begin{thmA}[{Corollary~\ref{2FrenetHffDim}} and Corollary~\ref{inequality}] \label{thA}Let $\rho$ be $(1,1,2)$-hyperconvex. Then $$h^{\sroot_1}_\rho=\Hff\big(\xi^1(\bord\G)\big)\leq\Hff\big(\P(\K^2)\big).$$
\end{thmA}
The aforementioned analyticity result for $h^{\sroot_p}_\rho$, together with Theorem \ref{thA}, has the following consequence:

\begin{cor} Let $\{\rho_u:\G\to\PGL_d(\K)\}_{u\in D}$ be an analytic family of $(1,1,2)$-hyperconvex representations, then $u\mapsto \Hff\big(\xi^1_u(\bord\G)\big)$ is analytic. 
\end{cor}

In fact Theorem \ref{thA} holds in greater generality. We can replace $2's$ by any $p\in\lb2,d-1\rb$ if we additionally require that for every $\g\in\G$ one has $$\sigma_2\big(\rho(\g)\big)=\sigma_p\big(\rho(\g)\big),$$ see 
Corollary~\ref{11pLC} and Corollary~\ref{inequality}. This extra condition on the singular values should be interpreted as a restriction on the Zariski closure of the representation (see subsection \ref{sec:POK} for situations such as $\PSp(1,n)$ and $\PU(1,n),$  and subsection \ref{s.SOpq} for the group $\PSO(p,q)$).

A key ingredient for the proof of Theorem~\ref{thA}  is the following convergence property for hyperconvex representations, from the inequality readily follows. 

\begin{thmA}[{Theorem \ref{FRep->Fpoint}}]\label{thB} Let $\rho$ be $(p,q,r)$-hyperconvex, then for every $(w,y)\in\bord^{(2)}\G$ one has $$\lim_{(w,y)\to (x,x)}d\big(\xi^p(w)\oplus \xi^q(y),\xi^r(x)\big)=0.$$
\end{thmA}

We further investigate how vast the class of hyperconvex representations is. On the one hand one has the following remarks that provide many examples by the \emph{represent and deform} method (see subsection~\ref{repdef}): 
 \begin{itemize}
 \item[-] if $\rho:\G\to\PGL_d(\R)$ is hyperconvex then, by complexifiyng, one obtains a hyperconvex representation over $\C:$ this is direct from the definition; 
\item[-] the space of $(p,q,r)$-hyperconvex representations is open in $\hom(\G,\PGL_d(\K))$ (Proposition \ref{FrenetOpen}).
\end{itemize}

On the other hand there are  some `verifiable' restrictions imposed by the hyperconvexity condition. For example, a $(1,1,p)$-hyperconvex representation of $\G$ induces a continuous injective map $$\bord\G-\{\textrm{point}\}\to\P(\K^p),$$ (see Corollary \ref{r.obstruction}), and there might be topological obstructions for the existence of such a map.  More interesting restrictions arise when $\K=\R$ and $\bord\G$ is a manifold:

\begin{cor*}[Proposition \ref{diff}] Let $\G$ be such that $\bord\G$ is homeomorphic to a $(p-1)$-dimensional sphere. If $\rho:\G\to\PGL_d(\R)$ is $(1,1,p)$-hyperconvex then $\xi^1(\bord\G)$ is a $\class^1$ sphere.
\end{cor*}

Using openness of hyperconvexity we find new explicit examples of Zariski dense groups with $\class^1$ limit set.
\begin{cor*}[Corollary~\ref{Zdense}]
There exist Zariski dense subgroups $\G<\PGL_{d(d+1)}(\R)$ whose limit set is a $\class^1$ sphere of dimension $d-1$.
\end{cor*}
Sharper results of similar nature were obtained by Zhang-Zimmer \cite{ZZ}.

\medskip

We now turn to the special situation when $\bord\G$ is a circle. Then Theorem~\ref{thA} gives the following computation of $h^{\sroot_1}_\rho$:

\begin{cor*} Assume $\bord\G$ is homeomorphic to a circle, if $\rho:\G\to\PGL_d(\R)$ is $(1,1,2)$-hyperconvex then $h^{\sroot_1}_\rho=1.$
\end{cor*}

This implies Potrie-S. \cite[Theorem B]{exponentecritico} and further generalizes it to the Hitchin component of $\PSO(p,p).$ The proof of  Potrie-S. \cite[Theorem A]{exponentecritico} applies then verbatim also to the Hitchin component of $\PSO(p,p)$ and we thus obtain a rigid inequality for the critical exponent in the symmetric space of $\PSO(p,p)$.
We refer the reader to Sections \ref{HitchinPSL} and \ref{HitchinSO} for more details on Hitchin representations. 

While the property of having  constant $h^{\sroot_1}_\rho$ was expected to be a rare phenomenon, peculiar to Hitchin components, or possibly higher rank Teichm\"uller theories, we provide, in Section \ref{SL2}, many more examples of representations of fundamental groups of surfaces for which Theorem \ref{thA} applies. Interestingly enough, when $\bord\G$ is a circle (and $\K=\R$), $(1,1,2)$-hyperconexity is not only a local condition, but it can be pushed far away. We say that an $\{\sroot_1\}$-Anosov representation is \emph{weakly irreducible} if $\xi^1(\bord\G)$ is not contained in a proper subspace of $\P(\R^d)$.

\begin{prop*}[{Proposition \ref{FrenetClosed}}] Assume that $\bord\G$ is homeomorphic to a circle. Then the space of real weakly irreducible $(1,1,2)$-hyperconvex representations of  $\G$ is closed among real weakly irreducible $\{\sroot_1,\sroot_2\}$-Anosov representations. 
\end{prop*}

\medskip
\medskip 

Throughout the paper we allow $\K$ to be a local field (not necessarilly Archimedean, as we required in this introduction). Originally Anosov representations were only defined over Archimedean fields as it is possible to show that if $\G$ admits a Anosov representation $\rho:\G\to\PSL_d(\K)$ for non-Archimedean $\K$, then $\G$ is virtually free. The main result of our paper, however, associates to such an action an interesting geometric quantity, the Hausdorff dimension of the limit set, which we are able to relate to a dynamical data, the orbit growth rate. We find this very interesting, and this justifies the extra work needed to develop the theory in this more general setting.

The main results go through in this generality, except the analyticity of Hausdorff dimension: the key step is to show that for a $\{\sroot_1\}$-Anosov representation its \emph{entropy}, defined by $$\limsup_{t\to\infty}\frac1t\log\#\{\g\in\G:\log\sigma_1\big(\rho(\g)\big)\leq t\}$$ is analytic with $\rho.$ We don't know if this is true, but 
one can use the thermodynamical formalism to prove that the Hausdorff dimension depends continuously on the representation (and is actually as regular as the map $\rho\mapsto\xi_\rho$ is).
\subsection*{Outline of the paper}
The preliminaries of the paper, collected in Section \ref{s.2} come from three different areas: quantitative linear algebra, dynamics and geometric group theory. In \S~\ref{s.2.1} we recall relations between the singular values of an element in $\PGL_d(\K)$ and metric properties of its action on Grassmannian manifolds, in the general context of a local field $\K.$ In \S~\ref{s.2.2} we discuss the dynamical backgrounds and indicate how to extend Bochi-Gourmelon's theorem as well as the theory of dominated splittings to general local fields. \S~\ref{s.2.3} collects the facts about hyperbolic groups and cone types that we will need in the paper.

Section \ref{s.3} concerns Anosov representations: we extend to the non-Archimedean setup the definition  and the results we will need, particularly concerning the definition and properties of the equivariant boundary maps. Our discussion follows the lines of Bochi-Potrie-S. \cite{BPS}.

In Section \ref{s.4} we prove that for any Anosov representation the Hausdorff dimension of the limit curve provides a lower bound for the critical exponent for the first root. In Section \ref{s.5} we give a condition guaranteeing that such bound is optimal, namely the abundance of \emph{locally conformal points} with respect to a suitable measure.

Section \ref{section:FrenetReps} concerns the notion of $(p,q,r)$-hyperconvexity, an open condition (Proposition \ref{FrenetOpen}) that guarantees abundance of locally conformal points: this is the content of Proposition \ref{p.main}, the main technical result of the paper. Using the theory of $\SL_2$ representations we provide in \S~\ref{SL2} many examples of hyperconvex representations of fundamental groups of surfaces and hyperbolic three manifolds. 

In Section \ref{s.diff} we discuss another interesting consequence of hyperconvexity: such property guarantees a weak differentiability property for the limit set (Theorem \ref{FRep->Fpoint}) which allows us, on the one hand, to obtain good bounds on the Hausdorff dimension (Proposition \ref{inequality}), and on the other to provide examples of Zariski dense subgroups whose limit set in the projective space is a $\class^1$ manifold: we obtain these through the \emph{represent and deform} method explained, in a concrete example, in Proposition \ref{symm}. 

In Section \ref{s.exLC} we discuss in detail two families of representations for which all our results apply: on the one hand we detail the geometric meaning of our notions in the case of convex cocompact subgroups of rank one groups, rediscovering and generalizing results of Dufloux (\S~\ref{sec:POK}), on the other we give a concrete criterion that guarantees hyperconvexity for subgroups of $\SO(p,q)$ and provide examples of groups that satisfy it (\S~\ref{s.SOpq}).

The last section of the paper (Section \ref{s.surfaces}) concerns representations of fundamental groups of hyperbolic surfaces (or more generally compact hyperbolic orbifolds). For these we show that hyperconvexity is also a closed condition (Proposition \ref{FrenetClosed}), and discuss a new proof and generalization of a result of Potrie-S. \cite{exponentecritico}.


\section{Preliminaries}\label{s.2}
In the paper we will need preliminaries from three different sources: quantitative linear algebra, dynamics and particularly the work of Bochi-Gourmelon \cite{BG} and Bochi-Potrie-S. \cite{BPS}  on dominated sequences, and algebraic and metric properties of hyperbolic groups. We recall the results we need here.
\subsection{Quantitative linear algebra}\label{s.2.1}
As anticipated at the end of the introduction, in the paper we will be dealing with representations of finitely generated groups on finite dimensional vector spaces over local fields. We recall here some quantitative results we will need. More details on algebraic groups over local fields can be found in Quint \cite{Quint-Divergence}.

\subsubsection{Angles and distances on Grassmannians} We denote by $\K$  a local field, and by $|\,|:\K\to\R^+$ its absolute value.  Recall that if $\K$ is $\R$ or $\C$ then $|\,|$ is the usual modulus, if, instead, $\K$ is non-Archimedean, we require that $|\omega|=\frac 1q$ where $\omega$ denotes the \emph{uniformizing element}, namely a generator of the maximal ideal of the valuation ring $\calO$, and $q$ is the cardinality of the residue field $\calO/\omega\calO$ (this is finite because $\K$ is, by assumption, local).

Given a finite dimensional vector space $V$ over $\K$, we denote by $\|\,\|:V\to\R^+$ a \emph{good norm}: for an Archimedean field $\K$ this means that $\|\,\|$ is induced from an Hermitian product, if $\K$ is non-Archimedean this means that there exists a basis $\{e_1,\ldots, e_n\}$ such that $\|\sum a_i e_i\|=\max\{|a_i|\}$. In this second case we say that a decomposition $V=V_1\oplus V_2$ is \emph{orthogonal} if $\|v_1+v_2\|=\max\{\|v_1\|,\|v_2\|\}$ for all $v_1\in V_1$ and $v_2\in V_2$. In general, since $\K$ is locally compact, any two norms on $V$ are equivalent.

The choice of a good norm $\|\,\|$ on $V$ induces a good norm on every exterior power of $V$ (this is discussed in Quint \cite{Quint-Divergence}). This allows to generalize the notion of angle to the non-Archimedean setting: for $v,w\in V$, we define $\angle (v,w)$ to be the unique number in $[0,\pi]$ such that
$$\sin\angle (v,w)=\frac{\|v\wedge w\|}{\|v\|\|w\|}.$$
Observe that the angle crucially depends on the choice of the norm. Following Bochi-Potrie-S. \cite{BPS} we define the angle of two subspaces $P,Q<\K^d$ as 
$$\angle(P,Q)=\min_{v\in P^\times}\min_{w\in Q^\times}\angle(v,w), $$
where $P^\times=P\setminus\{0\}$, $Q^\times=Q\setminus\{0\}$.

The sine of the angle gives a distance, that we sometimes denote by $d$, on the projective space $\P(V)$, and more generally on every Grassmannian $\Gr_k(V)$: we set for $P,Q\in\Gr_k(V)$
$$d(P,Q):=\max_{v\in P^\times}\min_{w\in Q^\times}\sin\angle(v,w)=\min_{v\in P^\times}\max_{w\in Q^\times}\sin\angle(v,w),$$
this corresponds to the Hausdorff distance of $\P(P),\P(Q)$ regarded as subsets of $\P(V)$ with the aforementioned distance. 
Observe that 
$$d(P,Q)\geq \sin\angle(P,Q)$$
and the latter inequality is, apart from very special cases, strict. 

More generally we extend the distance to subspaces of possibly different dimension:  for $P\in\Gr_k(V)$, $Q\in\Gr_l(V)$, $k\leq l$ we set 
$$d(P,Q):=\max_{v\in P^\times}\min_{w\in Q^\times}\sin\angle(v,w)=\min_{W\in\Gr_k(Q)} d(P,W).$$
Such distance vanishes if and only if $P\subset Q$.

\subsubsection{Singular values}\label{ss.dom_sing}
Assume now that $\K$ is commutative.
Given a $\K$-norm on $V$ we say that $g\in\GL(V,\K)$ is a \emph{semi-homothecy} if there exists a $g$-invariant $\K$-orthogonal decomposition $V=V_1\oplus\cdots\oplus V_k$ and $\sigma_1,\cdots,\sigma_k\in\R_+$ such that for every $i\in\lb1,k\rb$ and every $v_i\in V_i$ one has $$\|gv_i\|=\sigma_i\|v_i\|.$$ The numbers $\sigma_i$ are called the ratios of the semi-homothecy $g.$

Consider a maximal abelian subgroup of diagonalizable matrices $A\subset\GL(V,\K),$ let $K\subset\GL(V,\K)$ be a compact subgroup such that if $N_{\GL}(A)$ is the normalizer of $A$ in $\GL(V,\K)$ then $N_{\GL}(A)=(N_{\GL}(A)\cap K)A.$ Following Quint \cite[Th\'eor\`eme 6.1]{Quint-localFields} there exists a $\K$-norm $\|\,\|$ on $V$ such that 
\begin{itemize}
\item[-] $\|\,\|$ is preserved by $K,$ 
\item[-] $A$ acts on $(V,\|\,\|)$ by semi-homothecies with respect to a common $\K$-orthogonal decomposition of $V$ in one dimensional subspaces.
\end{itemize}

Whenever such a norm is fixed, for every $g\in\GL(V)$ we denote the  \emph{norm} and its \emph{co-norm} by 
$$\|g\|:=\max_{v\in V^\times}\frac{\|gv\|}{\|v\|} \qquad\mm(g)=\inf_{v\in V^\times} \frac{\|gv\|}{\|v\|}.$$

Let $d=\dim V$. Keeping notation from Quint \cite{Quint-localFields}, we denote by $\EE:=\R^d$ a real vector space with a restricted root system of $\GL(V)$,  and by 
$$\EE^+=\{x=(x_1,\ldots, x_d)\in\R^d|\,x_1\geq\ldots\geq x_d\}$$
a Weyl chamber of $\EE.$
We will denote by $\sroot_i\in\sf E^*$ the simple roots of $\EE$, so that $$\sroot_i(x)=x_i-x_{i+1}\in\R.$$ The choice of an ordering $(e_1,\ldots, e_d)$ of the joint eigenlines of $A$ (the eigenlines are uniquely determined Quint \cite[Lemma II.1.3]{Quint-Divergence}) induces a map $\nu:A\to\EE$ given by  
$$\nu(a):=(\log\sigma_1(a),\ldots, \log\sigma_d(a)),$$ 
where $\sigma_1(a),\cdots,\sigma_d(a)$ are the semi-homotecy ratios in the basis $\{e_1,\ldots, e_d\}$. We set $A^+:=\nu^{-1}(\EE^+)$, so that $A^+$ consists of those elements $a\in A$ whose corresponding semi-homothecy ratios satisfy $\sigma_1(a)\geq\cdots\geq\sigma_d(a)$. 

With respect to the basis $\{e_1,\ldots, e_d\}$, when $\K$ is non-Archimedean, it holds that $K=\GL(d,\calO)$, and the map $\nu$ extends to the Cartan projection, still denoted $\nu$ from the whole $\GL(V,\K)$: indeed $\GL(V,\K)=K A^+K$, and, given $a_1, a_2\in A^+$, the element $a_1$ belongs to $Ka_2K$ if and only if $\nu(a_1)=\nu(a_2)$.  In particular we can set $\nu(g)=\nu(a_g)$ for any element $a_g\in A$ such that there exist $k_g,l_g\in K$ with $g=k_ga_gl_g$ (Bruhat-Tits \cite[Section 3.3]{BrTi}).

For every $g\in \GL(V,\K)$, we choose a Cartan decomposition $g=k_ga_gl_g$ as above  and define, for $p \in \lb1,d-1\rb,$  
$$u_p(g)=k_g\cdot  e_p\in  V.$$ 
If $\K$ is Archimedean, the set $\{u_p(g):p\in\lb1,d-1\rb\}$ is an \emph{arbitrary} orthogonal choice of axes (ordered in decreasing length) of the ellipsoid $\{Av \st \|v\| = 1\}.$ Note that for every $v$ that lies in the span of $g^{-1}u_p(g)$ one has $\|gv\|=\sigma_p(g)\|v\|.$ With a slight abuse of notation we will often also denote by $u_p(g)$ the corresponding point in $\P V$.

We furthermore denote by $U_p(g)$ the \emph{Cartan attractor} of $g$: $$U_p(g)=u_1(g)\oplus\cdots\oplus u_p(g)=k_g\cdot(e_1\oplus\cdots\oplus e_p).$$

\begin{defi}\label{d.gap}
An element $g\in \GL(V,\K)$ is said to have  \emph{a gap of index} $p$ if $\sigma_p(g) > \sigma_{p+1}(g)$. In that case, if $\K$ is Archimedean,  the $p$-dimensional space $U_p(g)$ is independent of the Cartan decomposition of $g.$
\end{defi}
Note that if $g$ has a gap of index $p$, then the decomposition $$U_{d-p}(g^{-1}) \oplus  g^{-1}(U_p(g))$$ is orthogonal: this is clear when $\K$ is Archimedean (see Remark \ref{r.2.4} for the general case) 

 \begin{remark}\label{rem.2.2}
 If $\K$ is not Archimedean, the components $k_g,l_g$ in the Cartan decomposition are not uniquely determined even if $g$ has gaps of every index; in particular the spaces $U_p(g)$ always depend on the choice of the Cartan decomposition. For example take $d=2$; if $|a|>|b|$ we have 
 $$\bpm a&0\\0&b\epm= \bpm1&0\\b/a&1\epm\bpm a&0\\0&b\epm\bpm1&0\\-1&1\epm$$ and both $\bsm1&0\\b/a&1\esm$ and $\bsm1&0\\-1&1\esm$ belong to $K=\GL(2,\calO)$. In this example it is easy to verify that the set of possible Cartan attractors $U_1(g)$ coincides with the ball of center $e_1$ and radius $|b/a|$. Note that, since $\K$ is non-Archimedean, any point in this ball is a center.
\end{remark}

\subsubsection{Quantitative results}
Many of the auxiliary technical results in \cite{BPS} rely on the min-max characterization of singular values of linear maps from $\R^d$ to $\R^d$. This characterization in fact generalizes to any local field if one replaces the singular values with the semi-homothecy ratios:
$$\sigma_p(A)=\max_{P\in\Gr_p(V)}\mm(A|_P)\qquad \sigma_{p+1}(A)=\min_{Q\in\Gr_{d-p}(V)}\|A|_Q\|.$$ Therefore the quantitative linear algebraic facts collected in \cite[Appendix 3]{BPS} carry through. We now state the ones that we will use in the following.

\begin{lemma}[{\cite[Lemma A.4]{BPS}}]\label{l.qg} 
Let $g ,h\in \GL(V,\K)$ have a gap of index $p$.
Then for any possible choice of Cartan attractor $U_p(g)$ (resp. $U_p(gh)$):
\begin{eqnarray}
d(U_p(gh), U_p(g)) \le \|h\|\|h^{-1}\|\frac{\sigma_{p+1}}{\sigma_p}(g)\label{l.qg1} \\
d(U_p(gh), gU_p(h)) \le \|g\|\|g^{-1}\|\frac{\sigma_{p+1}}{\sigma_p}(h) \label{l.qg2} .
\end{eqnarray}
\end{lemma}
\begin{remark}\label{r.2.4}
If $\K$ is non-Archimedean, the Cartan attractors $U_p(g)$ are not uniquely defined  (cfr. Remark \ref{rem.2.2}). However it follows from Lemma \ref{l.qg} that, given two  different Cartan decompositions for $g$, $g=k_g a_g l_g=k_g' a_g'l_g'$, and denoting  $V_p=\langle e_1,\ldots, e_p \rangle$, we have
$$ d(k_gV_p,k_g' V_p)\leq\frac{\sigma_{p+1}}{\sigma_p}(g),$$
namely all possible different choices for $U_p(g)$ are contained in a ball of radius  $\frac{\sigma_{p+1}}{\sigma_p}(g)$. As the distance $d$ is, in this case, non-Archimedean, we deduce, also in this case, that any choice of $U_p(g)$ is orthogonal to $gU_{d-p}(g^{-1})$ for any other choice of $U_{d-p}(g^{-1})$.
\end{remark}
\begin{lemma}[{\cite[Lemma A.6]{BPS}}]\label{l.dominationattractor} 
Let $g \in \GL(V,\K)$ have a gap of index $p$.
Then, for all $P \in \Gr_{p}(V)$ transverse to $U_{d-p}(g^{-1})$ we have:
$$
d(g(P), U_p(g)) \le \frac{\sigma_{p+1}}{\sigma_p}(g) \, \frac{1}{\sin\angle(P,U_{d-p}(g^{-1}))} \, .
$$
\end{lemma}

\begin{lemma}[{\cite[Lemma A.7]{BPS}}]\label{l.no_cancellation}
Let $g$, $h \in \GL(V,\K)$.
Suppose that $g$ and $gh$ have gaps of index $p$.
Let $\alpha \coloneqq \angle \big( U_p(h), U_{d-p}(g^{-1}) \big)$.
Then:
\begin{align*}
\sigma_p(gh)     &\ge (\sin \alpha) \, \sigma_p(g) \, \sigma_p(h) \, ,  		\\
\sigma_{p+1}(gh) &\le (\sin \alpha)^{-1} \, \sigma_{p+1}(g) \, \sigma_{p+1}(h) \, . 
\end{align*}
\end{lemma}

Given a subspace $P\in \Gr_p(V)$ we denote by $P^\perp$ a chosen orthogonal complement of $P$; this always exists, but is not unique if $\K$ is non-Archimedean.
Suppose that $P$, $W \in \Gr_p(V)$ satisfy $d(P,W)<1$.
Then $W \cap P^\perp  =  \{0\}$, and so there exists a unique linear map 
\begin{equation}\label{e.graph}
L_{W,P} \colon P \to P^\perp
\quad \text{such that} \quad
W  =  \big\{ v+ L_{W,P}(v) \st v \in P \big\}\, .
\end{equation}
The association $L_{W,P}\mapsto W$ provides an affine chart for $\Gr_p(V)$. The next lemma states that this chart is 1-Lipschitz, and it is 4-biLipschitz on a sufficiently small neighbourhood of $P$:

\begin{lemma}[{\cite[Lemma A.11]{BPS}}]\label{l.d_and_norm}
Let $P$, $P_1$, $P_2 \in \Gr_p(V),$ with $d(P_i,P)<1$, then 
$$ d(P_1,P_2) \leq \|L_{P_1,P} - L_{P_2,P}\|$$
for all choices of $P^\perp$.
 If moreover $d(P_i,P)< 1/\sqrt2$ then
$\|L_{P_1,P} - L_{P_2,P}\| \leq 4 d(P_1,P_2) \,.
$
\end{lemma}
\begin{proof}
The proof of \cite[Lemma A.11]{BPS} smartly combines the triangular inequality for the distance $d$ and the characterization
$$d(P_1,P_2)=\max_{w\in P_1^*}\min_{v\in P_2^*}\frac{\|v-w\|}{\|w\|}.$$
Since both hold when $V$ is a vector space over a local field $\K$, the proof generalizes without modifications. In case $\K$ is non-Archimedean, one could also deduce the better estimate $\|L_{P_1,P} - L_{P_2,P}\| \leq 2 d(P_1,P_2)$.
\end{proof}

The next lemma is a variation of  \cite[Lemma A.10]{BPS}. In  \cite{BPS} there is an assumption on $d(P_i,P)$ depending on $g$ that we replace here with the contraction assumption $d(gP_i,gP)<1/\sqrt 2$. 
Despite the proof is very similar to  \cite[Lemma A.10]{BPS}, we include it for completeness:

\begin{lemma}\label{l.expand}
Let $V$ be a $d$-dimensional $\K$-vector space, and $g\in \GL(V)$.
Choose $P \in \Gr_p(V)$ and
$Q \in \Gr_{d-p}(V)$ such that the pairs $(P,Q)$ and  $(gP,gQ)$ are orthogonal.
Then for every $P_i \in \Gr_p(V),$ ($i=1,2$) with $P_i\cap Q=\{0\}$ and $d(gP_i,gP)<1/\sqrt2$, it holds
\begin{equation}\label{e.min_expansiono}
d(gP_1,gP_2) \ge  \, \frac{\mm(g|_Q)}{4\| g|_P \|} \, d(P_1,P_2) \, .
\end{equation}
\end{lemma}

\begin{proof}[Proof of Lemma~\ref{l.expand}] Using the same notation as in \eqref{e.graph}, for each $i = 1,2$, we consider the linear map $L_i = L_{P_i,P}: P \to P^\perp$ and $M_i =L_{g P_i,g P}: gP \to gP^\perp$;
these are well defined since $P_i \cap Q  =  \{0\}$.
Clearly the two maps are related by
$L_i  =  (g^{-1}|_{gQ}) \circ  M_i \circ (g|_P)$.
As a consequence,
\[
\|L_1 - L_2\|  =  \big\| (g^{-1}|_{gQ}) \circ (M_1-M_2) \circ (g|_P) \big\|
\le \frac{\| g|_P \|}{\mm(g|_Q)} \, \|M_1 - M_2\| \, .
\]
Lemma~\ref{l.d_and_norm} gives:
$$
\|L_1 - L_2\| \ge d(P_1,P_2) \, .
$$ Since by assumption $d(g P_i, g P) < 1/\sqrt 2$, Lemma~\ref{l.d_and_norm} implies:
$$
\|M_1 - M_2\| \leq 4 d(gP_1, gP_2) \, .
$$ Putting these three estimates together, we get 
$$
d(gP_1, gP_2)\geq \frac{1}{4} \|M_1 - M_2\| \geq \frac{1}{4}\frac{\mm(g|_Q)}{\| g|_P \|}\|L_1 - L_2\|\geq  \frac{1}{4}\frac{\mm(g|_Q)}{\| g|_P \|}d(P_1,P_2). $$
\end{proof}
The following corollary of Lemma \ref{l.expand} will be useful in Section \ref{coarselyball}:
\begin{cor}\label{c.expand}
Let $V$ be a $\K$-vector space, $W<V$ a subspace of dimension 2, and $g\in\GL(V)$. Denote by $\sigma_i(g|_{W})$ the semi-homothecy ratios of $g:W\to gW$ where the norm on $W$ (resp. $gW$) is induced by the norm on $V$. For every $P_i\in\P W$ with $P_i\cap U_{d-1}(g^{-1})=\{0\}$, and $d(gP_i,u_1(g|_W))<1/\sqrt 2$ it holds
$$
d(gP_1,gP_2) \ge  \, \frac{\sigma_2(g|_W)}{4 \|g|_W \|} \, d(P_1,P_2) \, .
$$
\end{cor}
\begin{proof}
This follows directly from Lemma \ref{l.expand} once we choose $P=u_2(g^{-1}|_{gL})$,  $Q=u_1(g^{-1}|_{gL})$. 
\end{proof}
Another useful corollary of Lemma \ref{l.expand} is the following.
\begin{cor}\label{l.expandold}
Given $\alpha > 0$, there exist positive $\delta$ and $b$ with the following properties.
Let $V$ be a $d$-dimensional $\K$-vector space, and $g\in \GL(V)$.
Suppose that $P \in \Gr_p(V)$ and
$Q \in \Gr_{d-p}(V)$ satisfy
\begin{equation}\label{e.bi_separation}
\min \{ \angle(P,Q), \angle(gP,gQ) \} \ge \alpha \, .
\end{equation}
Then for every $P_i \in \Gr_p(V),$ ($i=1,2$) with $P_i\cap Q=\{0\}$ such that $d(gP_i,gP)<\delta$ one has 
\begin{equation}\label{e.min_expansion}
d(gP_1,gP_2) \ge b \, \frac{\mm(g|_Q)}{\| g|_P \|} \, d(P_1,P_2) \, .
\end{equation}
\end{cor}
\begin{proof}
Since all good norms are equivalent, the general case follows from Lemma \ref{l.expand} by {considering two norms, one for which $P$ and $Q$ orthogonal and one that makes $gP$ and $gQ$ orthogonal, the operator norm and $\mm$ are to be computed using both these norms.}
\end{proof}
{Along the same lines we get a bound on how elements $g\in\GL(V)$ contract on open sets in Grassmannians:
\begin{cor}\label{c.least contraction}
Let $g \in \GL(V,\K)$ have a gap of index $p$.
Then, for every $\alpha>0$ there is $b$ such that  for all $P_1, P_2 \in \Gr_{p}(V)$  with $\angle(P_i,g^{-1}U_{d-p}(g^{-1}))>\alpha$  we have:
$$
d(g(P_1), g(P_2)) \le b\frac{\sigma_{p+1}}{\sigma_p}(g) \, {d(P_1,P_2)} \, .
$$
\end{cor} 
\begin{proof}
If we  assume that $d(P_i, U_{p}(g))\geq 1/\sqrt 2$ the result follows readily from Lemma \ref{l.d_and_norm} by 
considering the linear maps $L_i:= L_{P_i,U_p(g)}$ and $M_i:=L_{gP_i,gU_p(g)}$. As above $L_i  =  (g^{-1}|_{U_{d-p}(g^{-1})}) \circ  M_i \circ (g|_{U_p(g)})$. In this case the result follows as $\mm( g|_{U_p(g)})=\sigma_p(g)$, and  $\mm(g^{-1}|_{U_{d-p}(g^{-1})})=1/\sigma_{p+1}(g)$.  The general statement follows by comparison of different norms.
\end{proof}
}
\subsection{Dynamical background}\label{s.2.2}
We now turn to the dynamical preliminaries. The goal of this section is to extend the results of Bochi-Gurmelon \cite{BG} and Bochi-Potrie-S. \cite{BPS} to the non-Archimedean setting.
\subsubsection{Dominated splittings and Bochi-Gurmelon's Theorem}\label{sec:2.2}
In this section we recall the definition of dominated splittings and review its connection with cone fields. 

Let $X$ be a compact metric space equipped with a continuous homeomorphism $\vartheta:X\to X.$ Let $V$ be a finite dimensional $\K$-vector space and let $\psi_0:X\to\GL(V,\K)$ be continuous. We will denote by $\psi:X\times V\to X\times V$ the induced cocycle defined by $$\psi_x(v)=\psi(x,v)=(\vartheta (x),\psi_0(x)v).$$

\begin{defi}\label{d.domination} Consider a good norm $\|\,\|$ on $V.$ Let $\L\subset X$ be a $\vartheta$-invariant subset, then we say that $\psi|\L$ has a \emph{dominated splitting} if the trivial bundle $\L\times V$ splits as a Whitney sum of two $\psi$-invariant sub-bundles $V=E\oplus F$ with the following extra condition: there exist positive $\mu$ and $c$ such that for every $n$ positive, $x\in\L,$ $u\in E_x$ and $w\in F_x$ one has $$\frac{\|\psi_x^nu\|}{\|\psi_x^nw\|}\leq c e^{-\mu n}\frac{\|u\|}{\|w\|}.$$ In this situation we say moreover that $F$ (resp. $E$) is the unstable (resp. stable) bundle and that $F$ \emph{dominates} $E.$ 
\end{defi}

Note that this condition is independent of the chosen norm. The dominated splitting of $\psi|\L$ is unique provided \emph{its index}, i.e. $\dim_\K F,$ is fixed and it extends to the closure $\overline{\L}$ of $\L$ (see Crovisier-Potrie \cite[Proposition 2.2 and 2.5]{CP} whose proof works verbatim in our setting). Furthermore:
\begin{prop}[{\cite{CP}}]\label{p.splitting compatibility}
Suppose a linear flow $\psi$ has dominated splittings $E^1\oplus F^1$ and $E^2\oplus F^2$ of index $p_1\leq p_2$. Then $E^2\subseteq E^1$ and $F^1\subseteq F^2$.
\end{prop}

In the case when $\K=\R$ Bochi-Gourmelon \cite[Theorem A]{BG} gave the following criterion for dominated splittings to exist; their the proof generalizes to every local field $\K$, as Oseledets theorem holds in this generality:

\begin{thm}[Bochi--Gourmelon {\cite{BG}}]\label{t.BG} Let $X$ be a compact metric space, $V$ a $\K$-vector space and $\psi:X\times V\to X\times V$ a linear cocycle. Then the linear flow  $\psi$ has a dominated splitting $E \oplus F$ with $\dim F=p$ if and only if there exist $c>0$, $\mu>0$ such that for every $x \in X$ and $n\geq 0$ we have
$$
\frac{\sigma_{p+1}}{\sigma_p} (\psi^n_x)< c e^{-\mu n} \, .
$$
Moreover, the bundles\footnote{For completeness, let us note that the space $U_{p}$ associated to an operator from a vector space equipped with a good norm to itself, can be defined for an operator between two vector spaces both equipped with good norms.} are given by:
$$
F_x  =  \lim_{n \to +\infty} U_p \big( \psi^n_{\vartheta^{-n}(x)} \big)\textrm{ and }E_x =  \lim_{n \to +\infty} U_{d-p} \big( \psi^{-n}_{\vartheta^n(x)} \big),$$
and these limits are uniform.
\end{thm}

\begin{proof}
Bochi-Gourmelon's proof is based on the one hand on some angle estimates building upon the min-max characterization of singular values of matrices in $\GL_d(\R)$, and on the other hand on the multiplicative ergodic theorem (Oseledets theorem). The former hold verbatim in the general local field setting once the singular values are replaced by the semi-homothecy ratios as defined in Section \ref{ss.dom_sing}, the required multiplicative ergodic theorem was established (following Oseledets original proof) by Margulis \cite[Theorem V.2.1]{Marg}, the integrability of $\psi$ follows from its continuity and the compactness of the base $X$. With these ingredients at hand, the sketch of the proof explained in \cite[Section A.4]{BPS} applies verbatim.
\end{proof}

The existence of a dominated splitting can be furthermore characterized in terms of cone fields; this will be crucial to prove openness of Anosov representations in Section \ref{s.anosov} (note that the non-Archimedean case has not yet been established). Given a decomposition $V=V_1\oplus V_2$ and a positive $a$, then the subset defined by 
$$\{v\in V: a\|v_1\|\geq \|v_2\|\}$$ is called a $a$-\emph{cone} (of dimension $\dim V_1$) on $V.$ 

A \emph{cone field} on $\L\subset X$ is a continuous choice $x\mapsto \scr C_{a(x),x}$ of a $a(x)$-cone on $V$ (of fixed dimension) for each $x\in \L.$ Cone fields can be used to characterize dominated splittings. 

\begin{prop}[{See Sambarino \cite[Proposition 2.2]{Samb}}]\label{prop_conefields}
Let $\L\subset X$ be $\vartheta$-invariant. The cocycle $\psi|_\L$ has a dominated splitting of index $i$  if and only if there exists  a map $a:\Lambda\to\R^+$ bounded away from 0 and $\infty$, a cone field $\scr C_{a(x),x}$ on $\Lambda$ of dimension $i$, a number $0<\lambda<1$ and a positive integer $n_0$ such that for every $x\in\L$ the closure of $\psi_x^{n_0}(\scr C_{a(x),x}) $ is contained in $\scr C_{\lambda a(\vartheta^n(x)),\vartheta^n(x)}.$
\end{prop}

\subsubsection{Dominated sequences}
Bochi-Potrie-S. \cite[Section 2]{BPS} applied Bochi-Gourmelon's Theorem \ref{t.BG}  to the compact space of \emph{dominated sequences of matrices}, and got useful implications on the relative position of the axes of the ellipsoid associated to the products of such sequences: we recall now the relevant definitions and results from \cite{BPS} where these were first established.

Given $C>1$, define the following compact set: 
$$
\cal D(C) \coloneqq \big\{ g \in \GL(V,\K), \ \|g\| \le C , \ \|g^{-1}\| \le C \big\}.
$$
If $I$ is a (possibly infinite) interval in $\Z$, the set $\cal D(C)^I$ is endowed with the product topology, turning it into a compact metric space.

Let $p \in \lb1,d-1\rb$, $\mu>0$, $c>0$.
For each interval $I\subset \Z$, we denote by $\cal D(C,p,c,\mu,I)$ the set of sequences of matrices $(g_n) \in \cal D(C)^I$
such that for all $m, n \in I$ with $m \ge n$ we have
$$ 
\frac{\sigma_{p+1}}{\sigma_p} (g_{m} \cdots g_{n+1} g_n)  \leq c e^{-\mu (m-n+1)}  \, .
$$
\begin{defi}\label{dominatedsequence} An element of $\GL(V,\K)^I$ is a \emph{dominated sequence} if it belongs to $\cal D(C,p,c,\mu,I)$ for some $C,\,p,\,c,$ and $\mu.$
\end{defi}

Consider the map $\shift:\cal D(C)^\Z\to\cal D(C)^\Z$ defined by $$\shift((g_n)_{-\infty}^\infty)=(g_{n+1})_{-\infty}^\infty$$ and let $\psi_0:\cal D(C)^\Z\to\GL(V,\K)$ be $\psi_0((g_n))=g_0.$ The subsets $\cal D(C,p,c,\mu,\Z)$ are $\shift$-invariant and automatically verify the hypothesis of Theorem \ref{t.BG}.

\begin{prop}[{Bochi-Potrie-S. \cite[Proposition 2.4]{BPS}}]\label{prop:BPS2.5} For each sequence $x  =  (g_n) \in \cal D(C,p,\mu,c,\Z)$, the limits:
\begin{align*}
F_x &\coloneqq \lim_{n \to +\infty} U_p \big( g_{-1} g_{-2} \cdots g_{-n} \big) \, , \\
E_x &\coloneqq \lim_{n \to +\infty} U_{d-p} \big(g_0^{-1}\cdots g_{n-2}^{-1} g_{n-1}^{-1} \big) \, ,
\end{align*}
exist and are uniform over $\cal D(C,p,\mu,c,\Z)$. 
Moreover, $F$ dominates $E$ and $E \oplus F$ is a dominated splitting for the linear cocycle over the shift defined above.
\end{prop} 

By a compactness argument, the proposition above ensures transversality for Cartan attractors and repellers computed in finite, but sufficiently long, sequences of matrices:
\begin{lemma}[{Bochi-Potrie-S. \cite[Lemma 2.5]{BPS}}]\label{l.seq_splitting}
Given $C>1$, $\mu>0$, and $c>0$, there exist $L \in \N$ and $\delta>0$
with the following properties.
Suppose that $I \subset \Z$ is an interval and $\{g_i\}_{i\in I}$ is an element of $\cal D(C,p,c,\mu,I)$.
If $n < k < m$ all belong to $I$ and $\min \{ k-n, m-k \} > L$ then:
$$
\angle \big( U_p(g_{k-1} \cdots g_{n+1} g_n) , \, U_{d-p}(g_k^{-1}g_{k+1}^{-1}\cdots g_{m-1}^{-1} \big) > \delta \, .
$$
\end{lemma}

\subsection{Hyperbolic groups}\label{s.2.3}
The last source of preliminaries comes from geometric group theory. Here we recall basic facts about hyperbolic groups and cone types.

Let $\G$ be a finitely generated group. We fix a finite symmetric generating set $S$ and denote by $|\,|$ the associated word length: for $\g\in\G-\{e\}$ we denote by $|\g|$ the least number of elements of $S$ needed to write $\g$ as a word on $S,$ and define the induced distance $d_\G(\g,\eta)=|\g^{-1}\eta|.$ A geodesic segment on $\G$ is a sequence $\{\alpha_i\}_0^k$ of elements in $\G$ such that $d_\G(\alpha_i,\alpha_j)=|i-j|.$

In the paper we will be only interested in word-hyperbolic groups, namely such that the metric space $(\G,|\,|)$ is Gromov hyperbolic. Following the footprints of \cite{BPS}, our analysis will be based on the study of cone types, and natural objects associated to them. 
\subsubsection{Cone types}\label{conetypes}
In the paper we follow Cannon's original definition of cone types, which is more convenient for our geometric purposes, but the reader should be warned that the definition used in \cite{BPS} is slightly different
\begin{defi}
The \emph{cone type} of $\g\in\G$ is defined by $$\cone(\g)=\{\eta\in\G:|\g\eta|=|\eta|+|\g|\}.$$
\end{defi}
Notice that if $\eta\in\cone(\g)$ then $$d_\G(\g^{-1},\eta)=|\g\eta|=|\eta|+|\g|=|\eta|+|\g^{-1}|=d_\G(e,\eta)+d_\G(e,\g^{-1}),$$ i.e. there exists a geodesic segment through $e$ with endpoints $\g^{-1}$ and $\eta.$ Reciprocally, the endpoint of a geodesic segment starting at $\g^{-1}$ and passing through $e$ necessarily belongs to $\cone(\g).$

\begin{figure}[hbt]
	\centering
	\begin{tikzpicture}[scale = 0.5]
	\draw [thick] circle [radius = 3];
	\draw [fill] circle [radius = 0.03];	
	\node [left] at (0,0) {$e$};
	\draw [thick] (0.004,-1.3) -- (0.1,-0.7) -- (-0.1,-0.3 ) -- (0,0);
	\draw [thick] (0,0) -- (-0.3,1) -- (0.2,2) -- (0,3);
	\draw [thick] (0.2,2) -- (0.4, 2.7) -- (0.7,2.91);
	\draw [thick] (1,0.2) -- (0.89,1.37);
	\draw [thick] (0.2,2) -- (0.4, 2.7);
	\draw [thick] (1.8,0.8) -- (2.7,1.3);
	\draw [thick] (0,0) -- (1,0.2) -- (1.8,0.8) -- (2.846,0.948);
	\node [below] at (0.004,-1.3) {$\g^{-1}$};
	\draw [fill] (0.7,2.91) circle [radius = 0.03];
	\draw [fill] (2.7,1.3) circle [radius = 0.03];
	\draw [fill] (0.004,-1.3) circle [radius = 0.03];
	\draw [fill] (2.12,2.12) circle [radius = 0.03];
	\draw [fill] (0,3) circle [radius = 0.03];	
	\draw [fill] (2.846,0.948) circle [radius = 0.03];	
	\draw [thick] (0,0) -- (0.89,1.37) -- (1.2,2) -- (2.12,2.12);
	\draw [thick] (1.2,2) -- (1.4,2.65);
	\draw [fill] (1.4,2.65) circle [radius = 0.03];
	\node [right] at (3,-2) {$\G$};
	\draw [->] (3.5,3) to [out = 180,in = 90] (1.5,1.2);
	\node [right] at (3.5,3) {$\cone(\g)$};
	\end{tikzpicture}
	\caption{The cone type of $\g\in\G$}\label{figure:conetype}
\end{figure}
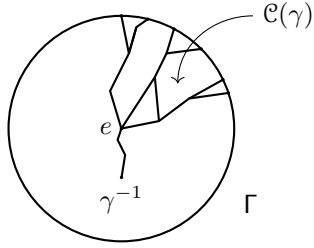
A fundamental result of Cannon is that, provided $\G$ is hyperbolic, there are only finitely many cone types (see for example Bridson-Heafliger \cite[p.~455]{BH} or Coornaert-Delzant-Papadopoulos \cite[p.~145]{CDP}).

\subsubsection{The geodesic automaton}\label{s.geodaut}

See Bridson-Haefliger \cite[p.~456]{BH}.

Given a cone type $\cone$ and $a \in S \cap \cone$, one easily checks that for every $\g \in \G$ with $\cone(\g)  =  \cone$ one gets 
$$
a\cone(\g a)\subset \cone(\g).
$$ 
Furthermore it is easy to verify that in such case the cone type $\cone(\g a)$ doesn't depend on $\gamma$ (see for example \cite[Lemma 4.3]{CDP}), and, with a slight abuse of notation we will denote such cone type $a\cdot \cone$.

The \emph{geodesic automaton} of $\G$ (this also depends on $S$) is the labelled graph $\cal G$ defined as follows:
\begin{itemize}
\item the vertices are the cone types of $\G$; 
\item there is an edge $\cone_1 \xrightarrow{a} \cone_2$ from vertex $\cone_1$ to vertex $\cone_2$, labelled by a generator $a \in S$, 
iff  $a\in \cone_1$ and $\cone_2  =   a\cdot\cone_1$.
\end{itemize}

Since $\G$ is hyperbolic there are only finitely many cone types and thus the geodesic automaton has a finite number of vertices.

Let us explain the relation with geodesics.
Consider a \emph{geodesic segment} $(\g_0, \g_1, \dots, \g_\ell)$
that is, a sequence of elements of $\G$ such that $d(\g_n,\g_m)  =  |n-m|$,
and assume that $\g_0  =  \id$.
Then, there are $a_0$, \dots $a_{\ell-1}$ in a generating set $S$ such that 
$\g_n  =  a_0 a_1 \cdots a_{n-1}$.
Note that for each $n$, the following is an edge of the geodesic automaton graph $\cal G$:
$$
\cone(\g_n) \xrightarrow{a_n} \cone(\g_{n+1}). 
$$
Thus we obtain a (finite) walk on $\cal G$ starting from the vertex $\cone(\id)$.
Conversely, for each such walk we may associate a geodesic segment starting at the identity.

Let us define also the \emph{recurrent geodesic automaton} as the maximal recurrent subgraph $\cal G^*$ of $\cal G$; its vertices are called \emph{recurrent cone types}.

Let $\L_\G$ be the subset of all bi-infinite labelled sequences of $\cal G^*.$ It is a closed shift-invariant subset of $(\cal G^*)^\Z$ and the induced dynamical system $\shift:\L_\G\to\L_\G$ is a sofic shift (as in Lind-Marcus \cite{LM}).

The following concept will be useful in Section \ref{s.5}.
\begin{defi}\label{d.nested}
Given an integer $k$ we say that two cone types $\cone_1, \cone_2$ are  \emph{$k$-nested} if there is a path of length $k$ in the geodesic automaton from $\cone_1$ to $\cone_2$. In this case there is an element $\beta\in\Gamma$ with $|\beta|=k$ and such that $\beta\cone_2\subset\cone_1.$
\end{defi}

Since $\G$ is hyperbolic, there are only finitely many cone types, therefore, for every $k$, there are only finitely many $k$-nested pairs of elements (however, as soon as $\G$ is non-elementary, the number of $k$-nested pairs grows exponentially with $k$). The following is clear from the definitions:
 
\begin{lemma}\label{l.nested}
If $\{\alpha_i\}\subset \G$ is a geodesic, then the pair $(\cone(\alpha_i),\cone(\alpha_{i+k}))$ is $k$-nested.
\end{lemma}

\subsubsection{Coverings of the Gromov boundary}\label{s.covering}
Recall that, as $\G$ is Gromov hyperbolic, its boundary $\bord\G$, consisting of equivalence classes of geodesic rays, is well defined up to homeomorphism.  
We associate to every cone type $\cone$ which is not the cone type of the identity a subset of $\bord\G$, the \emph{cone type at infinity},  by considering limit points of geodesic rays starting on $e$ and totally contained in $\cone:$
$$\bordcone =\{[(\alpha_i)]| (\alpha_i)\text{  geodesic ray }, \alpha_0=e, \alpha_i\in\cone\}.$$
It follows from the discussion in the previous paragraph that every point in $\bord\G$ is contained in at least one of the sets $\bordcone$. As there are only finitely many cone types, we obtain a finite covering of $\bord\G$ by considering $\cal U=\{\bordcone(\gamma)\}$. Starting from this covering we will construct new coverings that will serve as our Sullivan shadows:
\begin{figure}[hh]
	\centering 
	\begin{tikzpicture}[scale = 0.5]
	\begin{scope}[rotate=45]
	
	\draw [thick] circle [radius = 3];
	\draw [fill] circle [radius = 0.03];	
	\node [below left] at (0,0) {$\g$};
	\draw [thick] (0.004,-1.3) -- (0.1,-0.7) -- (-0.1,-0.3 ) -- (0,0);
	\draw [thick] (0,0) -- (-0.3,1) -- (0.2,2) -- (0,3);
	\draw [thick] (0.2,2) -- (0.4, 2.7) -- (0.7,2.91);
	\draw [thick] (0.2,2) -- (0.4, 2.7);
	\draw [thick] (1,0.2) -- (0.89,1.37);
	\draw [thick] (1.8,0.8) -- (2.7,1.3);
	\draw [thick] (0,0) -- (1,0.2) -- (1.8,0.8) -- (2.846,0.948);
	\node [below] at (0.004,-1.3) {$e$};
	\draw [fill] (0.7,2.91) circle [radius = 0.03];
	\draw [fill] (2.7,1.3) circle [radius = 0.03];
	\draw [fill] (0.004,-1.3) circle [radius = 0.03];
	\draw [fill] (2.12,2.12) circle [radius = 0.03];
	\draw [fill] (0,3) circle [radius = 0.03];	
	\draw [fill] (2.846,0.948) circle [radius = 0.03];
	\draw [thick] (0,0) -- (0.89,1.37) -- (1.2,2) -- (2.12,2.12);
	\draw [thick] (1.2,2) -- (1.4,2.65);
	\draw [fill] (1.4,2.65) circle [radius = 0.03];
		\node [right] at (3,-2) {$\G$};
 \filldraw [red](0,3) circle [radius=2pt];
 \filldraw [red](0.7,2.91) circle [radius=2pt];
\filldraw [red](2.846,0.948) circle [radius=2pt];
\filldraw [red](2.7,1.3) circle [radius=2pt];
\filldraw [red](2.12,2.12) circle [radius=2pt];
\filldraw [red] (1.4,2.65) circle [radius=2pt];
\end{scope}

	\end{tikzpicture}
	\caption{The set $\g\cdot\bordcone(\g)$}\label{figure:ImageOfConetype}
\end{figure}

\begin{lemma}\label{SullivanShadows} Given $T>0,$ the family of open sets $$\cal U_T\coloneq\{\g\bordcone(\gamma):|\g|\geq T\}$$ defines an open covering of $\bord\G.$
\end{lemma}

\begin{proof} We have to check that every point $x\in\bord\G$ is covered, but this is evident since considering a geodesic ray $(\alpha_i)_0^\infty$ in $\G$ starting from $e$ converging to $x,$ one has that for all $i,$ $x\in\alpha_i\bordcone(\alpha_i),$ see Figure \ref{figure:ImageOfConetype}.
\end{proof}

\section{Anosov representations}\label{s.3}
Anosov representations  from fundamental groups of negatively curved closed manifolds to $\PGL(d,\R)$ were introduced by Labourie \cite{Labourie-Anosov} and generalized by Guichard-W. \cite{GW-Domains} to any hyperbolic group. In this section we will generalize  to non-Archimedean local fields the work of \cite{BPS}, which provides a simplified definition\footnote{Morse actions on Euclidean buildings (and thus in particular Anosov subgroups of $\PGL_d(\K)$ when $\K$ is non-Archimedean) were already defined by Kapovich-Leeb-Porti \cite[Definition 5.35]{KLP-Morse}, the interest of such concept was also suggested in G\'eritaud-Guichard-Kassel-W. \cite[Remark 1.6 (a)]{GGKW}.}.

\subsection{Anosov representations and dominated splittings}\label{s.anosov} 

Let $\G$ be a discrete group of finite type, fix a finite symmetric generating set $S_\G$ and denote by $|\,|$ the associated word length.
\begin{defi}\label{defAnosov} Consider $p\in\lb1,d-1\rb.$ A representation $\rho:\G\to\PGL_d(\K)$ is \emph{$\{\sroot_p\}$-Anosov}\footnote{In the language of Bochi-Potrie-S. \cite[Section 3.1]{BPS} a $\{\sroot_p\}$-Anosov representation  is called \emph{$p$-dominated}.}
 if  there exist positive constants $c,\mu$ such that for all $\g\in\G$ one has 
\begin{equation}\label{def1}
\frac{\sigma_{p+1}}{\sigma_{p}}\big(\rho(\g)\big)\leq ce^{-\mu|\g|}.
\end{equation} A $\{\sroot_1\}$-Anosov representations will be called \emph{projective Anosov}.
\end{defi}

One has the following direct remark.

\begin{obs}[Bochi-Potrie-S.]\label{Anosov-dom}
Let $\rho:\G\to\PGL_d(\K)$ be $\{\sroot_p\}$-Anosov. Given a geodesic $\{\alpha_i\}_{i\in\Z}$ let us denote by $\over{\alpha}_i=\alpha_{i+1}^{-1}\alpha_i\in S_\G,$ then we have 
 $$\big(\rho(\over\alpha_i)\big)_{i\in\Z}\in \cal D(C,p,c,\mu,\Z)$$ 
 where $c,\mu$ come from equation (\ref{def1}) and $C=\max\{\|\rho(a)\|:a\in S_\G\}.$ Note also that  
$$\big(\rho(\over\alpha_i)\big)_{i\in\Z} \in\cal D(C,d-p,c,\mu,\Z),$$
 and thus Theorem \ref{t.BG} provides the following splittings of $\K^d$ $$E^p_{(\rho(\over\alpha_i))}\oplus F^{d-p}_{(\rho(\over\alpha_i))}\textrm{ and } E^{d-p}_{(\rho(\over\alpha_i))}\oplus F^{p}_{(\rho(\over\alpha_i))},$$ 
with the obvious inclusions according to dimension. {By domination, these four bundles vary continuously\footnote{{This follows from Proposition \ref{prop_conefields}, see also, for example, \cite[Theorem A.15]{BPS}.}} in $\cal D(C,p,c,\mu,\Z)\cap \cal D(C,d-p,c,\mu,\Z).$}
 Finally, Proposition \ref{prop:BPS2.5} yields, for $k\in\{p,d-p\}$ and $m\geq0$ $$U_k(\rho(\alpha_m))=U_k(\over\alpha_0^{-1}\cdots\over\alpha_m^{-1})\to E^k_{(\rho(\over\alpha_i))}$$ and $$U_k(\rho(\alpha_{-m}))=U_k(\over\alpha_{-1}\cdots\over\alpha_{-m})\to F^k_{(\rho(\over\alpha_i))}$$ as $m\to\infty.$
 \end{obs}

Using dominated splittings it is possible to deduce strong angle estimates between Cartan attractors along geodesic rays through the origin; for example the next result is a direct consequence of Lemma \ref{l.seq_splitting}.
\begin{prop}[Bochi-Potrie-S.]\label{weakmorselemma} 
Let $\rho:\G\to\PGL_d(\K)$ be a $\{\sroot_p\}$-Anosov representation. 
 Then there exists $\delta>0$ and $L\in\N$ such that for every geodesic segment $(\alpha_i)_0^k$ in $\G$  through $e$ with $|\alpha_0|,|\alpha_k|\geq L$ one has 
 $$\angle\Big(U_p\big(\rho(\alpha_k)\big),U_{d-p}\big(\rho(\alpha_0)\big)\Big)>\delta.$$ 
\end{prop}

Bochi-Potrie-S. \cite{BPS} applied the theory of dominated splittings to the sofic shift $\Lambda_\G\to\Lambda_\G$ induced by the recurrent geodesic automaton (see Section \ref{s.geodaut}), to get  an easy proof of openness of Anosov representations. Their proof easily  extends to every local field:

\begin{prop}\label{p.open}
The set of $\{\sroot_p\}$-Anosov representations is open in $\hom(\G,\PGL_d(\K)).$ 
\end{prop}
\begin{proof}
A representation $\rho:\G\to\PGL_d(\K)$ induces a linear cocycle $A_\rho$ over $\L_\G$, which admits a dominated splitting if and only if the representation $\rho$ is $\{\sroot_p\}$-Anosov.
Observe that the cocycle $A_\rho$ varies continuously with the representation, since it only depends on the value of $\rho$ on a generating set of $\G.$
Since, by Proposition \ref{prop_conefields}, having a dominated splitting is an open condition on the space of cocycles, the result follows.
\end{proof}

\subsection{Boundary maps}
From now on we will assume that $\G$ is a word hyperbolic group. This is not a restriction:  Kapovich-Leeb-Porti proved that the only groups admitting Anosov representations are hyperbolic \cite[Theorem 6.15]{KLP-Morse} (cfr. also Bochi-Potrie-S. \cite{BPS} for a different proof in the Archimedean case). We can thus talk freely  about the Gromov boundary $\partial \G.$

 An important property of $\{\sroot_p\}$-Anosov representations is that they admit equivariant boundary maps:
\begin{prop}[{Bochi-Potrie-S. \cite[Proposition 4.9]{BPS}}]\label{p.bdry}
If $\rho:\G\to\PGL_d(\K)$ is $\{\sroot_p\}$-Anosov, then for any geodesic ray $\{\gamma_n\}$ with endpoint $x$, the limits
$$\xi^p_\rho(x):=\lim_{n\to\infty}U_p(\rho(\gamma_n))\quad \xi^{d-p}_\rho(x):=\lim_{n\to\infty}U_{d-p}(\rho(\gamma_n))$$
exist and do not depend on the ray; they define continuous $\rho$-equivariant transverse maps $\xi^{p}:\bord\G\to\Gr_p(\K^d)$, $\xi^{d-p}:\bord\G\to\Gr_{d-p}(\K^d)$.  
\end{prop} 

\begin{proof}
The proof in \cite[Proposition 4.9]{BPS} works without modification in our context: despite $U_p(\rho(\gamma_n))$ is not uniquely defined, Lemma \ref{l.qg} (\ref{l.qg1}) guarantees that, for every choice of $U_p(\rho(\gamma_n))$, the sequence $\{U_p(\rho(\gamma_n))\}$ is Cauchy, and therefore has a limit; furthermore, since any pair of geodesic rays defining $x$ is at bounded distance, Lemma \ref{l.qg} (\ref{l.qg1})  shows that the limit doesn't depend on the chosen sequence, and the maps are continuous. The equivariance follows from Lemma \ref{l.qg} (\ref{l.qg2}).\end{proof}

The uniformity of the limits in Proposition \ref{p.bdry} can be quantified explicitly (cfr. \cite[Lemma 4.7]{BPS}). This will be useful in the proof of Theorem \ref{FRep->Fpoint}:
\begin{lemma}\label{l.uniformbdry}
 Let $\rho:\G\to\PGL_d(\K)$ be  $\{\sroot_p\}$-Anosov. Then there exist constants $C,\mu$ such that, for every $\alpha\in\Gamma$ and every $x\in\alpha\bordcone(\alpha)$, 
$$d(\xi^p(x),U_p(\rho(\alpha)))\leq Ce^{-\mu |\alpha|}.$$
In particular, given $\eps>0$ there exists $L\in\N$ such that 
$$\overline{\bigcup_{\g:|\g|\geq L} U_p\big(\rho(\g)\big)}\subset \cal N_\epsilon(\xi^p_\rho(\bord\G)).$$
\end{lemma}
\begin{proof}
If $x\in\alpha\bordcone(\alpha)$ there exists a geodesic ray $(\alpha_i)_{i\in\N}$ through $\alpha$ with endpoint $x$. In particular we get 
$$d(U_p\big(\rho(\alpha)\big),\xi_\rho^p(x))\leq \sum_{i\geq |\alpha|}d(U_p\big(\rho(\alpha_i)\big),U_p\big(\rho(\alpha_{i+1})\big))\leq c\sum_{i\geq |\alpha|} e^{-\mu i}.
$$
Here the first inequality is a consequence of the triangular inequality, the second follows from Lemma \ref{l.qg} (\ref{l.qg1}).

The second statement follows since, as $\G$ is word hyperbolic, there is a constant $D$ such that, for every $\g\in\G$ we can choose  a geodesic ray $(\alpha_i)_{i\in\N}$, and $k\in\N$ such that $d(\gamma,\alpha_k)\leq D$. This implies that $\gamma=\alpha_k h$ with $|h|\leq D$. Let $x$ be the endpoint of the geodesic $(\alpha_i)_{i\in\N}$.
 Then 
$$d(U_p\big(\rho(\g)\big),\xi_\rho^p(x))\leq d(U_p\big(\rho(\alpha_kh)\big),U_p\big(\rho(\alpha_k)\big))+d(U_p(\rho(\alpha_i)),\xi^p(x))).$$

\end{proof}

Bochi-Potrie-S. \cite{BPS} observed that the boundary map has an explicit characterization in term of the linear cocycle $A_\rho$ over the sofic shift $\Lambda_\G$ (described in the proof of Proposition \ref{p.open}). Recall from Definition \ref{d.domination} that whenever a cocycle $A_\rho:\Lambda_\G\times \K^d\to \Lambda_\G\times \K^d$ has a dominated splitting we denote by $E$ (resp. $F$) the stable (resp. unstable) bundle. 

 \begin{prop}[{ \cite[Proposition 5.2]{BPS}}]\label{cocycle.automaton} Let $\rho:\G\to\PGL_d(\K)$ be  $\{\sroot_p\}$-Anosov. Let $x,y\in\bord\G$ and  $(\over\alpha_i)_{-\infty}^\infty\in\L_\G$ a geodesic from $y$ to $x$. Then one has 
$$\xi^{p}(x)=E_{(\rho(\over\alpha_i))}^p\textrm{ and  }\;\xi^{d-p}(y)=F_{(\rho(\over\alpha_i))}^{d-p}.$$
 \end{prop}

As a corollary we can follow \cite{BPS} and generalize to the non-Archimedean case the following important fact originally proved by Labourie \cite{Labourie-Anosov} and Guichard-W. \cite{GW-Domains}. 

\begin{cor}\label{bdry.continuous}Let $\rho:\G\to\PGL_d(\K)$ be  $\{\sroot_p\}$-Anosov. The boundary maps $\xi^{p}:\partial\G\to\Gr_p(\K^d)$, $\xi^{d-p}:\partial\G\to\Gr_{d-p}(\K^d)$
vary continuously with the representation.
\end{cor}
\begin{proof}
This follows at once from the arguments in the proof of Propositions \ref{p.open} and \ref{cocycle.automaton}, as splittings vary continuously with the cocycles by Proposition \ref{prop_conefields} {(this is a standard argument, see \cite[Theorem A.15]{BPS} for a proof).} 
\end{proof}

The boundary map, which is unique, gives a realization of the boundary $\bord\G$ in $\Gr_p(\K^d)$, a  space where the dynamics is governed by ratios of semi-homothecy ratios of  elements in the projective linear group. To stress this fact and the dependence on $\rho$ we introduce the following notation, which will be heavily used in the rest of the paper: 
if $\rho:\G\to\GL_d(\K)$ is a $\{\sroot_p\}$-Anosov representation with equivariant boundary map $\xi^p$, and $x\in\bord\G$, we will write 
$$x^p_\rho:=\xi^p(x).$$We noticed that this notation improves readability of many formulas and conveniently stresses the dependence of $\xi$ on $\rho$.

\subsection{Geometric estimates}
We conclude  the section on Anosov representations by collecting a number of geometric lemmas that will be useful later on. The first result provides the quantification we will need of the following geometric principle: endpoints of a geodesics through the origin are uniformly far in the visual boundary, the same  holds for their image under the boundary map associated to an Anosov representation. 

\begin{lemma}\label{appart} Let $\rho:\G\to\PGL_d(\K)$ be  $\{\sroot_p\}$-Anosov representation. Then there exists $\nu>0,$ depending only on $\rho$ such that: if $\{\alpha_i\}_{i\in\Z}\subset\G$ is a geodesic through $\id$ with endpoints $x,z\in\bord\G,$ then for all $i\in\Z$ one has $$\angle(\rho\big(\alpha^{-1}_i\big)x^p_\rho,\rho\big(\alpha^{-1}_i\big)z^{d-p}_\rho)>\nu.$$
\end{lemma}

\begin{proof}
Recall that we denote by $\cal D(C,d-p,c,\mu,\Z)$ the compact, shift invariant space of dominated sequences (cfr. Definition \ref{dominatedsequence}). The bundle $\cal D(C,d-p,c,\mu,\Z)\times \K^d$ admits a dominated splitting $E^p\oplus F^{d-p}$ and, by compactness,
 we get 
  $$\nu=\inf_{\{g_i\}\in\cal  D(C,d-p,c,\mu,\Z)}\angle( E_{(g_i)}^{p},F_{(g_i)}^{d-p})>0.$$ 

Remark \ref{Anosov-dom} implies that, since $\rho$ is $\{\sroot_p\}$-Anosov, $(\rho(\alpha^{-1}_{i+1}\alpha_i))_{i\in\Z}\in\cal D(C,d-p,c,\mu,\Z)$, furthermore one directly computes that 
 $$\psi^n((\rho(\alpha^{-1}_{i+1}\alpha_i))_{i\in\Z},v)=((\rho(\alpha^{-1}_{i+1}\alpha_i))_{i-n\in\Z},\rho\big(\alpha_n^{-1}\big)v).$$

As we know from Proposition \ref{prop:BPS2.5}  that 
$$x^p_\rho=\lim_{i\to\infty} U_p(\rho(\alpha_i))=E_{(\rho(\over\alpha)_i)}^p$$
 and 
 $$z^{d-p}_\rho=\lim_{i\to\infty} U_{d-p}(\rho(\alpha_{-i}))=F_{(\rho(\over\alpha)_i)}^{d-p},$$
 we deduce that for all $i\in\Z$ one has 
  $$\angle(\rho\big(\alpha_i^{-1}\big)x^p_\rho,\rho\big(\alpha_i^{-1}\big)z^{d-p}_\rho)>\nu.$$

\end{proof}
 
The next lemma will be crucial in Section \ref{section:FrenetReps}. It quantifies how the inverse of elements in a geodesic expand the distances exponentially in neighbourhoods of their Cartan attractors; this should be compared with \cite[Corollary A.14]{BPS}:  

\begin{lemma}\label{expand2}
Let $\rho:\G\to\PGL_d(\K)$ be  $\{\sroot_p\}$-Anosov. There exist positive constants  $\overline c,\overline\mu,\delta$ depending only on $\rho$, and $L\in\N$ such that, for every geodesic ray $\{\alpha_i\}_{i\in\N}\subset\G,$ with $\alpha_0=\id$ and endpoint $x$, every $i\geq L$, and every $z,w\in\partial\G$ satisfying $z^p_\rho,w^p_\rho\in B_\delta(x^p_\rho)$, and $\rho(\alpha_i^{-1})\{z^p_\rho,w^p_\rho\}\subset B_\delta(\rho(\alpha_i^{-1})x^p_\rho)$, we have 
$$d(\rho(\alpha_i^{-1})w^p_\rho,\rho(\alpha_i^{-1})z^p_\rho)\geq \overline ce^{\overline\mu i}d(w^p_\rho,z^p_\rho).$$
\end{lemma}
\begin{proof}
We complete the ray $\{\alpha_i\}_{i\in\N}$ to a biinfinite geodesic $\{\alpha_i\}_{i\in\Z}$ with second endpoint $y$. The sequence $s=\{\rho(\alpha_{i+1}^{-1}\alpha_i)\}_{i\in\Z}$ belongs to $\cal D(C,p,c,\mu,\Z)$. It follows from Propositions \ref{prop:BPS2.5} and \ref{p.bdry} that the sequence $s$ has the dominated splitting $E\oplus F$ where $F_{s}=x^p_\rho$ and $E_s=y^{d-p}_\rho$. So there exist constants $\overline\mu, c_1$ such that
$$\frac{\mm(\rho(\alpha_i^{-1})|_{y^{d-p}_\rho})}{\|\rho(\alpha_i^{-1})|_{x^{p}_\rho} \|}\geq c_1e^{\overline \mu i}.$$

Since, by Lemma \ref{appart}, the angles  $\angle( x^p_\rho,y^{d-p}_\rho)$ and $\angle(\rho(\alpha_i^{-1})x^p_\rho,\rho(\alpha_i^{-1})y^{d-p}_\rho)$ are bounded below by a uniform constant $\nu$, we can apply Corollary \ref{l.expandold} with $P=x^{p}_\rho,$ $Q=y^{d-p}_\rho$ and $g=\rho(\alpha_i)^{-1}$ 
and get
$$d(\rho(\alpha_i^{-1})w^p_\rho,\rho(\alpha_i^{-1})z^p_\rho)\geq b \, \frac{\mm(\rho(\alpha_i^{-1})|_{y^{d-p}_\rho})}{\|\rho(\alpha_i^{-1})|_{x^{p}_\rho} \|} \, d(w^p_\rho,z^p_\rho).$$
\end{proof}

\section{An upper bound on the Hausdorff dimension of the limit set}\label{s.4}

In this section we will prove the following upper bound, this result is independently obtained by Glorieux-Monclair-Tholozan \cite{GMT} for Archimedean $\K.$ Recall from the introduction that if $\rho:\G\to\PGL_d(\K)$ is projective Anosov then $h^{\sroot_1}_\rho$ is the critical exponent of the Dirchlet series $$s\mapsto\sum_{\g\in\G}\left(\frac{\sigma_2}{\sigma_1}\big(\rho(\g)\big)\right)^s.$$ We denote by $\Hff(A)$ the Hausdorff dimension of a subset $A\subset\P(\K^d)$ for the metric induced by a good norm on $\K^d$.

\begin{prop}\label{Hffand1st} Let $\rho:\G\to\PGL_d(\K)$ be $\{\sroot_{1}\}$-Anosov, then $\Hff(\xi^1(\partial\G))\leq h^{\sroot_1}_\rho.$
\end{prop}

Recall that for a metric space $(\L,d)$ and for $s>0$ its  \emph{$s$-capacity} is defined as
\begin{equation}\label{e.capacity}\cal H^s(\L)=\inf_{\eps}\left\{\sum_{U\in\cal U}\diam U^s: \cal U\textrm{ is a covering of $\L$ with }\sup_{U\in\cal U}\diam U<\eps\right\}\end{equation} 
and that 
\begin{equation}\label{HffDef}
\Hff(\L)=\inf\{s:\cal H^s(\L)=0\}=\sup\{s:\cal H^s(\L)=\infty\}.
\end{equation}

In order to prove Proposition \ref{Hffand1st} we will analyze the image, under the bounary map $\xi^1$ of the covering $\cal U_T$ described in subsection \ref{s.covering}, whose elements consist of images of cone types at infinity under sufficiently big group elements. 
The following crucial lemma will allow us to show that the images of the boundaries of cone types transform as expected under group elements: 
\begin{lemma}\label{UniformAngle} 
Let $\rho:\G\to\PGL_d(\K)$ be a projective Anosov representation. Then there exist $\delta>0$, $L\in\N$ such that for all $\g\in\G$ with $|\g|>L$ and every $x\in\bordcone(\g)$ one has 
$$\angle(x^1_\rho,U_{d-1}(\rho(\g^{-1})))>\delta.$$
\end{lemma}

\begin{proof} By definition of $\cone(\gamma),$ for all $x\in\bordcone(\g)$ there exists a geodesic ray $\{\alpha_i\}_0^\infty$ in $\G$ with $\alpha_0=\g^{-1}$ and $\alpha_i\to x$ as $i\to\infty.$ The lemma then follows combining Proposition \ref{weakmorselemma} and Proposition \ref{p.bdry}.
\end{proof}

\subsection{Proof of Proposition \ref{Hffand1st}}

For each $T>0$ consider the covering $\cal U_T$ of $\bord\G$ given by Lemma \ref{SullivanShadows}. By definition, $U=U_\g\in\cal U_T$ is of the form $\g\bordcone(\g)$ for some $\g\in\G$ with $|\g|\geq T.$

Lemma \ref{UniformAngle} implies that there exists $\delta$ such that for every  $x\in \bordcone(\g)$ one has 
$$d\Big(x^1_\rho,U_{d-1}\big(\rho(\g^{-1})\big)\Big)\geq\delta$$
 and thus Lemma \ref{l.dominationattractor} applied to $\rho(\g)$ implies that 
$$d\Big(\rho(\g) x^1_\rho,U_1\big(\rho(\g)\big)\Big)\leq \frac1{\delta}\frac {\sigma_2}{\sigma_1}\big(\rho(\g)\big)$$ 
which implies that 
$$\diam \xi(U_\g)\leq\frac2{\delta}\frac{\sigma_2}{\sigma_1}\big(\rho(\g)\big)\leq Ce^{-\mu T}.$$ 
In particular $\sup_{U\in \cal U_T}\diam U$ is arbitrarily small as $T\to\infty.$ Hence, 
$$\cal H^s\big(\xi^{1}(\bord\G)\big)\leq
\inf_T\sum_{U\in\cal U_T}\big(\diam \xi^{1}(U)\big)^s\leq \left(\frac2{\delta}\right)^s\inf_T\sum_{\g:|\g|\geq T}\left(\frac{\sigma_2}{\sigma_1}\big(\rho(\g)\big)\right)^s.$$

By definition, if $s>h^{\sroot_1}_\rho$ then the Dirichlet series $\Phi^{\sroot_1}_\rho(s)$ is convergent (recall eq. (\ref{DirichletRoot})). Hence, for every $s>h^{\sroot_1}_\rho$ one has $$\lim_{T\to\infty}\sum_{\g:|\g|\geq T}\left(\frac{\sigma_2}{\sigma_1}\big(\rho(\g)\big)\right)^s =0,$$ which implies that the $s$-capacity $\cal H^s(\xi^{1}(\bord\G))$ vanishes and thus $\Hff(\xi^{1}(\bord\G))\leq h^{\sroot_1}_\rho.$

\section{Local conformality and Hausdorff dimension}\label{s.5}
The goal of this section is to find a class of representations for which the equality in Proposition \ref{Hffand1st} holds. This happens in three steps. \begin{itemize}

\item[-] In Section \ref{s.lc1} we study the \emph{thickened cone types} $X_\infty(\alpha),$ these are a thickening, in $\P(\K^d),$ of the image by $\xi^1$ of $\bordcone(\alpha)$ for a given $\alpha,$ and define the \emph{locally conformal points}. 

\item[-] In Section \ref{coarselyball}  we prove that if $x$ is a locally conformal point, then there is a geodesic ray $\alpha_i\to x$ such that the sets $\rho(\alpha_i)X_\infty(\alpha_i)$ behave coarsely like balls around $x;$ the harder inequality is the lower containment, which is achieved in Corollary \ref{c.conetypeBall}. 

\item[-] In Section \ref{Patterson-Sullivan}  we define a measure that behaves like an Ahlfors regular measure for the sets $\alpha X_\infty(\alpha).$ Putting this together with the previous section, arguments coming from Sullivan's original paper allow us to conclude the desired equality, provided we can guarantee existence of many locally conformal points, this is the purpose of Section \ref{s.manyLC}. 
\end{itemize}

\subsection{Thickened cone types at infinity and locally conformal points}\label{s.lc1}
Let $\rho:\G\to\PGL_d(\K)$ be projective Anosov, it follows from Proposition \ref{weakmorselemma} that there is a positive lower bound on the distance of Cartan attractors and repellers of geodesic rays through the origin. Such number will play an important role in our study:

\begin{defi}\label{def:leastangle} Let  $\rho:\G\to\PGL_d(\K)$ be projective Anosov, and let $L$ be fixed and big enough. The \emph{least angle} $\delta_\rho$ is
$$\delta_\rho=\inf \sin\Big(\angle\Big(U_1\big(\rho(\alpha_k)\big),U_{d-1}\big(\rho(\alpha_{-m})\big)\Big)$$  
where $(\alpha_i)_{i\in\Z}$ ranges among biinfinite geodesics through the origin, and $k,m>L$.
\end{defi}
We consider coverings of $\xi(\bord\G)$ obtained by translating  \emph{thickened cone types at infinity}:
$$X_\infty(\alpha):=\mathcal N_{\delta_\rho/2}(\xi^1(\bordcone(\alpha)))\cap\xi^{1}(\bord\G).$$
By construction the sets $X_\infty(\alpha)$ are coarsely balls of $\xi^1(\bord\G)$ centered at points in $\xi^1(\bordcone(\alpha))$:
\begin{obs}
For every $\alpha$ in $\G$, and every $x\in\bordcone(\alpha)$, the thickened cone type  at infinity $X_\infty(\alpha)$ contains a ball centered at the point $x^1_\rho$ of uniform radius:
$$B(x^1_\rho,\delta_\rho/2)\cap\xi^{1}(\bord\G)\subset X_\infty(\alpha)$$
\end{obs}

Thanks to Proposition \ref{weakmorselemma} we can control how thickened cone types shrink under the action of group elements:
\begin{lemma}\label{r.5.4}
Let $\rho:\G\to\PGL_d(\K)$ be projective Anosov. Then there exist $K,L$ such that, for every geodesic ray $(\alpha_i)_{i=0}^\infty$, for every $i>L$, and every $z^1_\rho,w^1_\rho\in X_\infty(\alpha_i)$,
$$d(\alpha_iz^1_\rho,\alpha_iw^1_\rho)\leq K \frac{\sigma_2}{\sigma_1}(\rho(\alpha_i))\;d(z^1_\rho,w^1_\rho).$$ 
\end{lemma}
\begin{proof}
As $\rho$ is projective Anosov we have $d(z^1_\rho,U_{d-1}(\rho(\alpha_i)^{-1}))>\delta_\rho/2$ (Lemma \ref{UniformAngle}). The result is then a direct consequence of Corollary \ref{c.least contraction}.
\end{proof}

\begin{cor} If $\rho:\G\to\PGL_d(\K)$ is projective Anosov, and $x\in\alpha_i\bordcone(\alpha_i)$, then
$$\rho(\alpha_i)X_\infty(\alpha_i)\subset B\left(x^1_\rho,K\frac{\sigma_2}{\sigma_1}(\rho(\alpha_i))\right)\cap\xi(\bord\G).$$
\end{cor}
In particular, if $\{\alpha_i\}_1^\infty$ is a geodesic ray with endpoint $x$, the sets $\rho(\alpha)X_\infty(\alpha)$ form a fundamental system of open neighbourhoods of $x$ in $\xi(\bord\G)$ (cfr. Figure \ref{figure:alphaXalpha}).

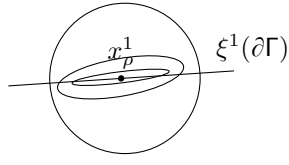
\begin{figure}[hh]
	\centering 
	\begin{tikzpicture}[scale = 0.5]
	
	\filldraw (0,0) circle [radius=2pt] node [above]{$x^1_\rho$};	
\draw (0.1,0) circle [radius= 2];
\begin{scope}[rotate=10];
\draw (-0.01,0.03) ellipse (1.7 and .5);
\end{scope}	
\begin{scope}[rotate=7];
\draw (-0.01,0.03) ellipse (1.3 and .15);
\end{scope}	
\draw (-3,-0.2) to (3,0.2);
 \node at (3.5,0.8) {$\xi^1(\bord\G)$};

	\end{tikzpicture}
	\caption{The sets of the form $\rho(\alpha_i)X_\infty(\alpha_i)$ for a geodesic ray $\{\alpha_i\}$ with endpoint $x$ are the intersections of thinner and thinner ellipses with the limit curve.}\label{figure:alphaXalpha}
\end{figure}
Definition \ref{LocallyConformal} gives conditions guaranteeing that the sets  $\rho(\alpha)X_\infty(\alpha)$ are coarsely balls whose sizes we can precisely estimate.
Given $g\in\GL_d(\K)$ we denote by $$1\leq p_1(g)<\ldots<p_{k(g)}(g)<d$$the indices of the gaps of $g$ (as in Definition \ref{d.gap}).

\begin{defi}\label{LocallyConformal}
Let $\rho:\G\to\PGL_d(\K)$ be projective Anosov. We say that $x\in\bord\G$ is a \emph{$(\epsilon,L)$-locally conformal point} for $\rho$ if there exists a geodesic ray $\{\alpha_i\}_0^\infty$ in $\G$ based at the identity and with endpoint $x$ such that the following conditions hold: 
\begin{enumerate}
\item for all big enough $i$ one has $p_2(\alpha_i)=p_2$ does not depend on $i,$
\item for every $i>L$, and for every $z\in {(\xi^1_\rho)^{-1}}\big(X_\infty(\alpha_i)\big)$ one has 
$$\sin\Big(\angle \big(z^1_\rho\oplus \rho(\alpha_i^{-1})x^1_\rho,U_{d-p_2}\big(\rho(\alpha_i^{-1})\big)\big)\Big)>\epsilon.$$
\end{enumerate}

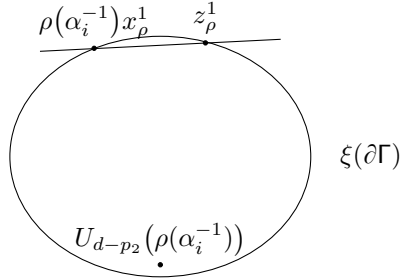
\begin{figure}[hh]
	\centering 
	\begin{tikzpicture}[scale=.8]
	\draw (0,0) ellipse (2.5 and 2);
	\filldraw(0,-1.8) circle [radius=1pt] node[above]{$U_{d-p_2}\big(\rho(\alpha_i^{-1})\big)$};
    \filldraw(.75,1.89) circle [radius=1pt] node[above]{$z^1_\rho$}; 
    \filldraw(-1.1,1.8) circle [radius=1pt] node[above]{$\rho\big(\alpha_i^{-1}\big)x^1_\rho$};
\draw (-2, 1.75)to (2, 1.95);  
\node at (3.5,0) {$\xi(\bord\G)$};  
	\end{tikzpicture}
	\caption{The second condition in Definition \ref{LocallyConformal}.}\label{figure:locally conformal}
\end{figure}
\end{defi}
Note that, in general, the index $p_2$ might depend on  the point $x$ and we do not require that the representation $\rho$ is $\{\sroot_{p_2}\}$-Anosov. In the special case when $\rho$ is $\{\sroot_{1},\sroot_{2}\}$-Anosov, the condition $(i)$ is automatically satisfied with $p_2=2$, but $(ii)$ can only hold if the dimension of $\bord\G$ is very small (cfr. Corollary \ref{r.obstruction}).
\begin{remark}
A generic element $g\in\PGL(V)$ has $p_2(g)=2$. Nevertheless there are many interesting geometric situations in which condition $(i)$ holds for $p_2>2$. For example if $g$ is a generic element in $\SO(m,n)$, we have that $\Wedge^mg\in\SL(V)$ has $p_2(\Wedge^mg)=n-m+1$, so one can enforce $p_2>2$ by considering representations in smaller subgroups. In Section \ref{sec:POK} we will describe another interesting class of examples.
\end{remark}

\subsection{Neighborhoods of locally conformal points that are coarsely balls}\label{coarselyball}
We will now show that if $x$ is a locally conformal point for $\rho,$ and $\alpha_i\to x$ is a geodesic ray, then the sets $\rho(\alpha_i)X_\infty(\alpha_i)$  are coarsely balls  centered at $x^1_\rho$ 
of radius 
${\sigma_2}/{\sigma_1}(\rho(\alpha_i))$
 for the distance on $\xi^1_\rho(\bord\G)$ induced by $d,$ this will be achieved in Corollary \ref{c.conetypeBall}, and motivated the terminology \emph{locally conformal}.

\begin{prop} \label{p.conetypeBall}
Let $\rho:\G\to\PGL_d(\K)$ be projective Anosov.  There exist $ \overline L$ such that, for every $(\epsilon, L)$-locally conformal point $x$, there exists a geodesic ray $\{\alpha_i\}_0^\infty$  from the identity with endpoint $x$, such that for every $i>\overline L$ and every $z\in X_\infty(\alpha_i)$ it holds
$$\frac{\epsilon}4\frac{\sigma_2}{\sigma_1}\big(\rho(\alpha_i)\big)\cdot d\big(z^1_\rho,(\alpha_i^{-1}x)^1_\rho\big)\leq d\big((\alpha_iz)^1_\rho,x^1_\rho\big).$$
\end{prop}

\begin{proof} 
Let $W_i:=z^1_\rho\oplus \big(\alpha_i^{-1}x\big)^1_\rho$. As $x$ is $(\epsilon,L)$-locally conformal, for every $i>L$, we have $d (W_i,U_{d-p_2}(\rho(\alpha_i^{-1})))>\epsilon.$ {From Lemma \ref{l.dominationattractor} one concludes that $$\angle\Big(\rho(\alpha_i)W_i,U_{p_2}\big(\rho(\alpha_i)\big)\Big)\to0$$
 as $i\to\infty$ at a speed only depending on $\eps$ and the Anosov constants of $\rho,$ and thus, possibly increasing $L,$ one concludes that }for every $i>L$ it holds 
$$\frac{\sigma_2}{\sigma_1}\big(\rho(\alpha_i)|_{W_i}\big)\geq \epsilon\frac{\sigma_{p_2}}{\sigma_1}\big(\rho(\alpha_i)\big)=\epsilon\frac{\sigma_2}{\sigma_1}\big(\rho(\alpha_i)\big).$$
Here the last equality is due to the fact that $p_2$ is the first gap for $\rho(\alpha_i)$ and thus $\sigma_{p_2}\big(\rho(\alpha_i)\big)=\sigma_2\big(\rho(\alpha_i)\big)$.

Furthermore $\rho\big(\alpha_i^{-1}\big)u_2\big(\rho(\alpha_i)|_{W_i}\big)\in W_i\cap U_{d-1}\big(\rho(\alpha_i^{-1})\big)$, and then, since $\rho$ is projective Anosov and $z^1_\rho\in X_\infty(\alpha)$, we have $d \big(z^1_\rho,\rho\big(\alpha_i^{-1}\big)u_2\big(\rho(\alpha_i)|_{W_i}\big)\big)>\delta_\rho/2$, where $\delta_\rho$ is the constant from Definition \ref{def:leastangle}. This implies that we can find $\overline L$ depending on $\rho$ and $\epsilon$ only such that for every $i>\overline L$
$$d\Big(\rho(\alpha_i)z^1_\rho, U_1\big(\rho(\alpha_i)|_{W_i}\big)\Big)<d\Big(\rho(\alpha_i)z^1_\rho, U_1\big(\rho(\alpha_i)\big)\Big)+d\Big(U_1\big(\rho(\alpha_i)\big), U_1\big(\rho(\alpha_i)|_{W_i}\big)\Big)<1/\sqrt 2,$$ {since both quantities converge to $0$ as $i\to \infty$ at a speed only depending on the Anosov constants of $\rho.$} The proposition then follows from Corollary \ref{c.expand}.
\end{proof}

Recall from Definition \ref{d.nested} that we say that a pair of cone types $(\cone(\alpha_1),\cone(\alpha_2))$ of $\G$ are $k$-nested if there exists a path in the geodesic automaton of lenght $k$ between $\cone(\alpha_1)$ and $\cone(\alpha_2)$. In this case we say that $\beta\in\G$ is a \emph{nesting word} if $\beta$ labels one such path.
\begin{lemma}\label{p.lowerconetype}
For every $\overline L$ big enough (depending only on $\rho$) there exists a constant $c$ (depending on $\rho$ and $\overline L$) such that for every $\overline L$-nested pair $(\cone(\alpha_1),\cone(\alpha_2))$ and any nesting word $\beta$ it holds 
\begin{enumerate}
\item $\rho(\beta)X_\infty(\alpha_2)\subset X_\infty(\alpha_1)$
\item for every $z^1_\rho\in\xi(\beta\bordcone(\alpha_2))$ and every  $w^1_\rho\in X_\infty(\alpha_1)\setminus\rho(\beta)X_\infty(\alpha_2)$, it holds 
$$d(z^1_\rho,w^1_\rho)>c.$$
\end{enumerate}
\end{lemma}

\begin{proof}\item\begin{itemize}\item[(i)]
{By definition of $\delta_\rho$ and $X_\infty(\alpha_2)$, whenever $|\beta|\geq L$ and $\beta$ is a nesting word, then $d(x, U_{d-1}\big(\rho(\beta^{-1})\big))\geq \delta_\rho/2$ for every point $x$ in $X_\infty(\alpha_2)$. Here $L$ is as in Definition \ref{def:leastangle}.}
Up to possibly enlarging $L$ we can assume, by Corollary \ref{c.least contraction}, that $\rho(\beta)$ contracts distances on $X_\infty(\alpha_2)$ so that 
$$\rho(\beta)X_\infty(\alpha_2)\subseteq \mathcal N_{\delta_\rho/2}(\rho(\beta)\xi^1_\rho(\bordcone(\alpha_2)))\cap\xi(\bord\G)\subseteq X_\infty(\alpha_1).$$

\item[(ii)] Since, by construction, $X_\infty(\alpha_2)$ contains the intersection of $\xi^1_\rho(\bord\G)$ with a ball around any point $z^1_\rho\in\xi^1_\rho(\bordcone(\alpha_2))$ of radius $\delta_\rho/2$, the set $\rho(\beta)X_\infty(\alpha_2)$ contains the intersection of $\xi^1_\rho(\bord\G)$ with the ball around any point $z^1_\rho\in\rho(\beta)\xi^1_\rho(\bordcone(\alpha_2))$ of radius 
$\frac{\delta_\rho}2\frac{\sigma_d}{\sigma_1}\big(\rho(\beta)\big):$
 $\frac{\sigma_d}{\sigma_1}(g)$ is the smallest contraction for the action of $g\in\SL(d,\K)$ on $\P(\K^d)$.  
{Recall that only finitely many $\beta$ can occur, as, by construction, $|\beta|=\overline L$. The result follows taking $$c=\min_{|\beta|=\overline L}\frac{\delta_\rho}2\frac{\sigma_d}{\sigma_1}\big(\rho(\beta)\big)$$}
\end{itemize}
\end{proof}

Combining Proposition \ref{p.conetypeBall} and Lemma \ref{p.lowerconetype} we obtain:

\begin{prop}\label{2p.conetypeBall}
There exists $c_1$ depending only on $\rho$ such that, if $L$ is as in Lemma \ref{p.lowerconetype} and  $\{\alpha_i\}\subset\G$ is a geodesic ray with endpoint $x$, for every $y$ with $y^1_\rho\in\rho(\alpha_n)X_\infty(\alpha_n)\setminus \rho(\alpha_{n+L})X_\infty(\alpha_{n+L})$, it holds 
$$d(y^1_\rho,x^1_\rho)\geq c_1\frac{\sigma_2}{\sigma_1}\big(\rho(\alpha_{n+L})\big).$$
\end{prop}

\begin{proof}
It follows from Lemma \ref{l.nested} that for every $n,L$  the pair $(\bordcone(\alpha_n),\bordcone(\alpha_{n+L}))$ is $L$-nested. Furthermore, up to choosing $L$ large enough, we can apply Lemma \ref{p.lowerconetype} to the pair  $(\bordcone(\alpha_n),\bordcone(\alpha_{n+L}))$. If we denote by $z:=\alpha_n^{-1}x$ and $w:=\alpha_n^{-1}y$ we deduce that $d(z^1_\rho, w^1_\rho)>c$. Proposition \ref{p.conetypeBall}
 implies then that  
$$d(y^1_\rho,x^1_\rho)\geq \frac{c\epsilon}{4}\frac{\sigma_2}{\sigma_1}(\alpha_{n})\geq c_1\frac{\sigma_2}{\sigma_1}(\alpha_{n+L})$$
Where in the last inequality we used that, as $L$ is fixed the homothecy ratio gap of $\alpha_n$ is uniformly comparable to the one of $\alpha_{n+L}$.
\end{proof}
As a corollary of Proposition \ref{p.conetypeBall} we can finally get the main result of the section (cfr. Figure \ref{figure:alphaXalpha}):

\begin{cor}\label{c.conetypeBall} 
Let $\rho:\G\to\PGL_d(\K)$ be a projective Anosov; then for every locally conformal point $x\in\bord\G$ there exists a geodesic ray $\alpha_i\to x$ with 
$$B\left(x^1_\rho,c_1\frac{\sigma_2}{\sigma_1}\big(\rho(\alpha_i)\big)\right)\cap\xi(\partial\G)\subset\rho(\alpha_i)X_\infty(\alpha_i).$$ 
\end{cor}

\begin{proof}
This follows from the above proposition by observing that the sets $\rho(\alpha_i)X_\infty(\alpha_i)$ form a fundamental system of neighborhoods of $x^1_\rho$ in $\xi(\bord\G)$.
\end{proof}

\subsection{A regular measure for conformal points}\label{Patterson-Sullivan} The goal of this section is to construct, following Patterson's original idea, a measure, supported on $\xi^1_\rho(\bord \G)$, for which we can get good estimates on the measure of the cone types. This will be used in Section \ref{s.manyLC} to obtain the desired lower bound on the Hausdorff dimension of the limit set\footnote{See Remark \ref{r.Quint} for a comparison with the work of Quint \cite{quint1}}.

Let $\rho:\G\to\PGL_d(\K)$ be a projective Anosov representation. {Recall from the introduction that we have defined $$\Phi_\rho^{\sroot_1}(s)=\sum_{\g\in\G}\left(\frac{\sigma_2}{\sigma_1}\big(\rho(\g)\big)\right)^{s}.$$}We can assume that 
$\Phi_\rho^{\sroot_1}(h^{\sroot_1}_\rho)=\infty:$
otherwise, as it is standard in Patterson-Sullivan theory, we would carry out the same construction with the aid of the modified Poincar\'e series
$$\Phi_\rho^{\sroot_1}(s)=\sum_{\g\in\G} f\Big(\sroot_1\big(\rho(\g)\big)\Big)\left(\frac{\sigma_2}{\sigma_1}(\rho(\g))\right)^{h^{\sroot_1}_\rho},$$
where $f(s)$ is the function constructed  (for example) in Quint \cite[Lemma 8.5]{quint1}.

We will therefore assume from now on that the Poincar\'e series diverges at its critical exponent; for every $s>h^{\sroot_1}_\rho$, we define $$\mu^s_\rho=\frac1{\Phi_\rho^{\sroot_1}(s)}\sum_{\g\in\G}\left(\frac{\sigma_2}{\sigma_1}\big(\rho(\g)\big)\right)^s\delta_{U_1(\rho(\g))}.$$ 
Recall from Section \ref{s.2.1} that, for every element $\g\in\G$ we chose a Cartan decomposition of $\rho(\g)$ and therefore a 1-dimensional subspace $U_1\big(\rho(\g)\big).$

One easily checks that for every $s>h^{\sroot_1}_\rho$ the functional $f\mapsto\int fd\mu^s_\rho$ is continuous on $C(\P(\K^d),\R)$ with the uniform topology and hence one can take a weak* accumulation point of $\mu^s_\rho,$ as $s\to h^{\sroot_1}_\rho,$ in the space of Radon probability measures on $\P(\K^d).$ We will denote such Radon measure by $\mu^{\sroot_1}_\rho,$ (note that we do not show, nor require, that $\mu^{\sroot_1}_\rho$ is the only accumulation point of $\mu^s_\rho$).

\begin{lemma}\label{hU(g)} For any $\eta\in\G$  the (signed) measure
$$\epsilon(\eta,s):=\eta_*\mu^s_\rho-\frac1{\Phi_\rho^{\sroot_1}(s)}\sum_{\g\in\G}\left(\frac{\sigma_2}{\sigma_1}(\rho(\g))\right)^s\delta_{U_1(\rho(\eta\g))}$$ weakly* converges to zero as $s\to h^{\sroot_1}_\rho.$
\end{lemma}
\begin{proof} 
Indeed by definition 
$$\eta_*\mu^s_\rho=\frac1{\Phi_\rho^{\sroot_1}(s)}\sum_{\g\in\G}\left(\frac{\sigma_2}{\sigma_1}(\rho(\g))\right)^s\delta_{\rho(\eta)U_1(\rho(\g))}.$$
Furthermore, Lemma \ref{l.qg} (\ref{l.qg2}) implies that 
$$d(\rho(\eta)U_1(\rho(\g)),U_1(\rho(\eta\g)))\leq \|\eta\|\|\eta^{-1}\|\frac{\sigma_2}{\sigma_1}\rho(\g).$$
In order to show that $\epsilon(\eta,s)$ converges to zero, it is enough to show that for every continuous function $f:\P(V)\to\R$ the integral of $f$ on $\epsilon(\eta,s)$ tends to zero as $s$ converges to $h^{\sroot_1}_\rho$. However  every such function $f$ is uniformly continuous, and therefore for every $\epsilon$ we can find $\delta$ such that $|f(x)-f(y)|<\epsilon/2$ if $d(x,y)<\delta$. It is then enough to choose $s$ close enough to  $ h^{\sroot_1}_\rho$ so that the mass of $\mu^s_\rho$ of the elements $\g$ such that  $\sigma_1/\sigma_2(\g)<\|\eta\|\|\eta^{-1}\|/\delta$ is smaller than $\frac{\epsilon}{2\|f\|}$.
\end{proof}

One has the following proposition (compare with Sullivan's shadow Lemma \cite{sullivan}).

\begin{prop}\label{measure.conetype1} Let $\rho:\G\to\PGL_d(\K)$ be a projective Anosov representation, then for all $\eta\in \G$ one has 
$$\left(\frac{\sigma_d}{\sigma_1}(\rho(\eta))\right)^{h_\rho^{\sroot_1}}\leq\frac{\mu^{\sroot_1}_\rho(\rho(\eta) X_\infty(\eta))}{\mu^{\sroot_1}_\rho(X_\infty(\eta))}\leq \frac4{\delta_\rho^2}\left(\frac{\sigma_2}{\sigma_1}(\rho(\eta))\right)^{h_\rho^{\sroot_1}}.$$
\end{prop}

Recall that there are finitely many cone types, so the number $\mu^{\sroot_1}_\rho(X_\infty(\eta))$ is an irrelevant constant.

\begin{proof}
Consider $s>h^\sigma_\rho,$ $\eta\in\G$ and a continuous function $f:\P(\K^d)\to\R.$ One has \begin{alignat}{2} 
\mu^s_\rho(f\circ\rho(\eta)^{-1} )& = \frac1{\Phi_\rho^{\sroot_1}(s)}\sum_{\g\in\G}\left(\frac{\sigma_2}{\sigma_1}(\rho(\g))\right)^sf\big(\rho(\eta^{-1})U_1(\rho(\g))\big)\nonumber\\ 
& =\eps(\eta^{-1},s)(f)+ \frac1{\Phi_\rho{\sroot_1}(s)}\sum_{\g\in\G}\left(\frac{\sigma_2}{\sigma_1}(\rho(\eta\g))\right)^sf(U_1(\rho(\g))) \label{tutti} \\ 
& =\eps(\eta^{-1},s)(f)+ \frac1{\Phi_\rho^{\sroot_1}(s)}\sum_{\g\in\G}\left(\frac{\sigma_2}{\sigma_1}(\rho(\eta\g))\frac{\sigma_1}{\sigma_2}(\rho(\g))\right)^s\left(\frac{\sigma_2}{\sigma_1}(\rho(\g))\right)^sf(U_1(\rho(\g))), \nonumber 
\end{alignat}
where $\epsilon(\eta^{-1},s)$ is the term estimated in Lemma \ref{hU(g)}, so that $\eps(\eta^{-1},s)(f)$ converges to zero when $s\to h^{\sroot_1}_\rho.$

Assume that the support of $f$ contains $X_\infty(\eta)$ in its interior and $s$ is close enough to $h^{\sroot_1}_\rho$ {so that $\Phi_\rho^{\sroot_1}(s)$ is arbitrary large}. Then only {the tail of the sum involved in $\mu^s_\rho(f\circ\rho(\eta)^{-1} )$ is relevant, this is to say: \begin{itemize}\item[-] only $\g$'s for which $|\g|$ is large matter, \item[-] since we are integrating $f,$ $U_1(\rho(\g))$ has to be near to $X_\infty(\eta),$ so that there is a geodesic segment from $\eta^{-1}$ to $\g$ passing through the identity.\end{itemize}}\noindent This, together with Proposition \ref{weakmorselemma}, implies that such $\g$'s one has $$\sin\big(\angle(U_1(\rho(\g)),U_{d-1}(\rho(\eta)))\big)>\eps_f,$$ for some $\eps_f$ depending on the support of $f.$ Note that $\eps_f$ approaches $\delta_\rho/2$ {(recall Definition \ref{def:leastangle})} as $\supp f\to X_{\infty}(\eta).$ Choosing a sequence $s_k\to h^{\sroot_1}_\rho$ such that $\mu^{s_k}_\rho\to\mu^{\sroot_1}_\rho$ one has, using Lemma \ref{l.no_cancellation} and equation (\ref{tutti}), that, for any such $f$
$$\mu^{\sroot_1}_\rho(f\circ\rho(\eta)^{-1})=\lim_{s_k\to h^\sigma_\rho}\mu^{s_k}_\rho(f\circ\rho(\eta)^{-1})\leq\frac4{\delta_\rho^2} \left(\frac{\sigma_2}{\sigma_1}(\rho(\eta))\right)^{h^{\sroot_1}_\rho}\mu^{\sroot_1}_\rho(f).$$
By continuity of $f\mapsto \mu^{\sroot_1}_\rho(f),$ one concludes the desired upper bound. The lower bound follows similarly. 
\end{proof}

Since cone types shrink to any given point of $\partial\G$ one has the following consequences of Proposition \ref{measure.conetype1}.

\begin{cor} The measure $\mu^{\sroot_1}_\rho$ has total support and no atoms.  
\end{cor}

\begin{proof} If $\alpha_i$ is a geodesic ray converging to $x$ then $\rho(\alpha_i)X_\infty(\alpha_i)$ is a family of open neighborhoods decreasing to $x,$ and since $\rho$ is projective Anosov one has $(\sigma_2/\sigma_1)(\rho(\alpha_i))\to0$ as $i\to\infty.$ As $\mu_\rho^{\sroot_1}$ is a Radon measure we have on the one hand 
$$\mu_\rho^{\sroot_1}(\{x\})=\inf\{\mu_\rho^{\sroot_1}(\rho(\alpha_i)X_\infty(\alpha_i))\}\leq\frac4{\delta_\rho^2}\left(\frac{\sigma_2}{\sigma_1}(\rho(\eta))\right)^{h_\rho^{\sroot_1}},$$
on the other hand for every open set $A$ intersecting $\xi(\bord\G)$ we can find $\alpha$ such that $\rho(\alpha_i)X_\infty(\alpha_i)$ is contained in  $A$, and thus $$\mu_\rho^{\sroot_1}(A)\geq \mu_\rho^{\sroot_1}(\rho(\alpha_i)X_\infty(\alpha_i)) )\geq\left(\frac{\sigma_d}{\sigma_1}(\rho(\eta))\right)^{h_\rho^{\sroot_1}}.$$\end{proof}

\subsection{When conformal points are abundant}\label{s.manyLC}

Denote by $$\LC(\rho)=\{x\in\bord\G:x \textrm{ is locally conformal for }\rho \}.$$ We can now prove the following.

\begin{thm}\label{thm:eqifLC}Let $\rho:\G\to\PGL_d(\K)$ be a projective Anosov representation. If $\mu^{\sroot_1}_\rho(\LC(\rho))>0,$ then $$\Hff(\xi^1_\rho(\bord\G))=h^{\sroot_1}_\rho.$$
\end{thm}

\begin{proof} 
As we already established in Proposition \ref{Hffand1st}, $\Hff(\xi^1_\rho(\bord\G))\leq h^{\sroot_1}_\rho,$ so we only need to show the reverse inequality. The proof will follow the main ideas in Sullivan's original work \cite{sullivan}, using Corollary \ref{c.conetypeBall} and Proposition \ref{measure.conetype1} as key replacement for the conformality of a Kleinian group action on its boundary, and Sullivan's shadow lemma. 

Given $x\in\LC(\rho)$ and a geodesic ray $\{\alpha_i\}$ on $\G$ converging to $x,$ Corollary \ref{c.conetypeBall} implies that for all $i\geq N_0(x)$ the set $\rho(\alpha_i)X_\infty(\alpha_i)$ is coarsely (with constants independent of $x$) a ball of radius $$r_i(x)=\frac{\sigma_2}{\sigma_1} \big(\rho(\alpha_i)\big)$$ about $x^1_\rho$ (for the induced metric on $\xi_\rho^1(\bord\G)$).

Proposition \ref{measure.conetype1} then states that for all {$i\geq N_0(x)$}
\begin{equation}\label{aux}
\mu^{\sroot_1}_\rho\big(B(x,r_i(x))\big)\leq  c r_i(x)^{h^{\sroot_1}_\rho}.\end{equation} 

Observe that we can extend equation (\ref{aux}) for any $0<r\leq r_{N_0}(x)$, up to possibly worsening the constant $c$: Since $\rho$ is projective Anosov, the word length of $\g\in\G$ is coarsely $\log\sigma_2/\sigma_1 (\rho(\g)),$ thus $$r_i(x)/r_{i+1}(x)\leq K_\rho$$ for some constant $K_\rho$ only depending on $\rho;$ given $r$ it suffices to consider $r_{i+1}(x)\leq r\leq r_i(x)$ and thus 
$$\mu^{\sroot_1}_\rho(B(x,r))\leq c\Big(\frac{r_i(x)}{r_{i+1}(x)}\Big)^{h^{\sroot_1}_\rho}r_{i+1}(x)^{h^{\sroot_1}_\rho}\leq L_\rho r^{h^{\sroot_1}_\rho}. $$ 

 Furthermore, there exists $\eps$ such that the set $X_\epsilon=\{x\in\LC(\rho): r_{N_0}(x)\geq\eps\}$ has positive $\mu^{\sroot_1}_\rho$-mass: this follows from the general fact that countable union of sets with measure 0 has measure 0, since we assumed $\mu^{\sroot_1}_\rho(\LC(\rho))>0,$ 

The remainder arguments are verbatim as in Ha\"issinsky \cite[Th\'eor\`eme F.4]{CoursH}. We include them for completeness: as $X_\epsilon$ is a subset of $\xi(\bord\G)$, it is enough to verify that $\Hff(\xi(X_\epsilon))\geq h^{\sroot_1}_\rho$; we will show that, denoting by $\sigma:=h^{\sroot_1}_\rho$, we have $\cal H^\sigma(X_\epsilon)>0$. 
Indeed let us denote by $d:=\frac{\mu_\rho^{\sroot_1}(X_\epsilon)}{2L_\rho}$. By definition of $\sigma$-capacity we can find an open covering $\cal B=\{B(x_i,r_i)\}$ of $X_\epsilon$ consisting of balls of radius $r_i<\epsilon$ and such that 
$$\sum r_i^\sigma\leq \cal H^\sigma(X_\epsilon)+d.$$
Recall from (\ref{e.capacity}) at the beginning of Section \ref{s.4} that we denote by  $\cal H^\sigma(X_\epsilon)$ the $\sigma$-capacity of the set $X_\epsilon$.
On the other hand we have
$$\mu^{\sroot_1}_\rho(X_\epsilon)\leq \mu^{\sroot_1}_\rho\Big(\bigcup_i B(x_i,r_i)\Big)\leq \sum_i  \mu^{\sroot_1}_\rho\big(B(x_i,r_i)\big)\leq \sum_i L_\rho r_i^\sigma.$$
This shows that $\cal H^s(X_\epsilon)$ is positive, and concludes the proof. \end{proof}

\begin{remark}\label{r.Quint}
Patterson-Sullivan measures in a setup close to ours were extensively studied by Quint \cite{quint1}. For our geometric applications it is crucial to have an Ahlfors regular measure of exponent $h^{\sroot_1}_\rho$. Let us denote by $G$ the Zariski closure of $\rho(\G)$,  assume that $G$ is reductive (despite this is not always the case in the examples we have in mind), and let $\cal F_G$ denote the full flag space associated to $G$.  Quint \cite[Theorem 8.4]{quint1} provides a quasi-invariant measure $\mu$ on  $\cal F_G$ called a  $(\rho(\G),h^{\sroot_1}_\rho\sroot_1)$-Patterson-Sullivan, with the desired transformation rule, as long as a tecnical condition is satisfied, namely that the form $h^{\sroot_1}_\rho\sroot_1$ is tangent to the growth indicator function $\psi_{\rho(\G)}$. In order to guarantee that this is the case we would have to further assume that the representation $\rho$ is $\{\sroot_p\}$-Anosov, and that $p_2(\rho(\g))=p$ for every $\g\in\G.$ The measure $\mu$ could then be pushed forward via the projection $\cal F_G\to\P(\K^d)$ and the fact that $\rho$ is $\{\sroot_1,\sroot_p\}$-Anosov would imply that the new measure on $\P(\K^d)$ would still be quasi-invariant. However deducing the analogue of Proposition \ref{measure.conetype1} in that setting would require some work as our representations are, in most interesting cases, not Zariski dense. \end{remark}

\section{$(p,q,r)$-hyperconvexity}\label{section:FrenetReps}
In this section we introduce $(p,q,r)$-hyperconvex representations, establish geometric properties and provide the link with local conformality.
\subsection{Hyperconvex representations}
The following definition is inspired from Labourie \cite{Labourie-Anosov} for surface groups. Let $\G$ be a word-hyperbolic group and denote by $$\bord^{(3)}\G=\{(x,y,z)\in(\bord\G)^3:\textrm{ pairwise distinct}\}.$$

\begin{defi} Consider $p,q,r\in\lb1,d-1\rb$ such that $p+q\leq d.$  We say that a representation $\rho:\G\to\PGL_d(\K)$ is \emph{$(p,q,r)$-hyperconvex} if it is $\{\sroot_p,\sroot_q,\sroot_r\}$-Anosov and for every triple $(x,y,z) \in\bord^{(3)}\G$ one has$$( x^p_\rho\oplus y^q_\rho)\cap z^{d-r}_\rho=\{0\}.$$
\end{defi}

Note that, since $p+q<d$ and the representation is $\{\sroot_p,\sroot_q\}$-Anosov, the sum $x^p_\rho+ y^q_\rho$ is necessarily direct. Hence, hyperconvexity implies that $p+q\leq r.$  We will observe in Corollary \ref{r.obstruction} that $\rho$ can only be $(p,q,r)$-hyperconvex if $r-p-q\geq \dim(\bord\G)-1$. Note that we do not require $p$ and $q$ to be different.

\begin{prop}\label{FrenetOpen} The space of $(p,q,r)$-hyperconvex representations is open in $$\hom(\G,\PGL_d(\K)).$$\end{prop}

 \begin{proof} The proof follows the same lines as Labourie \cite[Proposition 8.2]{Labourie-Anosov}. Since the action of $\G$ on $\bord^{(3)}\G$ is properly discontinuous and co-compact, given a triple $(x,y,z)\in\bord^{(3)}\G$ there exists $\g\in\G$ such that the the points $\g x,\g y$ and $\g z$ are pairwise far apart. Considering a $(p,q,r)$-hyperconvex representation $\rho,$ one concludes that the angles between any pair of the spaces $(\g x)^p_\rho,(\g y)^q_\rho$ and $(\g z)^{d-r}_\rho$ are bounded away from zero. Corollary \ref{bdry.continuous} states that the Anosov condition is open and that equivariant maps vary continuously with the representation, hence, since the map $\bord^{(2)}\G\to\Gr_{p+q}(\K^d)$ $$(a,b)\mapsto a^p_\rho\oplus b^q_\rho$$ is continuous away from the diagonal the result follows.
\end{proof}

Since hyperconvexity is an open property, one can provide interesting examples of hyperconvex representations by looking at representations of the form $\G\to G\to\GL_d(\K),$ where the first arrow is convex co-compact (see Section \ref{repdef}) furthermore hyperconvexity behaves well with field extensions:
\begin{lemma}\label{FieldExt}Let $\K\subset\F$ be a field extension, if $\rho:\G\to\PGL_d(\K)$ is $(p,q,r)$-hyperconvex then so is $\rho:\G\to\PGL_d(\F).$
\end{lemma}

We conclude the subsection providing obstructions to the existence of  $(1,1,r)$-hyperconvex representations. A useful tool for this is the stereographic projection:
\begin{defi}\label{stereo}
Let $\rho:\G\to\PGL_d(\K)$ be $\{\sroot_1,\sroot_r\}$-Anosov.  Given $z\in\bord\G,$ the \emph{stereographic projection} defined by $z$ (and $\rho$) is the continuous map 
$$\stp_{z,\rho}:\bord\G-\{z\}\to\P(\K^d/z^{d-r}_\rho)$$
defined as follows: since $\rho$ is $\{\sroot_1\}$-Anosov, for every point $x\in\bord\G$ different from $z$, the vector space $x^1_\rho\oplus z^{d-r}_\rho$ has dimension $d-r+1$ and projects to a line in the quotient space $\K^d/z^{d-r}_\rho;$ we define $\stp_{z,\rho}(x)\in\P(\K^d/z^{d-r}_\rho)$ to be the projectivisation of this line.
\end{defi}
The following is immediate from the definitions:
\begin{lemma}
If the representation $\rho$ is $(1,1,r)$-hyperconvex then for every $z\in\bord\G$ the map $\stp_{z,\rho}$ is continuous and injective. 
\end{lemma}
\begin{proof}
The stereographic projection $\pi_{z,\rho}$ is the composition of the boundary map $\xi^1:\bord\G\setminus \{z\}\to \P(\K^d)$ with the  projection $\P(\K^d\setminus z_\rho^{d-r})\to \P(\K^d/ z_\rho^{d-r})$, which is algebraic outside the $(d-r)$-dimensional subspace $z_\rho^{d-r}$; it is well defined as $\rho$ is $\sroot_1$-Anosov, and is therefore continuous.  Injectivity follows directly from the definition of hyperconvexity.
\end{proof}
\begin{cor}\label{r.obstruction}
If there is no continuous injective map $\bord\G-\{z\}\to \P(\K^r)$, then there is no $(1,1,r)$-hyperconvex representation $\rho:\G\to\PGL_d(\K)$.
\end{cor}

\subsection{From hyperconvexity to local conformality}
We now find a link between hyperconvexity and local conformality.  The following statement is the main technical result of Section \ref{section:FrenetReps}, and will be crucial in the proof of Theorem \ref{FRep->Fpoint}.

Recall from Section \ref{s.lc1} that we defined, for every projective Anosov representation $\rho:\G\to\PGL_d(\K)$, the thickened cone type at infinity $X_\infty(\alpha)$ as the intersection of the $\delta_\rho/2$-neighbourhood of $\xi^1_\rho(\bordcone(\alpha))$ with the image of the boundary map. In a similar way, if $\rho$ is $\{\sroot_p\}$-Anosov, we set   
$$X_\infty^p(\alpha):=\cal N_{{\delta_{p,\rho}}/2}\xi^p_\rho\big(\bordcone(\alpha)\big)\cap\xi^p_\rho(\bord\G),$$
where ${\delta_{p,\rho}}$ is the number $\delta$ from Proposition \ref{weakmorselemma}.
\begin{prop}\label{p.main}
Let $\rho:\G\to\PGL_d(\K)$ be $(p,q,r)$-hyperconvex. Then there exist constants $L,\epsilon$ such that for every $\alpha\in\G$ with $|\alpha|>L$, for every $x\in\bordcone(\alpha)$ and every $y\in(\xi^q_\rho)^{-1}X^q_\infty(\alpha)$, it holds
\begin{equation}\label{e.lc}
\sin\angle \Big(x^p_\rho\oplus y^q_\rho,U_{d-r}\big(\rho(\alpha^{-1})\big)\Big)>\epsilon.
\end{equation}
\end{prop}

{Observe that the conclusion of the proposition is the second condition required for a locally conformal point (Definition \ref{LocallyConformal}).}

{Before proving the Proposition let us fix a distance $d$ on $\bord\G$ inducing its topology and for $\nu>0$ define a triple of points $x,y,z\in\bord\G$ is $\nu$-\emph{separated}, if all distances $d(x,y),$ $d(y,z)$ and $d(x,z)$ are bounded below by $\nu.$ The following lemma follows from the convergence property of hyperbolic groups, see for example Tukia \cite{tukia}.}

{
\begin{lemma}\label{separated} Let $(\alpha_i)_{i\in\Z}$ be a bi-infinite geodesic through $e\in\G$ with $\alpha_i\to x$ and $\alpha_{-i}\to z$ say, as $i\to+\infty.$ Then the function $y\mapsto d(\alpha_i^{-1}y,\alpha_i^{-1}z)$ converges to 0 uniformly on compact sets of $\bord\G-\{x\}$ is $i\to+\infty.$ Consequently, for fixed $\nu,$ the positive integers $n$ such that the triple $\alpha_n^{-1}x,\,\alpha_n^{-1}y,\,\alpha_n^{-1}z$ is $\nu$-separated is bounded above uniformly on compact sets of $\bord\G-\{x\}.$ Finally, there exists $\nu_0>0$ such that for every $0<\eps<\nu_0$ and $y\in\bord\G-\{x\}$ with $d(x,y)<\eps$ there exists $n\in\N$ such that $\alpha_n^{-1}x,\,\alpha_n^{-1}y,\,\alpha_n^{-1}z$ are $\nu_0$-separated.
\end{lemma}}

{\begin{proof} Let us give an idea of the proof in our situation, i.e. assuming that $\G$ admits a projective Anosov representation $\rho.$ We focus on finding $\nu_0$ and $n\in\N$ so that the last sentence of the statement holds. 

Consider the distance $d$ induced by our chosen distance on $\P(\K^d)$ through the boundary map $\xi^1_\rho.$ The fact that there is a lower bound on the values  $d(\alpha_n^{-1}x,\alpha_n^{-1}z)$ for all $n$ follows from Lemma \ref{appart}, and the fact that we can find a suitable $n$, such that both $d(\alpha_n^{-1}z,\alpha_n^{-1}y)>\nu_0$ and $d(\alpha_n^{-1}x,\alpha_n^{-1}y)>\nu_0$ is a consequence of Lemma \ref{expand2} combined with the fact that the action of the images of the generators on $\P(\K^d)$ is uniformly Lipschitz.
\end{proof}
}

\begin{proof}[Proof of Proposition \ref{p.main}] Since the representation $\rho:\G\to \PGL_d(\K)$ is $(p,q,r)$-hyperconvex we can find $\epsilon_0$ such that if $s,w,t\in\bord\G$ are $\nu_0$-separated one has 
\begin{equation}\label{hyper}\sin\angle(s^p_\rho\oplus t^q_\rho,w^{d-r}_\rho)>\epsilon_0:\end{equation}
this is guaranteed since the set of $\nu_0$-separated triples is precompact as the group is hyperbolic.

Let us first show that if $y$ is close enough to $x$ (depending on $\epsilon_0$, as well as the representation $\rho$), we can find $\epsilon_1,L_1$ for which Equation (\ref{e.lc}) holds.

In order to do so, observe that, since the group is hyperbolic, and thus the cone-type graph is finite, there exists $K$ smaller than the diameter of the cone-type graph such that, if $x\in\bordcone(\alpha)$, there exists a bi-infinite geodesic $(\alpha_i)_{i\in\Z}$ passing through the identity, and an integer $M$ such that $d(\alpha_{-M},\alpha^{-1})<K$; of course in this case $||\alpha|-M|<K$. We denote by $z$ be the second endpoint of such geodesic. 

\begin{figure}[hh]
	\centering 
	\begin{tikzpicture}[scale=.8]
	\draw (0,0) ellipse (2.5 and 2);
\draw(-2.5, 0) to (2.5,0);
	\filldraw(-2.5,0) circle [radius=1pt] node[left]{$z$};
    \filldraw(2.5,0) circle [radius=1pt] node[right]{$x$}; 
    \filldraw(-1.5,0) circle [radius=1pt] node[above]{$\alpha_{-M}$};
   \filldraw(-1.5,-0.3) circle [radius=1pt] node[below]{$\alpha^{-1}$};
   \filldraw(0,0) circle [radius=1pt] node[above]{$e$};
   \filldraw(1.8,1.35) circle [radius=1pt] node[above]{$y$};
   \filldraw(1.4,0) circle [radius=1pt] node[below]{$\alpha_{N}$};
\draw (1.4,0) to(1.8,1.35);
\node at (0,2)[above] {$\bord\G$};
	\end{tikzpicture}
	\caption{The first step in the proof of Proposition \ref{p.main} }
\end{figure}
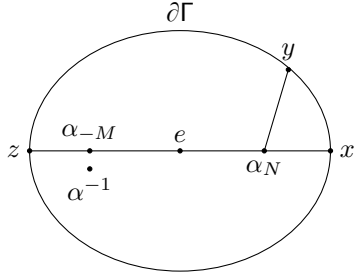
{By Lemma \ref{separated}} we can choose $N\in\N$ such that   
{$\alpha_N^{-1}x,\alpha_N^{-1}y,\alpha_N^{-1}z$} are $\nu_0$-separated. The size of $N$ measures how close $y$ is to $x$. Using the triangular inequality we get 
$$
\sin\angle \big(x^p_\rho\oplus y^q_\rho,U_{d-r}(\rho(\alpha^{-1}))\big)\geq \sin\angle \big(U_r(\rho(\alpha_N)),U_{d-r}(\rho(\alpha_{-M}))\big)$$
$$-d\Big(U_{d-r}\big(\rho(\alpha_{-M})\big),U_{d-r}\big(\rho\big(\alpha^{{-1}}\big)\big)\Big)- d \Big(\rho(\alpha_N)\big((\alpha_N^{-1}x)^p_\rho\oplus (\alpha_N^{-1}y)^q_\rho\big),U_{r}\big(\rho(\alpha_N)\big)\Big).
$$
The first term of the expression is bigger than $\delta_{r,\rho}$ provided $|\alpha|$ is big enough, by Lemma \ref{weakmorselemma}. The second term is smaller than $\delta_{r,\rho}/3$ if $|\alpha|$ is big enough by Lemma \ref{l.qg} (2): indeed $\alpha^{-1}=\alpha_{-M}a$ for some $a\in\G$ with $|a|<K$. We chose $L_1$ so that these two conditions are satisfied. In order to prove our claim it is enough to verify that we can find $N_0$ big enough, depending on the representation only, such that for every $N\geq N_0$, it holds 
$$d (\rho(\alpha_N)((\alpha_N^{-1}x)^p\oplus (\alpha_N^{-1}y)^q,U_{r}(\rho(\alpha_N)))<\delta_\rho/3.$$ 
{Since $z\neq x$ are fixed, the subspaces} $z^{d-r}_\rho$ and $x^r_\rho$ {have a positive angle} and thus,  since $U_r(\rho(\alpha_N))\to x^r_\rho$ as $N\to\infty$ uniformly in $N$, { the angle between $z^{d-r}_\rho$ and $U_r(\rho(\alpha_N))$ is bounded below for all positive big enough $N$ depending only on the representation $\rho$.} Using Lemma \ref{l.dominationattractor} we deduce that 
\begin{equation}\label{nearZ}d\Big(U_{d-r}\big(\rho(\alpha^{-1}_N)\big),\rho(\alpha^{-1}_N) z^{d-r}_\rho\Big)\leq \frac{\sigma_{d-r+1}}{\sigma_{d-r}}(\rho(\alpha_N^{-1}))\frac1{\sin\angle(z^{d-r}_\rho,U_r(\rho(\alpha_N)))}.\end{equation} Since the representation is $\{\sroot_{d-r}\}$-Anosov,{$\big((\sigma_{d-r+1})/(\sigma_{d-r})\big)(\rho(\alpha_N^{-1}))$} is smaller than $\epsilon_0/2$ for big enough positive $N.$ { By hyperconvexity (equation (\ref{hyper})) we know that $\big(\alpha_N^{-1}x\big)^p_\rho\oplus \big(\alpha_N^{-1}y\big)^q_\rho$ has a definite angle with $\big(\alpha_N^{-1}z\big)^{d-r}_\rho,$ consequently, by equation (\ref{nearZ}) we deduce that $$\angle \Big(\big(\alpha_N^{-1}x\big)^p_\rho\oplus \big(\alpha_N^{-1}y\big)^q_\rho,U_{d-r}\big(\rho(\alpha_N^{-1})\big)\Big)>\eps_0/2.$$ Thus,  by Lemma \ref{l.dominationattractor}}
$$d \Big(\rho(\alpha_N)\big((\alpha_N^{-1}x)^p_\rho\oplus (\alpha_N^{-1}y)^q_\rho\big),U_{r}\big(\rho(\alpha_N)\big)\Big)<\frac{\sigma_{r+1}}{\sigma_{r}}\big(\rho(\alpha_N)\big)\frac{2}{\epsilon_0}.$$
This concludes the first step, we can chose $\epsilon_1=\delta_{r,\rho}/3.$

We are thus left to verify that, up to possibly shrinking $\epsilon_1$ and enlarging $L_1$, Equation (\ref{e.lc}) is also verified in the case  $n$ for which $\alpha_n^{-1}(x,y,z)$ is $\nu_0$-far and smaller than{ a fixed }$N.$ Observe that, since the group $\G$ is finitely generated{ and $N$ is fixed,} we can find $C$, depending on $\rho$, such that $d(\alpha_n^{-1}x,\alpha_n^{-1}y)\leq C^n d(x,y)$, and therefore  we can find $\nu_1$ depending on $N$ only such that $d(y,x)>\nu_1$. Since furthermore $y^q_\rho\in X^q_\rho(\alpha)$, and thus we have a lower bound on $d(y,z)$, we deduce, up to further shrinking $\nu_1$,  that the triple $(x,y,z)$ is $\nu_1$-far. The same argument as above let us deduce that there exists $\epsilon_2$ such that     
$$\sin\angle(x^p_\rho\oplus y^q_\rho,z^{d-r}_\rho)>\epsilon_2.$$
It is then enough to chose $L_2$ big enough so that $d(z^{d-r}_\rho, U_{d-r}\big(\rho(\alpha^{-1})\big)\big)<\epsilon_2/2$. The Proposition holds with $L=\max\{L_1,L_2\}$ and $\epsilon=\min\{\epsilon_1,\epsilon_2\}$.
\end{proof}

Proposition \ref{p.main} combined with Theorem \ref{thm:eqifLC} yields the following Hausdorff dimension computations.

\begin{cor}\label{2FrenetHffDim} Let $\rho:\G\to\PGL_d(\K)$ be $(1,1,2)$-hyperconvex, then $$\Hff(\xi^1(\bord\G))=h^{\sroot_1}_\rho.$$
\end{cor}
\begin{cor}\label{11pLC} Let $\rho:\G\to\PGL_d(\K)$ be $(1,1,r)$-hyperconvex. Assume moreover that for every $\g\in\G$ one has $\sigma_2\big(\rho(\g)\big)=\sigma_r\big(\rho(\g)\big),$ then every point of $\bord\G$ is locally conformal for $\rho$ and thus $$h^{\sroot_1}_\rho=\Hff\big(\xi^1(\bord\G)\big).$$
\end{cor}

\subsection{Examples: (ir)reducible $\SL_2$}\label{SL2}
The easiest examples of hyperconvex representations are induced from representations of $\SL_2(\K)$ (see for example Humphreys's book \cite{james} for standard basic facts on the representation theory of $\SL_2$). 

Recall that for every $d\in\N-\{0,1\}$ there is a (unique up to conjugation) irreducible representation $\iota_d:\SL_2(\K)\to\SL_d(\K).$  This representation is given by the action of $\SL_2(\K)$ on the symmetric powers $\sym^{d-1}(\K^2),$ which can be identified with the space of homogenous polynomials on two variables of degree $d-1$ with coefficients in $\K.$ If we denote by $\EE^*_{\SL_2(\K)}$ the weight space, the representation $\iota_d$ has highest weight $\chi_{\iota_d}\in\EE^*_{\SL_2(\K)}$ given by $\chi_{\iota_d}(x)=(d-1)x.$

Let $\cal F(\sym^{d-1}(\K^2))$ denote the full flag space associated to $\SL(\sym^{d-1}(\K^2))$. The \emph{Veronese map} $\zeta:\P(\K^2)\to\cal F(\sym^{d-1}(\K^2))$ is defined by $$\zeta(x)=\{\zeta^k(x)\}_{k=1}^{d-1}$$ where $\zeta^k(\ell)$ is the $k$-dimensional vector subspace of $\sym^{d-1}(\K^2)$ consisting of polynomials that have $x^{d-k}$ as a factor. It is easy to check that $\zeta$ is $\iota_d$-equivariant and the image of an attractor in $\P(\K^2)$ is an attractor in $\cal F(\sym^{d-1}(\K^2)).$

\begin{obs}Note that for every pair of distinct points $x\neq y$ in $\P(\K^2)$ the flags $\z(x)$ and $\z(y)$ are in general position, i.e. for every $k\in\lb1,d-1\rb$, it holds $\z^k(x)\cap\z^{d-k}(y)=\{0\}.$
\end{obs}

Moreover, using the transitivity of the $\SL_2(\K)$-action on transverse pairs, it is easy to check the following:
\begin{prop}\label{VeronesseHyperconvex}Let $\zeta=\{\zeta^i\}_{i=1}^{d-1}$ be the Veronese embedding of $\P(\K^2)$ into $\cal F(\sf{S}_{d-1}(\K^2)),$ then for every triple $p+q+r=d $ and pairwise distinct $x,y,z\in\P(\K^2)$ one has $$\z^p(x)\oplus\z^q(y)\oplus\z^r(z)=\K^d.$$
\end{prop}

\begin{cor}
For every convex cocompact\footnote{For non-Archimedean fields $\K$, in analogy with the Archimedean case, we say that a representation is \emph{convex cocompact} if it is Anosov, as in Definition \ref{defAnosov}.
} subgroup $\G<\SL_2(\K)$, the representation $\iota_d|_\G:\G\to\SL_d(\K)$ is $(p,q,r)$-hyperconvex for every $(p,q,r)$ such that $r\geq p+q$. The same holds for small deformations.
\end{cor}

\medskip

We can obtain many more examples of hyperconvex representations by considering direct sums of irreducible representations.
A representation $\pi:\SL_2(\K)\to\SL(V,\K)$ decomposes in irreducible modules $$\pi=\bigoplus_1^k\iota_{d_i},$$ where we have ordered $d_1\geq\cdots\geq d_k.$ The highest weight $\chi_\pi\in\EE^*_{\SL_2(\K)}$ is $\chi_\pi(x)=(d_1-1)x.$ Let us denote by $$\chi^{(2)}_\pi\geq\cdots\geq \chi^{(\dim V)}_\pi$$ the remaining weights in decreasing order. 

\begin{defi}\label{coherentDef} Given $k\in\lb2,\dim V\rb,$ we say that $\pi$ is $k$-\emph{coherent} if $\chi^{(k)}_\pi>d_2-1,$ equivalently if $d_1>d_2+2(k-1)$. 
\end{defi}
Observe that a representation $\pi$ is $k$-coherent if and only if the representation has a gap of index $k$ and the  top $k$ eigenspaces are  eigenlines of a diagonalizable element in $\pi(\SL_2(\K))$ and belong to the top irreducible factor. 
An important example of 2-coherent representations are exterior powers:

\begin{ex}\label{Principal2coherent}For every $p\in\lb1,d-1\rb$ the representation $$\wedge^p\iota_d:\SL_2(\K)\to\SL(\wedge^p\K^d)$$ is $2$-coherent.
\end{ex}

\begin{proof} Considering a diagonalizable element in $\SL_2(\K)$ one explicitly checks that the top $3$ weights of $\wedge^p\iota_d$ are 
\begin{itemize}
\item[-]$\chi_{\wedge^p\iota_d}=d-1+\ldots+ d+1-2p=p(d-p),$\item[-]$\chi_{\wedge^p\iota_d}^{(2)}=\chi_{\wedge^p\iota_d}-2,$\item[-]$\chi_{\wedge^p\iota_d}^{(3)}=\chi_{\wedge^p\iota_d}^{(4)}=\chi_{\wedge^p\iota_d}-4.$
\end{itemize}
\end{proof}

Definition \ref{coherentDef} guarantees some hyperconvexity:

\begin{prop}\label{2-coherent} Let $\rho:\G\to\SL_2(\K)$ be convex co-compact. If $\pi:\SL_2(\K)\to\SL(V,\K)$ is $k$-coherent, then $\pi\circ\rho$ is $(p,q,k)$-hyperconvex for every $p,q$ with $ p+q\leq k.$
\end{prop}

\begin{proof} Since $\chi^{(k)}>d_2-1$ one has that $\chi^{(k)}>\chi^{(k+1)}$ and thus $\pi\circ\rho$ is $\{\sroot_k\}$-Anosov. Coherence implies thus that $\chi^{(l)}>\chi^{(l+1)}$ for every $l\in\lb1,k\rb$ and thus $\pi\circ\rho$ is also $\{\sroot_p,\sroot_q\}$-Anosov since both $p$ and $q$ are smaller than $k.$ The remainder of the statement follows from Lemma \ref{VeronesseHyperconvex}: if $N$ denotes the dimension of $V$, $\zeta^p_{d_1}:\bord\G\to \Gr_p(\sym^{d-1}(\K^2))\subset\Gr_p(V)$ is the $\iota_d$-equivariant  map induced by $\zeta$, and $\xi_\pi^l:\bord\G\to\Gr_l(V)$ denotes the boundary map associated to $\pi$, we have,  for every $l\leq k$, that $\xi_\pi^l=\zeta^l_{d_1}$ and $$\xi_\pi^{N-l}= \zeta^{d_1-l}_{d_1}\oplus \bigoplus_{i=2}^k \sym^{d_i-1}(\K^2).$$
\end{proof}

In particular Proposition \ref{2-coherent} can be used to construct example of representations of Kleinian groups satisfying the assumptions of Theorem \ref{thm:eqifLC}.
\section{Differentiability properties}\label{s.diff}
\subsection{Convergence on pairs and bounds on the Hausdorff dimension}
The following result, which follows from Proposition \ref{p.main} is inspired by Guichard \cite[Proposition 21]{guichard}, however,  Guichard's proof relies heavily on the fact that $\bord\G$ is a circle, and that the representation is $(p,q,r)$-hyperconvex for \emph{every} triple $p,q,r$ with $p+q=r.$

\begin{thm}\label{FRep->Fpoint}
Let $\rho:\G\to\PGL_d(\K)$ be $(p,q,r)$-hyperconvex then for every $(w,y)\in\bord^{(2)}\G$ one has $$\lim_{(w,y)\to (x,x)}d(w^p_\rho\oplus y^q_\rho,x^r_\rho)=0.$$
More precisely there exist constants $C,\mu$ such that, if $\{\alpha_i\}$ is a geodesic ray with endpoint $x$, for every $w,y\in\alpha_i\bordcone(\alpha_i)$ it holds
$$d(w^p_\rho\oplus y^q_\rho,x^r_\rho)\leq C e^{-\mu i}.$$
\end{thm}

\begin{proof}
The first claim is a direct consequence of the second, as the sets $\alpha_i\bordcone(\alpha_i)$ form a fundamental system of neighbourhoods of the point $x$. 

As the representation $\rho$ is $(p,q,r)$-hyperconvex, and $w,y\in\alpha_i\bordcone(\alpha_i)$, we deduce from Proposition \ref{p.main} that 
$$\sin\angle ((\alpha_i^{-1} w)^p_\rho\oplus (\alpha_i^{-1} y)^q_\rho,U_{d-r}(\rho(\alpha_i^{-1})))>\epsilon.$$
In particular Lemma \ref{l.dominationattractor} implies
$$d(w^p_\rho\oplus y^q_\rho, U_r(\alpha_i))\leq \frac{\sigma_{r+1}}{\sigma_{r}}(\rho(\alpha_i))\frac{1}{\epsilon}\leq \frac{C_1}{\epsilon}e^{-\mu_1 i}.$$
Where $C_1,\mu_1$ are the constants provided by the fact that $\rho$ is $\{\sroot_r\}$-Anosov.
The result now follows, via triangular inequality, from Lemma \ref{l.uniformbdry}, which guarantees that 
$$d(x^r_\rho, U_r(\alpha_i))\leq C_2e^{-\mu_2 i}.$$
\end{proof}

The following easy converse is useful for applications:

\begin{prop}\label{easyConverse}Consider $p,q,r\in\lb1,d-1\rb$ with $p+q\leq r.$ If $\rho:\G\to\PGL_d(\K)$ is $\{\sroot_p,\sroot_q,\sroot_r\}$-Anosov and for every $x\in\bord\G$ one has 
\begin{equation}\label{e.conv}
\lim_{(w,y)\to(x,x)}d(w^p_\rho\oplus y^q_\rho,x^r_\rho)=0,
\end{equation} then $\rho$ is $(p,q,r)$-hyperconvex.
\end{prop}

\begin{proof} Since $\rho$ is $\{\sroot_p,\sroot_q\}$-Anosov and $p+q\leq r\leq d-1,$ for every pair of distinct points $w,y$ the sum $w^p_\rho+y^q_\rho$ is direct. Since $\rho$ is $\{\sroot_r\}$-Anosov there is a lower bound on $\sin\angle (x^r_\rho, z_\rho^{d-r})$ if $x,z$ are the endpoints of a geodesic through the origin. Combining this fact with (\ref{e.conv}) we can find $\epsilon,\delta$ such that $$(x^p_\rho\oplus y^q_\rho)\cap z^{d-r}=\{0\}$$ for every triple with $d(x,y)<\eps$ and $d(x,z)>\delta>\eps.$ Any triple in $\bord^{(3)}\G$ can be transformed in such a triple by an element of $\G$ and thus the claim follows.
\end{proof}

Using the stereographic projection (see Definition \ref{stereo}) combined with Theorem \ref{FRep->Fpoint} it is possible to deduce the following estimate on Hausdorff dimension:

\begin{prop}\label{inequality} Let $\rho:\G\to\PGL_d(\K)$ be $(1,1,r)$-hyperconvex, then 
$$\Hff\big(\xi^1(\bord\G)\big)\leq \Hff\big(\P(\K^r)\big).$$
\end{prop}

\begin{proof}
{ We first claim that if $\rho:\G\to\PGL_d(\K)$ is $(1,1,r)$-hyperconvex, then  for every $x$ we can find a point $z$ an open neighbourhood $\cal U_x$ of $x$ in $\xi^1(\bord\G)$ such that  the stereographic projection $\pi_{z,\rho}$ is Lipschitz on $\cal U_x$. Indeed as $\rho$ is $\{\sroot_r\}$-Anosov, we can choose $z$ so that the subspaces $x_\rho^r$ and $z_\rho^{d-r}$ make a definite angle. The claim is then a consequence of Theorem \ref{FRep->Fpoint}: Indeed, it implies we can find an open neighbourhood $\cal U_x$ of $x$ such that for every pair $w, y\in \cal U_x$ the angle that $w^1_\rho\oplus y^1_\rho$ makes with $z_\rho^{d-r}$ is bigger than a fixed constant. This is enough to guarantee that the stereographic projection doesn't distort distances too much.

In particular, as Lipschitz maps  preserve the Hausdorff dimension, it follows that $\Hff(\cal U_x)\leq \Hff\big(\P(\K^r)\big)$.
Since the Hausdorff dimension of a compact set is the maximum of the Hausdorff dimensions of the sets in a finite open cover, the result follows.}
\end{proof}

\subsection{When $\bord\G$ is a manifold and $\K=\R$}\label{repdef} A classical result of Benoist \cite{convexes1} states that if a word hyperbolic group of projective transformations divides a convex set, then the boundary of this set has to be $\class^1.$ These, together with Hitchin representations, have become the paradigm of Zariski-dense projective Anosov representations whose limit set is a regular manifold. The purpose of this section is to provide new examples of such phenomena. Sharper results of similar nature have recently been obtained independently by Zhang-Zimmer \cite{ZZ}.

We begin by observing that Theorem \ref{FRep->Fpoint} has the following interesting consequence.

\begin{prop}\label{diff} Let $\rho:\G\to\PGL_d(\R)$ be a $(1,1,r)$-hyperconvex representation and assume that $\bord\G$ is topologically a sphere of dimension $r-1,$ then $\xi^1_\rho(\bord\G) $ is a $\class^1$ manifold with ${\sf{T}}_{x^1_\rho}\xi^1_\rho(\bord\G)={\sf{T}}_{x^1_\rho}\P( x^r_\rho).$ 
\end{prop}
\begin{proof} 
Theorem \ref{FRep->Fpoint} implies that the set $\xi^1_\rho(\bord\G)$ is differentiable at $x^1_\rho$ with tangent space ${\sf{T}}_{x^1_\rho}\P( x^r_\rho)$. The continuity of $x\mapsto x^r_\rho$ completes the proof. 

\end{proof} Proposition \ref{diff} can be applied to many different situations to produce interesting examples through the \emph{represent and deform} method, we now explain how this works in a specific situation. Denote by 
$$\sym^k: \PGL_{d+1}(\R)\to\PGL(\sym^k(\R^{d+1}))$$ 
the $k$-symmetric power.

Note that in $\PGL_{d+1}(\K)$ a $(1,1,d)$-hyperconvex representation is a projective Anosov representation $\rho$ such that for each triple $(x,y,z)\in\bord^{(2)}\G$ the sum $x_\rho^1+y_\rho^1+z_\rho^1$ is direct.

\begin{prop}\label{symm}
Let $\rho:\G\to\PGL_{d+1}(\R)$ be a $(1,1,d)$-hyperconvex representation and assume that there exist $c>0,\mu>1$ such that, for every $\gamma\in\G$, 
\begin{equation}\label{eqn.gaps}\frac{\sigma_1(\rho(\gamma))\sigma_d(\rho(\gamma))}{\sigma_2(\rho(\gamma))^2}>ce^{\mu|\gamma|}.\end{equation}  Then the composition 
$$
\sym^k\circ\rho:\G\to\PGL(\sym^k(\R^{d+1}))
$$
is $(1,1,d)$-hyperconvex.
\end{prop}

\begin{proof} We endow $\sym^k(\R^{d+1})$ with the norm induced by our choice of norm on $\R^{d+1}$. For this choice, and for every $g\in \PGL_{d+1}(\R)$, the semi-homotecy ratios of $\sym^kg$ are just the products of $k$-tuples of semihomotecy ratios of $g$. Assumption (\ref{eqn.gaps}) then gives that for all $\gamma$ apart from possibly finitely many exceptions
\begin{itemize}
\item[-]$\sroot_1(\nu(\sym^k\rho(\g)))=\sroot_1(\nu(\rho(\g))),$
\item[-] $\sroot_d(\nu(\sym^k\rho(\g)))= \min\{\sroot_d(\nu(\rho(\gamma))),\log\frac{\sigma_1(\rho(\gamma))\sigma_d(\rho(\gamma))}{\sigma_2(\rho(\gamma))^2}\} $
\end{itemize}
Since $\rho$ is $\{\sroot_1,\sroot_d\}$-Anosov, we deduce from Definition \ref{defAnosov}
that $ \sym^k\rho$ is also $\{\sroot_1,\sroot_d\}$-Anosov.

Observe that the map $\sym^k$ is equivariant with respect to the map between the partial flags 
$$\sym^k:\{\textrm{line}\subset\textrm{hyperplane}\}\to\{\textrm{line}\subset\textrm{$d$-dimensional subspace}\}$$ defined by 
$$\sym^k(l,H)=(l^{\odot k}, l^{\odot k-1}\odot H\rangle).$$
Here we denote by $\odot$ the symmetric tensors.

It is immediate to verify that Assumption (\ref{eqn.gaps}) also implies that $\sym^k\circ\xi$ sends attractors to attractors, therefore, by continuity of $\sym^k\circ\xi$, we have, for every $x\in\bord\G$,  
$${ \sym^k}(x^1_\rho,x^d_\rho)=(x^1_{ \sym^k\rho},x^d_{ \sym^k\rho}).$$ 

Finally, the convergence property (Theorem \ref{FRep->Fpoint}) for $\rho,$ together with the differentiability of $\sym^2:\P(\R^{d+1})\to\P\big(\sym^2(\R^{d+1})\big) $ implies that $$\lim_{(w,y)\to(x,x)} \angle\big(w^1_{ \sym^2\rho}\oplus y^1_{ \sym^2\rho},x^d_{ \sym^2\rho}\big)=0.$$ Proposition \ref{easyConverse} yields the result.
\end{proof}
As a direct corollary we get:
\begin{cor}
If $\rho:\G\to\PSO(d,1)$ is cocompact, every small deformation 
$$\eta:\G\to\PGL(\sym^k(\R^{d+1}))$$
of $\sym^k\rho$ is $(1,1,d)$-hyperconvex. Any such $\eta$ will have a $\class^1$-sphere as limit set in $\P(\sym^k(\R^{d+1})).$
\end{cor}

Applying Johnson-Millson's \cite{JM} bending technique we obtain the announced Zariski dense subgroups whose limit set is a $\class^1$ sphere:
\begin{cor}\label{Zdense}
There exists a Zariski dense subgroup $\G<\PGL(\sym^2(\R^{d+1}))$ whose limit set is a $\class^1$ sphere.
\end{cor}
\begin{proof}
Let $M$ be a $d$-dimensional closed hyperbolic manifold that has a totally geodesic, co-dimension one, closed submanifold $N$. The inclusion $\G=\pi_1M\subset\SO(d,1)\to\SL\big(\sym^2(\R^{d+1})\big)$ satisfies the hypothesis of Proposition \ref{symm}. Without loss of generality we can assume that $\pi_1N\subset \SO(d-1,1)$. Observe that the centralizer of $$\sym^2(\SO(d-1,1))\subset\SL\big(\sym^2(\R^{d+1})\big)$$ 
is non-trivial and strictly contains that of $\sym^2\big(\SO(d,1)\big)$:  as an $\sym^2\big(\SO(d,1)\big)$- module, $\sym^2(\R^{d+1})$ splits as a direct sum of an irreducible representation (usually denoted $\mathbb S_{[2]}(\R^{d,1})$) and a trivial representation, its centralizer is thus reduced to $\R^*$. The decomposition as a  $\sym^2(\SO(d-1,1))$-module splits as the sum $\mathbb S_{[2]}(\R^{d-1,1})\oplus \R^{d-1,1}\oplus \R^2$ where the action on the second factor is the standard action, while the action on $\R^2$ is trivial. In particular the centralizer of $\sym^2(\SO(n-1,1))$ is $\GL(2,\R)\times\R^*$. By bending the representation along $N$ with a nontrivial element in $\GL(2,\R)$ which doesn't leave invariant the factor $\R$, we obtain the desired representation.
\end{proof}

\section{Examples of locally conformal representations}\label{s.exLC}
The purpose of this section is to discuss some of the many examples in which restricting the Zariski closure of a representation to a non-split real form of $\SL_d(\K)$ gives room for $(1,1,p)$ hyperconvex representations for which we can also guarantee that the second gap $p_2$ is strictly bigger than 2.
\subsection{Hyperconvex representations in $\PU(1,d)$ and $\PSp(1,d)$}\label{sec:POK}
The first interesting setting in which Theorem \ref{thm:eqifLC} applies for large classes of representations is given by considering representations in the rank one groups $\PU(1,d)$ or $\PSp(1,d)$. 
To unify the treatment we will write $\POK(1,d)$ for either $\PU(1,d)$ if $\K=\C$ or $\PSp(1,d)$ if $\K=\HH$ and regard $\POK(1,d)$ as a subgroup of $\PGL(d+1,\K)$.
\begin{obs}
Unfortunately, as $\HH$ is non-commutative, we don't have the setup of Section \ref{sec:2.2} at our disposal (as the exterior algebra over a non-commutative field is not well defined), however the issue can be easily solved by considering $\SL(d+1,\HH)$ as a subgroup of $\SL(2d+2,\C)$. Given an element $g\in \SL(d+1,\HH)$ we denote by $g^\C$ the corresponding element in  $\SL(2d+2,\C)$; it is then immediate to verify that we can choose a Cartan decomposition of $g^\C$ so that, for every $p$, the subspace $U_{2p}(g^\C)$ is a quaternionic vector space, and we thus set $U_p(g):=U_{2p}(g^\C)$. Similarly we say that a sequence $(\alpha_i)_{i\in\Z}$ in  $\SL(d+1,\HH)$ is $p$-dominated if $(\alpha_i^\C)_{i\in\Z}$ is $2p$-dominated in  $\SL(2d+2,\C)$, and that a representation $\rho:\G\to\SL(d+1,\HH)$ is $(p,q,r)$-hyperconvex if the induced representation $\rho:\G\to\SL(2d+2,\C)$ is $(2p,2q,2r)$-hyperconvex. With this at hand it is easy to verify that Theorem \ref{thm:eqifLC} holds for representations with values in $\SL(d+1,\HH)$.
\end{obs}

Recall that $\POK(1,d)$ has rank one, therefore we have at our disposal a good notion of convex co-compactness: a representation $\rho:\G\to\POK(1,d)$ is convex co-compact if and only if there is a convex $\rho(\G)$-invariant subspace of $\mathbb H^d_\K$ whose quotient is compact. The induced representation $\rho:\G\to\PGL(d+1,\K)$ is $\{\sroot_1\}$-Anosov if and only if $\rho$ is convex co-compact,{ see for example Guichard-W. \cite[Section 6.1]{GW-Domains}}. 

 Observe that $\POK(1,d)$ preserves the closed codimension 1 submanifold $\bord\HdK\subset \P(\K^{d+1})$, furthermore one has the following.
\begin{lemma}\label{lem:POK}
For every $g\in\POK(1,d)$, we have $U_1(g)\in\bord\HdK\subset \P(\K^{d+1})$ and $U_d(g)=U_1(g)^\bot$, where the orthogonal is defined with respect to the Hermitian form defining the group $\POK(1,d)$. 
\end{lemma}

In particular, considering for every point $x\in\bord\HdK$ the subspace $x^\bot\subset {\sf{T}}_x \bord\HdK$, one obtains a non-integrable distribution that has (real) codimension 1 if $\K=\C$ and 3 if $\K=\mathbb H$. In the complex case this is the standard contact structure on the sphere. We will refer to this distribution also in the quaternionic case as the \emph{generalized contact distribution.} Given a distinct pair $x,y\in\bord\G$ we will denote by  $\mathcal C_{x,y}$ the intersection $\P(\langle x,y\rangle)\cap \bord\HdK$. Of course if $\K$ is $\C$ then $\mathcal C_{x,y}$ is a circle, while if $\K=\mathbb H$ it is a 3-sphere. In the complex case the sets $\mathcal C_{x,y}$ are often referred to as \emph{chains}, and their geometry was extensively studied by Cartan. The incidence geometry of chains (and of suitable generalizations) played an important role in Burger-Iozzi \cite{BI} and P. \cite{Po}. 

With these definition at hand we can rephrase our main results in the rank 1 setting:
\begin{prop}\label{prop:hor}
A convex cocompact action $\rho:\G\to \POK(1,d)$ is $(1,1,d)$-hyperconvex if and only if for every distinct pair $x,y\in\bord\G$, the chain $\mathcal C_{x^1_\rho,y^1_\rho}$ intersects $\xi(\bord\G)$ only in $x^1_\rho,y^1_\rho$. In this case $\LC(\rho)=\xi(\bord\G)$, and $\xi(\bord\G)$ is tangent to the generalized contact distribution.
\end{prop}
\begin{proof}
The first statement follows directly from the definitions: for every triple $x,y,z$ the sum $x^1_\rho+y^1_\rho+z^1_\rho$ is direct if and only if $z^1_\rho$ doesn't belong to $\mathcal C_{x^1_\rho,y^1_\rho}$. The second statement follows then from Proposition \ref{p.main}, and the last is a consequence of Theorem \ref{FRep->Fpoint}.
\end{proof}

There are many interesting examples of representations satisfying the assumption of Proposition \ref{prop:hor}, a natural class of examples can be obtained deforming totally real embeddings. The following is a direct consequence of Proposition \ref{prop:hor}:
\begin{lemma}\label{lem:POR}
Let $\G<\PO_\R(1,d)$ be a convex cocompact subgroup and let $\rho:\G\to\POK(1,d)$ be obtained extending the coefficients. Then $\rho$ is $(1,1,d)$-hyperconvex.
\end{lemma}
\begin{cor}
Every $\{\sroot_1\}$-Anosov representation $\beta:\G\to \POK(1,d)$ sufficiently close to a totally real representation $\rho$ is $(1,1,d)$-hyperconvex. In particular for each such representation 
$$\dim_{\Hff}(\xi(\bord\G))=h^{\sroot_1}_\beta\leq (d-1)\dim\K.$$
\end{cor}
\begin{proof}
The first statement is a direct consequence of Propositions \ref{FrenetOpen} and \ref{prop:hor}. Furthermore we know that for every element  $g\in\POK(1,d)$, we have $p_2(g)=d$, and hence  every point in $\bord\G$ is locally conformal for $\beta$.  Theorem \ref{thm:eqifLC} then applies and gives the second statement.
\end{proof}

Another class of examples was studied by Dufloux in his thesis \cite{dufloux-these, HCSchottky}. He says that a Schottky subgroup  $\G<\PU(1,d)$ generated by a symmetric set $W$ is \emph{well positioned} if,  for every $w\in W$ there is an open subsets $B(w)\subset\bord\HdC$ such that 
\begin{itemize}
\item the closures $\overline{B(w)}$ are pairwise disjoint;
\item $w(\bord\HdC\setminus B(w^{-1}))\subset B(w)$;
\item no chain passes through three of these open subsets $B(w)$.
\end{itemize}
Similarly one can define well positioned Schottky subgroups of $\PSp(1,d)$ replacing chains 
with quaternionic three spheres (recall that in $\bord \HdH$ any pair of points uniquely determines a 3 sphere, the boundary of a totally geodesic copy of $\HoH$). We will denote also these subspaces of $\bord\HdH$ \emph{chains} for notational ease.

Arguments analogue to the ones presented in \cite[Section 7.2]{HCSchottky} imply that well positioned Schottky groups are hyperconvex representations: 

\begin{prop}\label{p.9.3}
Let $\rho:\G\to \POK(1,d)$  be a well positioned Schottky subgroup. Then $\rho:\G\to\SL_{d+1}(\K)$ is $(1,1,d)$-hyperconvex. Furthermore  $\LC(\rho)=\xi(\bord\G)$. 
\end{prop}
\begin{proof}
Observe that since $\POK(1,d)$ is a rank one group, $p_2(\alpha)$ doesn't depend on $i$. Furthermore, as soon as the sequence $\{\alpha_i\}_{i=1}^\infty$ forms a geodesic ray, the sequence is $d$ dominated by a classical ping pong argument, and it follows from Lemma \ref{lem:POK} that $E_{p_2}^{\rho}(x)=x^\bot\subset\P(\K^{d+1})$.

In order to verify that every point $x\in\bord\G$ is locally conformal,  we need to check that there exists a constant $c$ such that  $\angle(\xi(y)\oplus\xi(z),U_{1}(\rho(\alpha^{-1})))>c$ for all $y,z\in X_\infty(\alpha)$. Since $\G$ is a well positioned Schottky group, we can choose $\delta_\rho$ as the smallest distance between two sets $B(w)$. Let $ w_\alpha$ be the first letter of $\alpha$. It follows from Lemma \ref{l.uniformbdry} that
if $|\alpha|$ is big enough $U_{1}(\rho(\alpha^{-1}))\in B(w_\alpha^{-1})$ and $X_\infty(\alpha)\subset \bigcup_{s\neq w_\alpha}B(s)$: by construction $\bordcone(\alpha)\subset  \bigcup_{s\neq w_\alpha}B(s)$ and the intersection of the $\delta_\rho/2$ neighbourhood of $\bigcup_{s\neq w_\alpha}B(s)$ with the image of the boundary map is already contained in $\bigcup_{s\neq w_\alpha}B(s)$. 

Since the chain $\mathcal C_{y,z}$ through $y_\rho^1$ and $z_\rho^1$ is the intersection of $\bord\HdK$ with $\P(y_\rho^1\oplus z_\rho^1)$,   and, by assumption,  $\mathcal C_{y,z}$ doesn't intersect the open subset $B(w_\alpha)\subset\bord\HdK$, the result follows. The fact that the representation $\rho:\G\to\SL_{d+1}(\K)$ is $(1,1,d)$-hyperconvex is a consequence of Theorem \ref{FRep->Fpoint}.
\end{proof}
\begin{cor}[{cfr. \cite[Corollary 43]{dufloux-these}}]
Let $\rho:\G\to \POK(1,d)$  be a well positioned Schottky subgroup. Then $$\Hff(\xi^1_\rho(\bord\G))=h^{\sroot_1}_\rho.$$
\end{cor}
\begin{proof}
If $\K=\C$ this follows directly from Theorem \ref{thm:eqifLC}. For $\K=\HH$ it is enough to observe that in the construction of the measure $\mu^{\sroot_1}$ performed in Section \ref{Patterson-Sullivan} we never used the commutativity of the field $\K$.
\end{proof}

We conclude the discussion on convex cocompact subgroups of $\POK(1,d)$ by showing that the set of $(1,1,d)$-hyperconvex representations is, in general, not closed within the space of projective Anosov representations. We will prove in Proposition \ref{FrenetClosed}, that, instead, $(1,1,2)$-hyperconvex representations of fundamental groups of surfaces are closed in the space of Anosov representations. Denote by $\F_{2}$ the free group on two generators.
\begin{prop}
There exists a continuous path of $\{\sroot_1\}$-Anosov representations $\rho_t:\mathbb F_2\to\PU(1,d)$  such that $\rho_0$ is $(1,1,d)$-hyperconvex and $\rho_1$ is not $(1,1,d)$-hyperconvex.
\end{prop}
\begin{proof}
As $\PU(1,d)$ has rank 1, for every 4-tuple $(a^+,a^-,b^+,b^-)$ of pairwise distinct points in $\bord\HdC$ we can find elements $a,b\in \PU(1,d)$ with prescribed attractive and repulsive fixed points and with translation length big enough so that the group generated by $a,b$ is free and convex cocompact on $\HdC$: this follows from a classical ping pong argument. Furthermore, if $(a^+_t,a^-_t,b^+_t,b^-_t)$ vary continuously in $t$ we can also arrange for the elements $a_t$, $b_t$ to vary continuously in $t$; in this way we can define a continuous path $\rho_t:\F_2\to\PU(1,d)$ of $\{\sroot_1\}$-Anosov representation.

Our claim follows if we choose $a_0,b_0$ contained in $\PO(1,d)$ (so that the representation is $(1,1,d)$-hyperconvex by Lemma \ref{lem:POR}), and $(a^+_1,a^-_1,b^+_1,b^-_1)$ so that $(a^+_1,a^-_1,b^+_1)$ belong to a single chain, but $b^-_1$ doesn't. In this case the representation $\rho_1$ is clearly not $(1,1,d)$-hyperconvex as the sum $\xi(a^+)+\xi(a^-)+\xi(b^+)$ is not direct. 
\end{proof}

\subsection{Locally conformal representations in $\SO(p,q)$}\label{s.SOpq}
We now turn our attention to the group $\SO(p,q)$.
Every semi-simple element $g\in\SO(p,q)$ has $|p-q|$ eigenvalues equal to 1. In this subsection, considering suitable exterior representations of $\SO(p,q)$ we will produce examples of hyperconvex representations for which every point is locally conformal, and thus Corollary  \ref{11pLC} applies. For these representations, the Hausdorff dimension of the limit set computes the critical exponent for the first simple root.

The following generalization of Labourie's property (H) \cite[Section 7.1.4]{Labourie-Anosov}        guarantees that a suitable exterior power is hyperconvex:

\begin{prop}\label{p.Opq1}
Let $\rho:\G\to\SO(p,q)$ be $\{\sroot_{p-1},\sroot_{p}\}$-Anosov (here $p\leq q$). Then $\wedge^p\rho:\G\to\PGL(\wedge^p(\R^{p,q}))$ is $\{\sroot_{1},\sroot_{q-p+1}\}$-Anosov. It is $(1,1,q-p+1)$-hyperconvex if and only if for every $x,y,z\in\partial\G$ pairwise distinct, the sum
$$x^p_\rho+(z^p_\rho\cap y^{q+1}_\rho)+ y^p_\rho$$
is direct. In this case every point in $\bord\G$ is locally conformal.
\end{prop}
\begin{proof}
Observe that the singular values of an element $g\in\SO(p,q)\subset\SL_{p+q}(\R)$ have the form  $\sigma_1(g)\geq\ldots\sigma_p(g)\geq 1=\ldots=1\geq\sigma_p(g)^{-1}\geq\ldots\sigma_1(g)^{-1}$, where $1$ has multiplicity at least $q-p$ (higher if $\sigma_p(g)=1$). If $\rho:\G\to\SO(p,q)$ is $\{\sroot_{p-1},\sroot_{p}\}$-Anosov, then, for every $\g$ with $|\g|$ big enough, it holds $\sigma_{p-1}(\rho(\g))>\sigma_{p}(\rho(\g))>1$, hence in particular 

\begin{alignat*}{2}
\sigma_1(\wedge^p\rho(\gamma)) & =\sigma_1(\rho(\gamma))\ldots\sigma_p(\rho(\gamma)), \\ 
\sigma_2(\wedge^p\rho(\gamma)) & =\sigma_{q-p+1}(\wedge^p\rho(\gamma))=\frac{\sigma_1(\wedge^p\rho(\gamma))}{\sigma_p(\rho(\gamma))}  \text{ and }\\
\sigma_{q-p+2}(\wedge^p\rho(\gamma)) &=\max\left\{\frac{\sigma_1(\wedge^p\rho(\gamma))}{\sigma_{p-1}(\rho(\gamma))},\frac{\sigma_1(\wedge^p\rho(\gamma))}{\sigma_p(\rho(\gamma))^2}\right\},\end{alignat*} 
which implies that $\wedge^p\rho$ is $\{\sroot_{1},\sroot_{q-p+1}\}$-Anosov.

Denote by $\cal F_{p-1,q}(\R^{p,q})$ the partial flag manifold consisting of pairs of $(p-1,q)$-dimensional isotropic subspaces and consider the map
$$\begin{array}{cccc}
L:&\cal F_{p-1,p}(\R^{p,q})&\to&\cal F_{1,q-p+1}(\wedge^p\R^{p,q})\\
&(P,Q)&\mapsto&(\wedge^p(Q),\wedge^{p-1}(P)\wedge Q^\bot)
\end{array}
$$
where the orthogonal is considered with respect to the bilinear form defining the group $\SO(p,q)$. The map $L$ is clearly equivariant with the homomorphism $\wedge^p:\SO(p,q)\to\SL(\wedge^p(\R^{p,q}))$; furthermore, if $g\in\SO(p,q)$ is $\cal F_{p-1,q}(\R^{p,q})$-proximal, namely $g$ has an attractive fixedpoint $g^+$ in $\cal F_{p-1,q}(\R^{p,q})$, then $L(g^+)=(\wedge^p g)^+$. Thus if $(\xi^{p-1},\xi^p):\bord\G\to\cal F_{p-1,q}(\R^{p,q})$ denote the boundary maps associated to $\rho:\G\to\SO(p,q)$, the boundary maps associated to $\wedge^p\rho $ have the form $L\circ (\xi^{p-1},\xi^p)$.

Let $N$ denote the dimension of $\wedge^p(\R^{p,q})$. In order to check if the representation $\wedge^p\rho$ is $(1,1,q-p+1)$-hyperconvex, it is enough to verify that for every distinct triple $x,y,z\in\bord\G$, the subspace $x^1_{\wedge\rho}+z^1_{\wedge\rho}$ intersects transversely $y^{N-q+p-1}_{\wedge\rho}$, or, equivalently, the image of $x^1_{\wedge\rho}+z^1_{\wedge\rho}$ in $\wedge^p\R^{p,q}/y^{N-q+p-1}_{\wedge\rho}$ is two dimensional. 

Recall that if $\rho:\G\to\SO(p,q)$ is $\sroot_{p}$-Anosov, then for every distinct pair $(x,y)\in\bord\G^2$ it holds $x^p_\rho\oplus y^q_\rho=\R^d$, furthermore we can interpret any other point $z^p_\rho$ as a linear map $z^p_\rho:x^p_\rho\to y^q_\rho$. With this notation the condition that the sum $x^p_\rho+(z^p_\rho\cap y^{q+1}_\rho)+ y^p_\rho$ is direct is equivalent to requiring that
$$z^p_\rho(x^p_\rho\cap y^{q+1}_\rho)\cap y^p_\rho=\{0\}.$$
Let us then choose a basis $\{b_1\ldots b_p\}$ of $x^p_\rho$ such that $\{b_1,\ldots b_{p-1}\}$ forms a basis of $x^{p-1}_\rho$ and $b_p=x^p_\rho\cap y^{q+1}_\rho$, then we have that a basis of $z^p$ is given by $c_i=b_i+z^p_\rho(b_i)$. Furthermore the only term of the explicit expression of $c_1\wedge\ldots \wedge c_p$ that might not belong to $y^{N-q+p-1}_{\wedge\rho}$ is $b_1\wedge\ldots\wedge b_{p-1}\wedge c_p$. This last vector doesn't belong to $y^{N-q+p-1}_{\wedge\rho}+x^1_{\wedge\rho}$ if and only if $z^p_\rho(x^p_\rho\cap y^{q+1}_\rho)\cap y^p_\rho=\{0\}$.
\end{proof}

\begin{prop}\label{p.Opq2}
Assume that there are convex cocompact representations $\rho_1:\G\to\SO(1,k)$, $\rho_2:\G\to\SO(1,l)$ such that $\rho_1$ strictly dominates $\rho_2$, namely there exists constants $c,\mu$ such that $\sigma_1(\rho_1(\g))>c \sigma_1(\rho_2(\g))^\mu$. Then the representation $\rho:=\rho_1\oplus\ldots\oplus\rho_1\oplus\rho_2:\G\to\SO\big(p,(p-1)k+l+s\big)$ satisfies the hypothesis of Proposition \ref{p.Opq1}.
\end{prop}
\begin{proof}
The representation $\rho$ is $\sroot_p$-Anosov as $\rho_2$ is convex cocompact, and is $\sroot_{p-1}$-Anosov as $\rho_1$ strictly dominates $\rho_2$. Explicitly writing down the boundary map $\xi^p$ associated to $\rho$ in term of the boundary maps $\xi^1_1:\G\to\bord\HR^k$, $\xi^1_2:\G\to\bord\HR^l$ associated to $\rho_1,\rho_2$ one verifies that $z^p_\rho(x^p_\rho\cap y^{q+1}_\rho)\cap y^p_\rho\cong z^1_{\rho_2}\cap y^1_{\rho_2}$ and the latter intersection is empty as the representation is Anosov.
\end{proof}
Danciger-Gueritaud-Kassel \cite[Proposition 1.8]{DGK18} gave an explicit construction of convex cocompact actions $\rho_1,\rho_2$ on $\mathbb H^8_\R$ of the the  group $\G$ generated  by  reflections  in  the  faces  of  a
4-dimensional regular right-angled  120-cell, such that $\rho_1$ strictly dominates $\rho_2$ and therefore Proposition \ref{p.Opq2} applies. In this case the boundary $\bord\G$ is a 3 sphere. 
It is also easy to construct representations satisfying the assumption of Proposition \ref{p.Opq2} when the group $\G$ is free, and in this case it one can  deform the representation $\rho:\F_n\to\SO(p,q)$ to obtain a Zariski dense representation whose image under $\wedge^p$ is locally conformal. 
We also expect that many more convex cocompact subgroups in rank one have the same property, and it is probably possible to give further examples of situations in which Proposition \ref{p.Opq1} applies for more complicated groups, as, for example, hyperbolic Coxeter groups.  

The same argument as in the proof of Proposition \ref{p.Opq1} gives the following
\begin{prop}\label{r.propH}
Let $\rho:\G\to\SL_d(\K)$ be $\{\sroot_{p-1},\sroot_p,\sroot_s\}$-Anosov. Assume that
\begin{enumerate}
\item  there exist constants $c,\mu$ such that
$$\frac{\sigma_{p-1}(\rho(\gamma))\sigma_s(\rho(\gamma))}{\sigma_p(\rho(\gamma))\sigma_{p+1}(\rho(\gamma))}>ce^{\mu|\gamma|},$$
\item for every $x,y,z\in\partial\G$ pairwise distinct, the sum
$$x^p_\rho+(z^p_\rho\cap y^{d-p+1}_\rho)+ y^{d-s}_\rho$$
is direct,
\end{enumerate}
then $\wedge^p\rho$ is $(1,1,s-p+1)$-hyperconvex. 
\end{prop}
Observe that the first condition, which guarantees that the map 
$$\begin{array}{cccc}
L:&\cal F_{p-1,p,s}(\K^{d})&\to&\cal F_{1,s-p+1}(\wedge^p\K^{d})\\
&(P,Q,R)&\mapsto&(\wedge^p(Q),\wedge^{p-1}(P)\wedge R)
\end{array}
$$
is proximal, is automatic if $s=p+1$.

\section{Fundamental groups of surfaces}\label{s.surfaces}

Let us denote by $\G_S$ a word-hyperbolic group such that\footnote{A celebrated Theorem of  Gabai \cite{Gabai} states that a hyperbolic group $\G_S$ such that $\bord\G_S$ is a circle is virtually the fundamental group of a connected, closed genus $\geq2$ surface. We will not use this fact.} $\bord\G_S$ is homeomorphic to $\mathbb S^1.$ One has the following direct consequence of Proposition \ref{diff} and Corollary \ref{2FrenetHffDim}.

\begin{cor}\label{h=1hyper} Let $\rho:\G_S\to\PGL_d(\R)$ be $(1,1,2)$-hyperconvex, then $h^{\sroot_1}_\rho=1.$
\end{cor}

\subsection{Weak irreducibility and closedness}
A projective Anosov representation $\rho:\G\to\PGL_d(\K)$ is \emph{weakly irreducible} if the image of its boundary map is not contained in a proper subspace of $\P(\K^d)$. Clearly if $\rho$ is irreducible, then $\rho$ is weakly irreducible, but it is possible to construct examples of weakly irreducible Anosov representations with non reductive image.

The assumption of weak irreducibility can be used to study properties of the stereographic projection  $\stp_{z,\rho}$ defined in Definition \ref{stereo}.

\begin{lemma}\label{stereographic2}
Let $\rho:\G\to\PGL_d(\K)$ be $\{\sroot_1,\sroot_p\}$-Anosov.
If the stereographic projection  $\stp_{z,\rho}:\bord\G\setminus \{z\}\to\P(\K^d)$ collapses an open set $U\subset \bord\G$, then $\stp_{z,\rho}$ is constant. In particular the representation $\rho$ is not weakly irreducible. 
\end{lemma}

\begin{proof}
Indeed, as fixed points of attractive elements are dense in $\bord\G$ we can find $\g\in\G$ with $\g^+\in U$. Up to shrinking $U$ we can assume that $\gamma\cdot U\subset U$. Let $V\subset \K^d$ be the smallest subspace containing $\xi(t)$ for every $t$ in $U$. As  $\stp_{z,\rho}|_U$ is constant, the subspace $V$ is proper, furthermore $\rho(\gamma)V=V$, since if $\xi(x_1),\ldots,\xi(x_k)$ is a basis of $V$ then   $\xi(\gamma x_1),\ldots,\xi(\gamma x_k)$ are also linearly independent vectors contained in $V$. In particular, for every $n$, $\stp_{z,\rho}(\gamma^{-n}U)$ is constant. As the union of the sets of the form $\gamma^{-n}U$ is the complement of a point in $\bord\G$, the first result follows by continuity of  $\stp_{z,\rho}$.

If the map $\stp_{z,\rho}$ is constant then, for every $x\in\bord\G- \{z\}$, the image of the boundary map is contained in the proper subspace $x^1_\rho+z_\rho^{d-r}$, hence the representation is not weakly irreducible.
\end{proof}

Lemma \ref{stereographic2} is particularly useful to analyze properties of $(1,1,2)$-representations of groups $\G_S.$
The following argument is very similar to Labourie \cite[Proposition 8.3]{Labourie-Anosov}. 

\begin{prop}\label{FrenetClosed} 
The space of real weakly irreducible $\{\sroot_1,\sroot_2\}$-Anosov representations of $\G_S$ that are not $(1,1,2)$-hyperconvex is open.
\end{prop}
\begin{proof}
Let $\rho:\G_S\to\PGL_d(\R)$ be $\{\sroot_1,\sroot_2\}$-Anosov and not $(1,1,2)$-hyperconvex. By definition, there exists a triple of pairwise distinct points $x,y,z\in\bord\G_S$ such that \begin{equation}
\label{not-Frenet}(x^1_\rho\oplus y^1_\rho)\cap z^{d-2}_\rho\neq0,
\end{equation} 
and thus the stereographic projection $\stp_{z,\rho}$ is not injective.

Note that $\P(\R^d/z_\rho^{d-2})$ is topologically a circle. Therefore the stereographic projection  $\stp_{z,\rho}$ is a map from an interval with a point removed to a circle that:
\begin{itemize}
\item[-] does not collapse intervals, 
\item[-] is not injective.
\end{itemize} 

One can therefore, using the intermediate value theorem, find  an interval $I\subset\bord\G_S\setminus\{z\}$ and a point $w\in\bord\G_S\setminus(\{z\}\cup I)$ such that  $\stp_{z,\rho}(w)$ belongs to the interior of $\stp_{z,\rho}(I).$ 

This last property will hold for any map close enough to $\stp_{z,\rho},$ in particular for the stereographic projection $\stp_{z,\eta}$ for some $\eta$ close to $\rho.$ Thus, $\stp_{z,\eta}$ is not injective and hence $\eta$ is not $(1,1,2)$-hyperconvex, as desired.

\end{proof}

Recall from Definition \ref{coherentDef} that a reducible representation $\pi:\SL_2(\K)\to\SL_d(\K)$ is $k$-coherent if it has a gap of index $k$ and its highest $k$ weights belong to the same irreducible factor.
Combining results from previous sections one has the following.

\begin{cor}\label{clopen} Let $\pi:\SL_2(\R)\to\SL_d(\R)$ be a $2$-coherent representation and $\rho:\G_S\to\PSL_2(\R)$ be co-compact, then any deformation $\eta$ of $\pi\rho$ among weakly irreducible $\{\sroot_1,\sroot_2\}$-Anosov representations into $\PSL_d(\R)$ is $(1,1,2)$-hyperconvex. In particular, verifies: 
\begin{itemize}
\item[-] has $\class^1$-limit set in $\P(\R^d),$
\item[-] the exponential growth rate $h^{\sroot_1}_\eta=1.$
\end{itemize}
\end{cor}

\begin{proof} Proposition \ref{2-coherent} states that $\pi\rho$ is $(1,1,2)$-hyperconvex. Proposition \ref{FrenetOpen} states hyperconvexity is an open property and, since $\bord\G_S$ is topologically a circle, Proposition \ref{FrenetClosed} implies that $(1,1,2)$-hyperconvex is closed among weakly irreducible $\{\sroot_1,\sroot_2\}$-Anosov representations. The remaining statements follow from Proposition \ref{diff} and Corollary \ref{2FrenetHffDim} for $\K=\R.$
\end{proof}
This result can be useful to distinguish some components of weakly irreducible Anosov representations (similar bounds on the number of connected components of Anosov representations were obtained with different techniques by Stecker-Treib \cite[Corollary 8.2]{SteTre}).

\subsection{The Hitchin component of $\PSL_d(\R)$}\label{HitchinPSL}

Let $S$ be a closed connected oriented surface of genus $\geq2$. The \emph{Hitchin component} of $\PSL_d(\R)$ is a connected component of the character variety $\frak X(\pi_1S,\PSL_d(\R))$ that contains a \emph{Fuchsian representation}, i.e. a representation that factors as $$\pi_1S\to\PSL_2(\R)\overset{\iota_d}{\longrightarrow} \PSL_d(\R),$$ where the first arrow is a the choice of a hyperbolic metric on $S.$ Such a connected component is usually denoted by $\hitchin_d(S)$ and an element $\rho\in\hitchin_d(S)$ is called a \emph{Hitchin representation}.

Recall from Labourie \cite{Labourie-Anosov} that a map $\xi:\bord\G\to\cal F(\R^d)$ satisfies Property (H) if for every triple of distinct points $x,y,z$ and every integer $k$ one has
$$\xi^{k+1}(y)+(\xi^{k+1}(z)\cap \xi^{n-k}(x))+\xi^{n-k-2}(x)=\R^d.$$ 
One has the following central result by Labourie \cite{Labourie-Anosov}.
\begin{thm}[Labourie \cite{Labourie-Anosov}]\label{HitchinHyperconvex} 
Every Hitchin representation $\rho:\G\to\PSL_d(\R)$ is $(p,q,r)$-hyperconvex, for every triple with $p+q=r.$ The equivariant boundary map $\xi:\bord\G\to\cal F(\R^d)$ has property (H).
\end{thm}
Thus, one concludes the following for deformations of the exterior powers.
\begin{prop}\label{apendix}
Let $\rho\in\hitchin_d(S)$ and consider any $k\in\lb1,d-1\rb.$ Then any weakly irreducible $\{\sroot_1,\sroot_2\}$-Anosov representation $\eta:\pi_1S\to\PSL(\wedge^k\R^d)$ connected by weakly irreducible $\{\sroot_1,\sroot_2\}$-Anosov representations to $\wedge^k\rho$ is $(1,1,2)$-hyperconvex and consequently verifies: \begin{itemize}
\item[-] has $\class^1$-limit set in $\P(\wedge^k\R^	d),$
\item[-] the exponential growth rate $h^{\sroot_1}_\eta=1.$
\end{itemize}
\end{prop}

\begin{proof} 
Observe that for every $s$, the representation $\wedge^s\rho$ is $\{\sroot_1,\sroot_2\}$-Anosov, furthermore Proposition \ref{r.propH} ensures that  $\wedge^s\rho$ is $(1,1,2)$-hyperconvex as the boundary curve satisfies Property (H) (for $k=s-1$). The result then follows from Proposition \ref{diff},  Corollary \ref{2FrenetHffDim} and Corollary \ref{clopen}.
\end{proof}

When no deformation is applied one recovers the following result from Potrie-S. \cite[Theorem B]{exponentecritico}.

\begin{thm}[{Potrie-S. \cite{exponentecritico}}] For every $\rho\in\hitchin_d(S)$ and every $k\in\lb1,d-1\rb$ one has $h^{\sroot_k}_\rho=1.$ 
\end{thm}

\subsection{Hitchin representations in other groups}\label{HitchinSO} More generally, let $G_\R$ be a simple real-split Lie group. These have been classified, i.e. up to finite coverings $G_\R$ is a group in the following list: $ \PSL_d(\R),\PSp(2n,\R),\SO(n,n+1)\SO(n,n),$ or it is the split real forms of the exceptional groups $\sf{F}_{4,\split{}}, \sf G_{2,\split{}}$,$E_{6,\split{}}$,$ E_{7,\split{}}$ and $E_{8,\split{}}.$

The work of Kostant \cite{kostant} provides a subalgebra $\iota_{\frak g_\R}:\frak{sl}_2(\R)\to\frak g_\R,$ unique up to conjugation and called \emph{the principal $\frak{sl}_2$},  such that the centralizer of $\iota_{\frak g_\R}\big(\begin{smallmatrix} 0&1\\ 0&0\end{smallmatrix}\big)$ has minimal dimension. Denote by $\iota_{G_\R}:\PSL_2(\R)\to G_\R$ the induced morphism. For example, $\iota_{\PSL_d(\R)}=\iota_d$ is the (unique up to conjugation) irreducible representation of $\SL_2(\R)$ in $\R^d$ defined in Subsection \ref{SL2}.

Let $S$ be a closed connected genus $\geq2$ surface. The \emph{Hitchin component} of $G_\R$ is the connected component of the character variety $\frak X(\pi_1S,G_\R)$ that contains a \emph{Fuchsian representation}, i.e. a representation that factors as $$\pi_1S\to\PSL_2(\R)\overset{\iota_{G_\R}}{\longrightarrow} G_\R,$$ where the first arrow is a the choice of a hyperbolic metric on $S.$ We will denote this connected component  by $\hitchin(S,G_\R)$ and an element $\rho\in\hitchin(S,G_\R)$ is called a \emph{Hitchin representation}.

\begin{obs}[Canonical inclusions]\label{inclusions} By construction, one sees that the irreducible representation $\iota_d:\SL_2(\R)\to\SL_d(\R)$ factors, depending on the parity of $d$, as 
$$\begin{array}{l}
\SL_2(\R)\overset{\iota_{\Sp(2n,\R)}}{\longrightarrow}\Sp(2n,\R)\to\SL_{2n}(\R),\\
\SL_2(\R)\overset{\iota_{\SO(n,n+1)}}{\longrightarrow}\SO(n,n+1)\to\SL_{2n+1}(\R),\\
\SL_2(\R)\overset{\sf{G}_{2,\R}}{\longrightarrow}\sf{G}_{2,\R}\to\SO(3,4)\to\SL_7(\R),
\end{array}
$$
where, in each case, the first arrows is the principal inclusion $\iota_{G_\R} $.
Thus 
$$\begin{array}{l}
\hitchin\big(S,\PSp(2n,\R)\big)\subset\hitchin_{2n}(S),\\
\hitchin\big(S,\PSO(n,n+1)\big)\subset\hitchin_{2n+1}(S),\\
\hitchin(S,\sf \P G_{2,\R})\subset\hitchin(S,\PSO(3,4)\big)\subset\hitchin_7(S).
\end{array}$$
On the other hand if we consider the embedding $\SO(n-1,n)\subset\SO(n,n)$ as the stabilizer of a positive definite line. The morphism $\iota_{\SO(n,n)}$ is the composition of $\iota_{\SO(n-1,n)}$ with such inclusion. Hence the induced action of $\iota_{\SO(n,n)}$
on $\R^{n,n}$ decomposes in $\SL_2(\R)$-irreducible modules as $$\iota_{2n-1}\oplus\iota_1,$$
and in particular $\hitchin\big(S,\PSO(n,n)\big)$ is not a subset of a $\PSL_{2n}(\R)$-Hitchin component. 
\end{obs}

It is known to experts that every Hitchin representation is Anosov with
respect to the minimal parabolic of $G,$ see for example Fock-Goncharov \cite{FG}.

Recall that the simple roots of the group $\PSO(n,n)$ are given by
$\{\sroot_1,\ldots,\sroot_{n-1}, \sf{b}_n\}$ where, as above,
$\sroot_i(x)=x_i-x_{i+1}$ and $\sf{b}_n$ is defined by
$$\sf{b}_n(x)=x_{n-1}+x_n.$$ Thus every representation $\rho\in
\hitchin\big(S,\PSO(n,n)\big)$, when considered as a representation in
$\SL_{2n}(\R)$ under the canonical inclusion, is $\{\sroot_p\}$-Anosov
for every $p\leq n-1$. 

Furthermore it is easy to check that 
the $n$-th exterior power $$\wedge^n:\PSO(n,n)\to\PSL(\wedge^n\R^{2n})$$ splits as the direct sum of two irreducible
$\PSO(n,n)$-modules, which have respectively $\sroot_{n-1}$ and $\sf b_n$ as first root (see for example Danciger-Zhang \cite{DZ}).
In particular we obtain the following result, independently announced by
Labourie \cite{Labourieh1}.

\begin{thm}\label{Hitn,n}
For every $\rho\in\hitchin(S,\PSO(n,n))$ and every $p\leq n-2$ the
exterior power $\wedge^p\rho$  is $(1,1,2)$-hyperconvex, and the same
holds for each one of the two irreducible submodules of
$\wedge^{n}\rho$. Thus the associated limit curve of $\rho$ on the
$p$-Grassmannian  for $p\leq n-2,$ as well as each one of the two limit
curves in the $n$-Grassmannian, is $\class^{1}$ and one has
$h^{\sroot}_\rho=1$ for every simple root $\sroot$.
\end{thm}

\begin{proof}
Considering a diagonalizable element in $\SL_2(\R)$ as in the proof of
Lemma \ref{Principal2coherent} we obtain that
$\wedge^k(\iota_{2n-1}\oplus\iota_1)$ is 2-coherent for every
$k\in\lb1,n-2\rb$. Similarly a direct computation shows that the 5
highest weights  of $\wedge^n(\iota_{2n-1}\oplus\iota_1)$ are
\begin{itemize}
\item[-]$\chi_{\wedge^n(\iota_{2n-1}\oplus\iota_1)}=\chi_{\wedge^n(\iota_{2n-1}\oplus\iota_1)}^{(2)}=2n+\ldots+
2=n(n+1),$
\item[-]$\chi_{\wedge^n(\iota_{2n-1}\oplus\iota_1)}^{(3)}=\chi_{\wedge^n(\iota_{2n-1}\oplus\iota_1)}^{(4)}=2n+\ldots+
4=n(n+1)-2$
\item[-]$\chi_{\wedge^n(\iota_{2n-1}\oplus\iota_1)}^{(5)}=2n+\ldots+
6=n(n+1)-4$
\end{itemize}
  and each of the first four weights appears with multiplicity one in
each irreducible $\SO(n,n)$-submodules of $\wedge^n\R^{2n}$. We deduce
that the restriction of the representation
$\wedge^n(\iota_{2n-1}\oplus\iota_1)$ to each of the two submodules is
also $(1,1,2)$-hyperconvex.
The result is then a consequence of Corollary \ref{clopen} together with
the classification of Zariski closures due to Guichard \cite{OlivierZC}.
\end{proof}

\begin{obs}
Danciger-Zhang \cite{DZ} recently proved that when a representation
$\rho\in\hitchin(S,\PSO(n,n))$ is regarded as a representation in
$\PSL_{2n}(\R)$, it is, instead, never $\{\sroot_n\}$-Anosov and the
limit curve in the $n-1$-Grassmannian is never $\class^1$.
\end{obs}
%
%
%
%
%

\bibliography{new}
\bibliographystyle{plain}

\bigskip


\author{\vbox{\footnotesize\noindent 
	Beatrice Pozzetti\\
Ruprecht-Karls Universit\"at Heidelberg\\ Mathematisches Institut, Im
Neuenheimer Feld 205, 69120 Heidelberg, Germany\\
	\texttt{pozzetti@mathi.uni-heidelberg.de}
\bigskip}}

\author{\vbox{\footnotesize\noindent 
	Andr\'es Sambarino\\
	Sorbonne Universit\'e \\ IMJ-PRG (CNRS UMR 7586)\\ 
	4 place Jussieu 75005 Paris France\\
	\texttt{andres.sambarino@imj-prg.fr}
\bigskip}}

\author{\vbox{\footnotesize\noindent 
	Anna Wienhard\\
	Ruprecht-Karls Universit\"at Heidelberg\\ Mathematisches Institut, Im
Neuenheimer Feld 205, 69120 Heidelberg, Germany\\ HITS gGmbH, Heidelberg Institute for Theoretical Studies\\ Schloss-Wolfsbrunnenweg 35, 69118 Heidelberg, Germany\\
	\texttt{wienhard@uni-heidelberg.de}
\bigskip}}

\end{document}